\documentclass[a4paper,twoside,10pt]{article}
\usepackage[a4paper,left=3cm,right=3cm, top=3cm, bottom=3cm]{geometry}
\usepackage{soul}
\usepackage{mathrsfs}
\usepackage{amsmath}
\usepackage{amsthm}
\usepackage{amssymb}
\usepackage{cite}
\usepackage{cancel}
\usepackage{appendix}
\usepackage{tikz}
\usetikzlibrary{calc}
\usepackage{soul}
\usepackage{dsfont}
\usepackage{enumerate}   
\usepackage{comment}
\usepackage{booktabs}
\usepackage{siunitx}
\usepackage{multirow}
\usepackage[dvipsnames]{xcolor}
\usepackage[normalem]{ulem}

\usepackage{hyperref}

\usepackage{todonotes}

\theoremstyle{plain}
\newtheorem{theorem}{Theorem}[section]
\newtheorem{corollary}[theorem]{Corollary}
\newtheorem{lemma}[theorem]{Lemma}

\theoremstyle{remark}
\newtheorem{remark}[theorem]{Remark}

\theoremstyle{definition}

\numberwithin{equation}{section}
\numberwithin{figure}{section}

\newcommand{\blue}[1]{{\color{blue}{#1}}}
\newcommand{\cyan}[1]{{\color{cyan}{#1}}}

\newcommand{\orange}[1]{{\color{orange}{#1}}} 
\newcommand{\red}[1]{{\color{red}{#1}}} 
\newcommand{\eremk}{\hbox{}\hfill\rule{0.8ex}{0.8ex}}

\newcommand{\Norm}[1]{{\left\|{#1} \right\|}}
\newcommand{\Normthreebars}[3]{\left|\!\left|\!\left| (#1,#2,#3) \right|\!\right|\!\right|}
\newcommand{\SemiNorm}[1]{{\left|{#1} \right|}}
\newcommand{\jump}[1]{\left[\!\left[#1\right]\!\right]}

\newcommand{\Pbb}{\mathbb{P}}
\newcommand{\Pbbk}{\mathbb{P}_k}
\newcommand{\Pbbkmo}{\mathbb{P}_{k-1}}
\newcommand{\Pbbkmt}{\mathbb{P}_{k-2}}

\newcommand{\Ccalh}{\mathcal{C}_h}
\newcommand{\Tauh}{\mathcal{T}_h}

\newcommand{\Rbb}{\mathbb{R}}
\newcommand{\Sbb}{\mathbb{S}}

\newcommand{\Rbbd}{\mathbb{R}^d}

\newcommand{\Rbbdd}{\mathbb{R}^{d\times d}}
\newcommand{\K}{K}
\newcommand{\F}{F}
\newcommand{\Sigmah}{\underline{\boldsymbol{\Sigma}}_h}

\newcommand{\Sigmahz}{\Sigmah^0}
\newcommand{\Sigmahkmo}{\Sigmah^{k-1}}
\newcommand{\Vh}{{\bf{V}}_h}
\newcommand{\Vhk}{\Vh^k}
\newcommand{\Vhkz}{\Vbf_{h,0}^k}
\newcommand{\Qh}{Q_h}
\newcommand{\Qhkmo}{\Qh^{k-1}}

\newcommand{\taubfh}{\underline{\boldsymbol{\tau}}_h}
\newcommand{\vbfh}{\mathbf{v}_h}
\newcommand{\qh}{q_h}
\newcommand{\sigmabfh}{\underline{\boldsymbol{\sigma}}_h}
\newcommand{\ubfh}{\mathbf{u}_h}
\newcommand{\ph}{p_h}
\newcommand{\taubf}{\underline{\boldsymbol{\tau}}}

\newcommand{\vbf}{\mathbf{v}}

\newcommand{\sigmabf}{\underline{\boldsymbol{\sigma}}}
\newcommand{\ubf}{\mathbf{u}}
\newcommand{\p}{p}
\DeclareMathOperator{\dive}{div}
\DeclareMathOperator{\divebf}{\mathbf{div}}
\DeclareMathOperator{\dev}{\underline{\mathbf{dev}}}
\DeclareMathOperator{\rot}{rot}
\DeclareMathOperator{\curlbf}{\mathbf{curl}}

\DeclareMathOperator{\tr}{tr}
\newcommand{\nbf}{\mathbf{n}}

\newcommand{\normtang}{{nt}}

\newcommand{\sym}{\underline{\bf{sym}}}

\DeclareMathOperator{\norm}{\bf{n}}
\DeclareMathOperator{\tang}{\bf{t}}
\newcommand{\tangscalar}{t}
\newcommand{\Ltwo}{L^2}
\newcommand{\LtwoOmega}{\Ltwo(\Omega)}
\newcommand{\LtwozeroOmega}{\Ltwo_0(\Omega)}
\newcommand{\Ltwozeroomega}{\Ltwo_0(\omega)}
\newcommand{\LtwoOmegad}{\Ltwo(\Omega)^d}
\newcommand{\Fcalhc}{\mathcal{F}_{\Ccalh}}
\newcommand{\Fcalhd}{\mathcal{F}_{\mathcal{I}_h}}
\newcommand{\Fcalh}{\mathcal{F}_h}

\newcommand{\VcalK}{\mathcal{V}_\K}
\newcommand{\FcalK}{\mathcal{F}_\K}

\newcommand{\bbf}{\mathbf{b}}
\newcommand{\xbf}{\mathbf{x}}
\newcommand{\bbfK}{\mathbf{b}_\K}

\newcommand{\dx}{\, \mathrm{d}x}
\newcommand{\ds}{\, \mathrm{d}s}
\newcommand{\Hcurldivbf}{H(\curlbf\divebf,\Omega)}
\newcommand{\Hdiv}{H(\dive,\Omega)}
\newcommand{\Hzerodiv}{H_0(\dive,\Omega)}
\newcommand{\zerobf}{\mathbf{0}}
\newcommand{\fbf}{\mathbf{f}}
\newcommand{\bone}{b_1}
\newcommand{\btwo}{b_2}
\newcommand{\numo}{\nu^{-1}}
\newcommand{\nablabf}{\boldsymbol{\nabla}}

\newcommand{\epsbfun}{\underline{\boldsymbol{\varepsilon}}}
\newcommand{\nablabfun}{\underline{\boldsymbol{\nabla}}}
\newcommand{\CcalhK}{\Ccalh(\K)}

\newcommand{\qkmo}{q_{k-1}}
\newcommand{\qkmt}{q_{k-2}}

\newcommand{\stildekmt}{\widetilde{s}_{k-2}}
\newcommand{\stildekmti}{\stildekmt^i}

\newcommand{\qkmoF}{q^\F_{k-1}}
\newcommand{\qkF}{q^\F_{k}}
\newcommand{\psiEi}{\psi^{\Fi}}

\newcommand{\psiKio}{\psi^{\Ki}_1}
\newcommand{\psiKit}{\psi^{\Ki}_2}

\newcommand{\lambdai}{\lambda_i}
\newcommand{\lambdaKi}{\lambda^{\Ki}_{i}}

\newcommand{\rkmoi}{r_{k-1}^i}
\newcommand{\skmoi}{s_{k-1}^i}
\newcommand{\lkmo}{\ell_{k-1}}
\newcommand{\mkmo}{m_{k-1}}
\newcommand{\Ki}{{\K_i}}
\newcommand{\Fi}{{\F_{\Ki}}}
\newcommand{\Fj}{\F_{\Ki}^j}
\newcommand{\ncompo}{n_1}
\newcommand{\ncompt}{n_2}
\newcommand{\tcompo}{t_1}
\newcommand{\tcompt}{t_2}
\newcommand{\itt}{\mathtt{i}}
\newcommand{\jtt}{\mathtt{j}}
\newcommand{\normi}{\norm_i}
\newcommand{\tangi}{\tang_i}
\newcommand{\normj}{\norm_j}
\newcommand{\tangj}{\tang_j}
\newcommand{\normitt}{\norm_{\itt}}

\newcommand{\normjtt}{\norm_{\jtt}}
\newcommand{\tangjtt}{\tang_{\jtt}}

\newcommand{\tangone}{\tcompo}
\newcommand{\tangtwo}{\tcompt}

\newcommand{\Idbf}{\underline{\mathbf{I}}}
\newcommand{\Id}{I}
\newcommand{\nutilde}{\widetilde{\nu}}
\newcommand{\nbfK}{\mathbf{n}_\K}
\newcommand{\tbfK}{\mathbf{t}_\K}
\newcommand{\hK}{h_\K}
\newcommand{\hKi}{h_{\Ki}}
\newcommand{\nbfF}{\mathbf{n}_\F}
\newcommand{\tbfF}{\mathbf{t}_\F}
\newcommand{\hF}{h_\F}

\newcommand{\FcalhB}{\mathcal{F}_h^{\text{ext}}}

\newcommand{\vbfn}{\vbf_{n}}
\newcommand{\vbft}{\vbf_{t}}

\newcommand{\sigmabfnt}{\sigmabf_{n t}}
\newcommand{\taubfnn}{\taubf_{n n}}
\newcommand{\taubfnt}{\taubf_{n t}}

\newcommand{\NormLtwoK}[1]{\Norm{#1}_{\Ltwo(\K)}}

\newcommand{\PizF}{\Pi_0^{\F}}
\newcommand{\PizFi}{\Pi_0^{\Fi}}

\newcommand{\Ione}{I_1}
\newcommand{\Itwo}{I_2}
\newcommand{\Kone}{\K_1}
\newcommand{\Ktwo}{\K_2}
\newcommand{\Kthree}{\K_3}
\newcommand{\Ctilde}{\widetilde{C}}

\newcommand{\ceqone}{c_{\text{eq},1}}
\newcommand{\Ceqone}{C_{\text{eq},1}}
\newcommand{\Ceqtwo}{C_{\text{eq},2}}
\newcommand{\Ceqthree}{C_{\text{eq},3}}
\newcommand{\IKone}{I_{\K,1}}
\newcommand{\IKtwo}{I_{\K,2}}
\DeclareMathOperator{\conv}{conv}

\newcommand{\aKi}{\mathbf{a}^\K_i}
\newcommand{\aKij}{\mathbf{a}^{\Ki}_j}
\newcommand{\aKii}{\mathbf{a}^{\Ki}_i}
\newcommand{\aone}{\mathbf{a}_1}
\newcommand{\atwo}{\mathbf{a}_2}
\newcommand{\athree}{\mathbf{a}_3}
\newcommand{\ai}{\mathbf{a}_i}
\newcommand{\aj}{\mathbf{a}_j}
\newcommand{\ak}{\mathbf{a}_k}
\newcommand{\Zh}{Z_h}
\newcommand{\IcalQhkmo}{\mathcal{I}_{\Qh}}
\newcommand{\IcalVhk}{\mathcal{I}_{\Vh}}
\newcommand{\IcalSigmahkmo}{\mathcal{I}_{\Sigmah}}
\newcommand{\IcalKi}{\mathcal{I}_{\Ki}}
\newcommand{\Sigmabf}{\underline{\boldsymbol{\Sigma}}}
\newcommand{\Vbf}{\mathbf{V}}
\newcommand{\sIi}{s_{\mathcal{I}}^i}
\newcommand{\rIi}{r_{\mathcal{I}}^i}
\newcommand{\PikmoFi}{\Pi_{k-1}^{\Fi}}

\newcommand{\Lcalk}{\mathcal{L}_k}
\newcommand{\Lcalzero}{\mathcal{L}_0}
\newcommand{\Lcalkmo}{\mathcal{L}_{k-1}}
\newcommand{\stildeIi}{\tilde{s}^i_{\mathcal{I}}}
\newcommand{\lambdaj}{\lambda_j}

\newcommand{\qbfunderkmo}{\underline{\mathbf{q}}_{k-1}}
\newcommand{\Pikmo}{\Pi_{k-1}}
\newcommand{\skmoFi}{s_{k-1}^{\Fi}}
\newcommand{\HDG}{\text{hyb}}
\newcommand{\VHDG}{\Vbf_{\HDG}}
\newcommand{\SigmaHDG}{\Sigmabf_{\HDG}}
\newcommand{\UHDG}{\mathbb{U}_{\HDG}}

\newcommand{\VhatHDG}{\widehat{\Vbf}_{\HDG}}
\newcommand{\QhatHDG}{\widehat{Q}_{\HDG}}

\newcommand{\vhhat}{\hat{\vbf}_h}
\newcommand{\uhhat}{\hat{\ubf}_h}

\DeclareMathOperator{\BDM}{BDM}
\newcommand{\Ekmosigmasym}{E_{\sigmabf,\rm{sym}}^{k-1}}
\newcommand{\Ekmosigma}{E_{\sigmabf}^{k-1}}

\newcommand{\Ekmop}{E_{p}^{k-1}}
\newcommand{\Eku}{E_{\ubf}^{k}}
\newcommand{\Ekepsu}{E_{\epsbfun(\ubf)}^{k}}
\newcommand{\abf}{\mathbf{a}}
\newcommand{\cbf}{\mathbf{c}}
\newcommand{\Pbbz}{\Pbb_0}
\newcommand{\Nel}{N_{el}}
\newcommand{\btwohyb}{b_{2,\text{hyb}}}
\newcommand{\Kj}{\K_j}
\newcommand{\pkmo}{p_{k-1}}
\newcommand{\piF}{p_{i}^F}
\newcommand{\qi}{q_{i}}

\newcommand{\wi}{w_i}

\title{A mass-conserving stress-yielding formulation for the Stokes equation
        with exact stress symmetry}
\author{Philip L. Lederer\thanks{Department of Mathematics, University of Hamburg, D-20146 Hamburg, Germany (philip.lederer@uni-hamburg.de)},
Marialetizia Mosconi \thanks{Department of Mathematics and Applications, University of Milano-Bicocca, 20125 Milan, Italy (m.mosconi@campus.unimib.it)}}
\date{}

\begin{document}
\maketitle

\begin{abstract}
\noindent
We introduce a new finite element discretization for a mixed formulation of
the two-dimensional incompressible Stokes equations
with symmetric viscous stresses.
The method is based on a mass-conserving stress-yielding formulation (Gopalakrishnan, Lederer, Schöberl; A mass conserving mixed stress formulation for the Stokes equations; IMA J. Numer. Anal., Vol 40, 2020) with symmetric stresses,
with the key novelty that the symmetry of the stress tensor is enforced exactly
as an intrinsic property of the discrete space.
This space is constructed on the Clough--Tocher macroelement split and
consists of matrix-valued functions whose
normal-tangential components across macroelement interfaces are continuous.
Since the exact symmetry is enforced via local polynomial bubble functions
on the subelements,
which can be eliminated via static condensation,
the number of globally coupled degrees of freedom coincide
with those of the original mass-conserving stress-yielding formulation.
The resulting pressure-robust method is stable without requiring any additional enrichment.
We establish optimal convergence rates for all the variables
and present numerical experiments that confirm the theoretical results.



\end{abstract}

\section{Introduction} \label{section:introduction}
We introduce a novel discretization of the mixed stress formulation of
the steady incompressible Stokes equations in two space dimensions,
where the symmetric viscous stresses are treated as an independent variable.
Let~$\Omega$ be an open, bounded domain in~$\Rbbd$, $d=2$,
with Lipschitz boundary~$\partial\Omega$, denoted by~$\Gamma$.
Throughout,
scalar-, vector-, and tensor-valued operators and functions
are denoted by standard, boldface, and underlined boldface symbols, respectively.
Given an external force~$\fbf: \Omega \to \Rbbd$
and kinematic viscosity~$\nutilde: \Omega \to \Rbb$,
consider the standard \emph{primal} velocity-pressure formulation of the Stokes equation:
find a velocity vector field~$\ubf: \Omega \to \Rbbd$
and a pressure scalar field~$p: \Omega \to \Rbb$
such that
\begin{subequations}\label{strong-standard}
\begin{alignat}{2}
-\divebf(2\nutilde\epsbfun(\ubf)) + \nablabf p &= \fbf      \quad && \text{in } \Omega, \\
\dive(\ubf)                                    &= 0         \quad &&\text{in } \Omega, \\
\ubf                                           &= \zerobf   \quad &&\text{on } \Gamma,
\end{alignat}
\end{subequations}
where $\epsbfun(\ubf) := 1/2(\nablabfun\ubf + (\nablabfun\ubf)^T)$ is the symmetric gradient of~$\ubf$.
Given a tensor~$\sigmabf$,
we consider its (matrix) trace~$\tr(\sigmabf) = \sum_{i=1}^d (\sigmabf)_{ii}$
and its deviatoric part~$\dev(\sigmabf) := \sigmabf - 1/d \tr(\sigmabf) \Idbf$, where $\Idbf$ is the identity matrix.
By introducing the viscous stress tensor~$\sigmabf := \nu\epsbfun(\ubf)$,
where~$\nu := 2 \nutilde$,
problem~\eqref{strong-standard} can be reformulated
as a \emph{mixed} formulation,
which involves the (pseudo) stresses as an additional unknown:
find
a stress tensor~$\sigmabf: \Omega \to \Rbbdd$,
a velocity vector field~$\ubf: \Omega \to \Rbbd$,
and a pressure scalar field~$p: \Omega \to \Rbb$
such that
\begin{subequations} \label{strong-mixed}
\begin{alignat}{2}
\nu^{-1}\dev(\sigmabf) - \epsbfun(\ubf)        &= \zerobf  \quad && \text{in }\Omega, \\
\dive\sigmabf - \nablabf p                     &= -\fbf    \quad && \text{in }\Omega, \\
\dive\ubf                                      &= 0        \quad && \text{in }\Omega, \\
\ubf                                           &= \zerobf  \quad && \text{on }\Gamma.
\end{alignat} 
\end{subequations}
Although the primal and mixed formulations are formally equivalent,
the latter demands less regularity on the velocity field~$\ubf$.
This formulation has been the subject of extensive study over the past decades; see, e.g.,
\cite{Farhloul:1995,Farhloul-Fortin:1993,Farhloul-Fortin:1997},
and the comprehensive literature on the design of stable discretizations
of the Stokes equation in both of the above formulations
is too vast to list here.

Our work builds upon
the \emph{Mass-Conserving Stress-yielding formulation} (MCS) framework
introduced in~\cite{Gopalakrishnan-Lederer-Schoberl:2020}
for the counterpart of the mixed formulation~\eqref{strong-mixed}
in which $\sigmabf$ is defined as $\nu\nablabfun\ubf$
(hence the stress is not necessarily symmetric)
and $\epsbfun(\ubf)$ in the first equation of~\eqref{strong-mixed}
is replaced by the gradient~$\nablabfun\ubf$.
The key feature of this formulation is that
the tensor-valued stress variable belongs to the space~$\Hcurldivbf$
of functions whose divergence acts continuously on~$\Hzerodiv$.
In the corresponding MCS method,
the stress variable is discretized using a family of finite elements
for which the continuity of the normal-tangential component of the tensors across element interfaces is imposed.
Given an integer~$k\geq1$,
the proposed method employs interior polynomial enrichment of order~$k$ to ensure stability,
whereas the traces on the edges are polynomials of order~$k-1$.
This discretization yields velocity approximations
that are exactly divergence-free,
thereby ensuring the exact mass conservation
together with the desirable (and non-trivial \cite{Linke:2014})
structure-preserving \emph{pressure-robustness} property.
Different approaches have been considered to achieve pressure-robustness.
Some possibilities are to modify the load in order to obtain
pressure-independent velocity error estimates;
see, e.g., \cite{Lederer-Linke-Merdon:2017,Linke:2012,Linke-Matthies-Tobiska:2016,Brennecke-Linke-Merdon-Schoberl:2015},
or to consider $H(\dive)$-conforming discretizations of the velocity field,
which are specifically designed to approximate the incompressibility constraint;
see, e.g., \cite{Cockburn-Kanschat-Schotzau:2005,Cockburn-Kanschat-Schotzau:2007}.
The latter approach
also relaxes the regularity requirement for the discrete velocity space
from~$H^1_0(\Omega,\Rbbd)$,
which is instead required in~\eqref{strong-standard}.

The MCS formulation was subsequently extended in~\cite{Gopalakrishnan-Lederer-Schoberl:2020-b}
to the physically relevant case of symmetric viscous stresses.
There,
symmetry is imposed weakly through an element-wise discontinuous Lagrange multiplier,
approximating the vorticity (curl of the velocity) variable,
and stability is established by enriching the stress space with suitable tensor-valued bubbles,
using techniques inspired by the classical work of~\cite{Stenberg:1988}.

More recently,
a \emph{hybrid} MCS formulation was proposed in~\cite{Gopalakrishnan-Kogler-Lederer-Schoberl:2023},
cast in the framework of hybrid discontinuous Galerkin (HDG) methods.
Such methods have been widely used in the literature
since their introduction in~\cite{Cockburn-Gopalakrishnan-Lazarov:2009},
where they were proposed as an efficient way of constructing
DG methods equipped with suitable interface variables,
with the aim of enabling static condensation.
In the same spirit,
the hybrid MCS formulation of~\cite{Gopalakrishnan-Kogler-Lederer-Schoberl:2023}
pursues a method with a minimal number of degrees of freedom per facet.
In addition,
the bubble enrichment of~\cite{Gopalakrishnan-Lederer-Schoberl:2020-b} is removed,
an interface variable is used to approximate the tangential component of the velocity,
and weak symmetry is imposed by means of a Lagrange multiplier in an $H(\dive)$-conforming space.

An alternative way to impose the symmetry of the stress tensor is exactly, i.e., pointwise,
rather than weakly through for instance a Lagrange multiplier.
In the recent work~\cite{Brubeck-Parker-Zerbinati:2026},
the importance of exactly versus weakly enforcing the symmetry constraint is discussed for incompressible flows and for linear elasticity.
Therein,
a series of benchmark problems is considered for both
weakly and exactly imposed symmetry,
showing that schemes with weak symmetry
can generate arbitrarily poor stress approximations,
even for exactly zero-stress configurations.
In contrast,
exact symmetry schemes are generally not affected by such a lack of accuracy,
independently of the chosen constitutive law.

In light of these developments,
it is natural to ask whether a stable MCS method can be constructed
for the mixed formulation~\eqref{strong-mixed},
where the symmetry of the stress tensor is imposed exactly
while avoiding higher-order interior bubble enrichment.

The exact enforcement of the symmetry constraint on the stress tensor is a challenging task
in the design of stable finite element methods,
both for incompressible flows and for linear elasticity,
and a variety of approaches have been proposed to address this challenge.
In the present work,
we will employ a \emph{macroelement} technique,
which has a long history as a crucial tool in the design
of stable finite element discretizations for the Stokes problem.
In particular,
we will consider the Clough--Tocher refinement
(also known in the literature as the HCT refinement);
see, e.g., \cite{Clough-Tocher:1965,Ciarlet:1991},
where each triangle of the original mesh is split into three subtriangles
by connecting the barycenter to the vertices.
This composite split has been widely used for the Stokes problem;
see, e.g., \cite{Qin:1994,Zhang:2005,Arnold-Qin:1992},
as well as for the analysis of the Hellinger--Reissner linear elasticity problem;
see, e.g., \cite{Watwood-Hartz:1968,Johnson-Mercier:1979}.
Motivated by these developments,
we build on the Clough--Tocher macroelement structure
to construct exactly symmetric stress spaces within the MCS framework.

\paragraph*{Goals of the paper.}
The goal of this work is to develop
a stable and exactly symmetric discretization of the MCS formulation in two space dimensions,
where the Clough--Tocher macroelement split is exploited
to design a new family of finite element spaces for the stress variable~$\sigmabf$.
In this construction,
the symmetry of the stress tensor is exactly imposed as a built-in property,
rather than being imposed weakly.
More precisely,
given an integer~$k \geq 1$,
the discrete stress space consists of \textbf{exactly symmetric} and trace-free
tensor-valued functions,
whose components are
piecewise polynomial of order~$k-1$
on the Clough--Tocher refinement.
In particular,
no higher-order interior bubble enrichment is required,
unlike in the original MCS construction.
For the velocity and pressure variables,
we will consider the same spaces as those in~\cite{Gopalakrishnan-Lederer-Schoberl:2020},
defined on the original (unrefined) mesh:
as discrete velocity space,
vector-valued piecewise polynomial functions of order~$k$
with square-integrable divergence;
as discrete pressure space,
discontinuous piecewise polynomials of order~$k-1$.

With these choices,
the method preserves the main structural properties of the MCS framework,
namely exact mass conservation,
pointwise divergence-free discrete velocities,
and pressure robustness.

Moreover,
after a static condensation step,
the resulting globally coupled unknowns coincide with those of the original MCS method,
hence no additional coupling is introduced via the macroelement strategy.
We prove stability of the method,
derive optimal a priori error estimates for all discrete variables,
and assess its performance through numerical experiments.

\paragraph*{Outline of the paper} 
The remainder of the paper is organized as follows.
In Section~\ref{section:preliminaries},
we introduce the continuous formulation and the functional setting.
Section~\ref{section:MCS-symmetric} is devoted to the presentation of the symmetric MCS method
and to the discussion of the structure and the properties of the stress discrete space,
while the a priori error analysis is presented in Section~\ref{section:a-priori}.
Numerical experiments illustrating the theoretical results are presented in Section~\ref{section:numerics}.

\section{Preliminaries} \label{section:preliminaries}
We recall the functional setting, introduce the necessary notation,
and consider the model problem in Section~\ref{subsection:functional-setting}.
Details on the employed meshes and broken polynomial, Sobolev, and finite element spaces
are given in Section~\ref{subsection:meshes}.

\subsection{Functional setting and model problem} \label{subsection:functional-setting}
We denote the unit outward normal vector to~$\Gamma$
by~$\nbf_\Omega$
and the diameter of~$\Omega$ by~$h_\Omega$.
Consider a vector-valued function~$\vbf = (v_i)_{i=1}^d$
and a tensor-valued function~$\taubf = (\tau_{ij})_{i,j=1,\dots,d} $.
We denote by~$\tr(\taubf)$ its trace and by~$\dev\taubf$ its deviatoric part, i.e.,
\begin{equation*}
\tr(\taubf) := \sum_{i=1}^d \tau_{ii},
\qquad\text{ and }\qquad
\dev\taubf := \taubf - \frac1d \tr(\taubf) \Idbf,
\end{equation*}
where~$\Idbf$ is the identity matrix in~$\Rbbdd$.
Denoting by~$(\cdot)^T$ the transpose operator,
we further consider the space of $d \times d$ symmetric matrices
$\Sbb := \{\taubf \in \Rbbdd \, | \, \taubf = \taubf^T\}$.

For a generic subset~$\omega$ of~$\Omega$,
we denote the space of square integrable real-valued functions over~$\omega$ by $\Ltwo(\omega) = \Ltwo(\omega,\Rbb)$;
the vector- and tensor-valued counterparts of $\Ltwo(\omega,\Rbb)$ are denoted by
\begin{equation*}
\Ltwo(\omega,\Rbbd) := \{ \vbf : \omega \to \Rbbd \, | \, v_i \in \Ltwo(\omega) \},
\quad\text{and}\quad
\Ltwo(\omega,\Rbbdd) := \{ \taubf : \omega \to \Rbbdd \, | \, \tau_{ij} \in \Ltwo(\omega) \}.
\end{equation*}
We endow the above spaces with the standard inner product~$(\cdot,\cdot)_\omega$
and norm~$\Norm{\cdot}^2_{L^2(\omega)} := (\cdot,\cdot)_\omega$.
Such notations are extended in an obvious fashion to other function spaces as needed.
The subspace of functions in~$\Ltwo(\omega)$ with zero average over~$\omega$ is~$\Ltwozeroomega$.
We denote by~$\Pbbk(\omega)$
the space of polynomials of order smaller than or equal to~$k$;
if $k<0$,
we set
$\Pbbk(\omega):=\emptyset$.
Given a topological space~$X$,
we denote by~$X^*$ the dual space of~$X$ and by~$\langle \cdot, \cdot \rangle_{X}$
the duality pairing between~$X^*$ and~$X$.
We may omit the subscript~$X$ when it is clear from the context.

Abbreviating~$\partial/\partial x_i$ by~$\partial_i$,
we consider the following differential operators:
given a scalar function~$q$
and a vector-valued function~$\vbf = (v_1, v_2)$,
we define
\begin{gather*}
\nablabf q := (\partial_1 q, \partial_2 q)^T, \quad
\curlbf q := (\partial_2 q, -\partial_1 q)^T, \quad
\dive \vbf := \partial_1 v_1 + \partial_2 v_2, \quad
\rot \vbf := -\partial_2 v_1 + \partial_1 v_2.
\end{gather*}
For a vector-valued function~$\vbf$
and a tensor-valued function~$\taubf$,
the tensor- and the vector-valued
operators~$\nablabfun \vbf$ and~$\divebf \taubf$
are defined row-wise.
The symmetric gradient operator~$\epsbfun$ is given by
\begin{equation*}
\epsbfun(\vbf) := \frac12 (\nablabfun \vbf + (\nablabfun \vbf)^T).
\end{equation*}
Given a tensor-valued function~$\taubf$,
we denote the squared Frobenius norm by
$|\taubf|^2 := (\taubf : \taubf)$.

On each boundary of elements
or on a facet in~$\FcalhB$
with corresponding outer unit normal vector~$\nbf$,
we define the scalar normal and tangential components of a smooth vector-valued function~$\vbf: \Omega \to \Rbbd$ by
\begin{equation*}
\vbfn := \vbf \cdot \norm,
\qquad\text{and}\qquad
\vbft := \vbf  \cdot \tang;
\end{equation*}
the normal-normal and normal-tangential components of
a smooth tensor-valued function~$\taubf: \Omega \to \Rbbdd$,
\begin{equation*}
\taubfnn := \taubf : (\norm \otimes \norm) = \norm^T \taubf \norm,
\qquad\text{and}\qquad
\taubfnt := \taubf : {(\tang \otimes \norm)} = \tang^T \taubf \norm.
\end{equation*}

We employ standard notation for the Sobolev spaces~$H^s(\Omega)$, $s$ in $\Rbb$,
and write~$H^{-1}(\Omega)$ for the dual space of
$H^1_0(\Omega):= \{q \in H^1(\Omega) \, | \, q|_{\Gamma} = 0\}$.
We consider the space of square integrable vector fields with square integrable divergence
and its subspace of vector fields with vanishing normal trace on the boundary:
\begin{equation*}
\Hdiv := \{ \vbf \in \LtwoOmegad \, | \, \dive\vbf \in \LtwoOmega \},
\quad 
\text{and}
\quad
\Hzerodiv := \{ \vbf \in \Hdiv \, | \, \vbf \cdot \nbf|_{\Gamma} = 0 \},
\end{equation*}
respectively, which is well-defined due to the trace theorem for~$\Hdiv$;
see, e.g., \cite{Boffi-Brezzi-Fortin:2013}.
For this space,
we have the following topological identification~\cite[Thm.~2.1]{Gopalakrishnan-Lederer-Schoberl:2020}:
\begin{equation} \label{topological-identification}
\Hzerodiv^* =
H^{-1}(\curlbf,\Omega)
:= \{ \vbf \in H^{-1}(\Omega,\Rbb^2) \, | \, \rot \vbf \in H^{-1}(\Omega,\Rbb) \}.
\end{equation}
Following \cite{Gopalakrishnan-Lederer-Schoberl:2020},
we further consider the space~$\Hcurldivbf$
\begin{equation} \label{Hcurldiv-2d}
\begin{split}
\Hcurldivbf
&:=
\{ \taubf \in \Ltwo(\Omega,\Rbbdd) \, | \, \divebf\taubf \in  H_0(\dive,\Omega)^*\} \\
&\overset{\eqref{topological-identification}}{=}
\{ \taubf \in \Ltwo(\Omega,\Rbbdd) \, | \, \rot \divebf\taubf \in  H^{-1}(\Omega,\Rbb)\}.
\end{split}
\end{equation}

\begin{remark}
The notation $\curlbf$ in the definition of~$\Hcurldivbf$
is motivated by the three-dimensional counterpart of the space.
Indeed,
given a vector-valued function~$\vbf = (v_1, v_2, v_3)$,
we introduce the differential operators
\begin{equation*}
\dive \vbf :=\partial_1 v_1 + \partial_2 v_2 + \partial_3 v_3, \qquad
\curlbf \vbf := (\partial_2 v_3 - \partial_3 v_2, \partial_3 v_1 - \partial_1 v_3, \partial_1 v_2 - \partial_2 v_1)^T,
\end{equation*}
and, for a tensor-valued function~$\taubf$,
we define~$\divebf \taubf$ row-wise.
The three-dimensional counterpart of the space in~\eqref{Hcurldiv-2d} is then
\begin{equation*}
\Hcurldivbf := \{ \taubf \in \Ltwo(\Omega,\Rbbdd) \, | \, \curlbf \divebf\taubf \in  H^{-1}(\Omega,\Rbb^3)\}.
\end{equation*}
This provides the motivation for the notation $\Hcurldivbf$,
and although the present work is restricted to the two-dimensional setting,
we retain this notation for consistency.
\eremk
\end{remark}

The standard variational formulation of the Stokes problem~\eqref{strong-standard}; see, e.g.,~\cite{Girault-Raviart:1986},
reads as follows:
find $(\ubf, p)$ in $H^1_0(\Omega,\Rbbd) \times \LtwozeroOmega$ such that
\begin{subequations} \label{variational-standard}
\begin{alignat}{2}
(\nu \epsbfun(\ubf),\epsbfun(\vbf)) -(\dive\vbf,p) &= (\fbf,\vbf) \quad &&\forall \vbf \in H^1_0(\Omega,\Rbbd),\ \\
- (\dive\ubf,q) &= 0 \quad &&\forall q \in \LtwozeroOmega.
\end{alignat}
\end{subequations}
However,
we are interested in a mixed formulation of the Stokes problem~\eqref{strong-mixed}
where the stress tensor is treated as an independent variable
and the approximation space for the velocity~$\ubf$ is $H(\dive)$-conforming.
We consider the \emph{mass-conserving stress-yielding formulation with symmetric stresses} of the Stokes problem~\eqref{strong-mixed}
introduced in~\cite{Gopalakrishnan-Lederer-Schoberl:2020-b}.
Given the spaces
\begin{align*}
\Sigmabf := \{ \taubf \in \Hcurldivbf \cap \Ltwo(\Omega,\Sbb) \, | \, \tr(\taubf) = 0\},
\qquad
\Vbf      := \Hzerodiv,
\qquad
Q      := \LtwozeroOmega,
\end{align*}
and using the data~$\fbf$ and~$\nu$ introduced in Section~\ref{section:introduction},
we consider the following variational formulation as model problem:
find $(\sigmabf, \ubf, p)$ in $\Sigmabf \times \Vbf \times Q$ such that
\begin{subequations} \label{variational-formulation}
\begin{alignat}{2}
\numo( \dev\sigmabf,\dev\taubf) + \langle\divebf\taubf,\ubf\rangle_{\Hzerodiv}   &= 0             &&\quad\forall \taubf \in \Sigmabf,\ \\
\langle\divebf\sigmabf,\vbf\rangle_{\Hzerodiv} + (\dive\vbf,p)                   &= -(\fbf,\vbf)  &&\quad\forall \vbf \in \Vbf,\        \\
(\dive\ubf,q)                                                                    &= 0             &&\quad\forall q \in Q.
\end{alignat}
\end{subequations}
Comparing~\eqref{variational-formulation}
to the formulation of~\cite{Gopalakrishnan-Lederer-Schoberl:2020},
the stress space in the latter was set to $\widetilde{\Sigmabf} := \{ \taubf \in \Hcurldivbf \, | \, \tr(\taubf) = 0\}$,
since no symmetry constraint was imposed on the stress variable~$\widetilde{\sigmabf}$
that is there defined as $\nu\nablabfun \ubf$,
instead of $\nu\epsbfun(\ubf)$.

\subsection{Meshes, and broken Sobolev and polynomial spaces} \label{subsection:meshes}
Let~$\Tauh$ be a conforming, shape-regular, and quasi-uniform simplicial partition of the domain~$\Omega$.
For a given element~$\K$ in~$\Tauh$,
we denote by~$\nbfK$, $\tbfK$, $\bbfK$, and~$\hK$
its unit outward normal and tangential vectors,
its barycenter,
and its diameter, respectively;
the maximum of the diameters of all elements in~$\Tauh$ is denoted by~$h := \max_{\K \in \Tauh}\hK$.
For a given edge~$\F$ in~$\Fcalh$,
we denote by~$\nbfF := (\ncompo,\ncompt)$
a fixed once-and-for-all unit normal vector and
we fix~$\tbfF := (\tangone, \tangtwo) = (\ncompt, -\ncompo)$
as the  unit tangential vector
given by the right rotated normal vector.
The length of~$\F$ is denoted by~$\hF$.
The sets of edges and vertices of a given element~$\K$ in $\Tauh$
are denoted by~$\FcalK$ and~$\VcalK$;
we denote by~$\aKi$, $i=1,2,3$ the vertices in~$\VcalK$.
When no confusion occurs,
we shall omit the superscript~$\K$.

In what follows,
we employ a \emph{macroelement} approach for the discrete stress space
based on the Clough--Tocher (HCT) refinement \cite{Clough-Tocher:1965}:
given an element~$\K\in\Tauh$ with vertices
$\{\aone,\atwo,\athree\}$ and barycenter~$\bbfK$,
we define the subelements $\Ki$, $i=1,2,3$,
as
\begin{equation*}
\Ki := \conv\{\bbfK, \aj, \ak\},
\qquad
\{i,j,k\} = \{1,2,3\},
\end{equation*}
where~$\conv$ denotes the convex hull of a set of points.
We denote the resulting partition of~$\K$
by~$\CcalhK := \{\Kone,\Ktwo,\Kthree\}$.
By construction,
each subelement shares exactly one edge with~$\K$,
which is the edge of~$\K$ opposite to the vertex~$\ai$.
We define this edge as the \emph{external edge} of~$\Ki$,
and we denote it by~$\Fi$.
In particular,
to each fixed subelement~$\Ki$,
we can associate its external edge~$\Fi$ in~$\FcalK$.
The external edges will play a crucial role
in both the construction of the local stress spaces
and the stability analysis.
On an element~$\Ki$ in~$\CcalhK$,
we denote by~$\lambdaKi$ the barycentric coordinate
corresponding to the vertex~$\bbfK = \aKii$, i.e.,
the linear Lagrange function such that
$\lambdaKi(\aKij) = \delta_{ij}$,
$j=1,2,3$.
Note that~$\aKii$ is the vertex of~$\Ki$ opposite to the external edge~$\Fi$;
hence,
$\lambdaKi$ vanishes on~$\Fi$.
When no confusion occurs,
we replace~$\lambdaKi$
by~$\lambdai$.
We refer to Figure~\ref{fig:one-element-subdivision}
for a visual representation of the aforementioned split
and the corresponding notation.

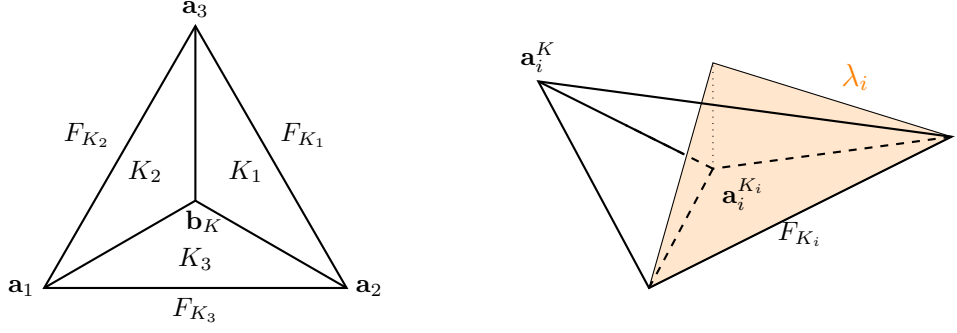
\begin{figure}[htpb]
\centering
\begin{tikzpicture}[scale=2]
\begin{scope}
\coordinate (A) at (-1,0);
\coordinate (B) at (1,0);
\coordinate (C) at ($(0,{sqrt(3)})$);
\coordinate (b) at ($1/3*(A) + 1/3*(B) + 1/3*(C)$);

\draw[thick,] (A) -- (B) -- (C) -- cycle;
\draw[thick,] (A) -- (b) -- (B);
\draw[thick,] (C) -- (b);

\node[below] at (b) {~~$\bbfK$};
\node[below] at ($1/2*(A) + 1/2*(B)$) {$\F_{\K_3}$};
\node[above left] at ($1/2*(A) + 1/2*(C)$) {$\F_{\K_2}$};
\node[above right] at ($1/2*(C) + 1/2*(B)$) {$\F_{\K_1}$};

\node at ($1/3*(A) + 1/3*(B) + 1/3*(b)$) {$\K_3$};
\node at ($1/3*(A) + 1/3*(C) + 1/3*(b)$) {$\K_2$};
\node at ($1/3*(C) + 1/3*(B) + 1/3*(b)$) {$\K_1$};

\node[left] at (A) {$\aone$};
\node[right] at (B) {$\atwo$};
\node[above] at (C) {$\athree$};
\end{scope}
\begin{scope}[
  x={(0.5cm,0.25cm)},
  y={(-0.5cm,0.25cm)},
  z={(0cm,0.7cm)},
  xshift=4cm,
  yshift=.5cm,
  scale=2]
\coordinate (A) at (-1,0);
\coordinate (B) at (1,0);
\coordinate (C) at ($(0,{sqrt(3)})$);
\coordinate (b) at ($1/3*(A) + 1/3*(B) + 1/3*(C)$);

\coordinate (A0) at ($(A)+(0,0,0)$);
\coordinate (B0) at ($(B)+(0,0,0)$);
\coordinate (b1) at ($(b)+(0,0,0.5)$);

\filldraw[fill = orange!20] (A0) -- (B0) -- (b1) -- cycle;

\draw[dotted] (b) -- (b1);
\draw[thick,] (A) -- (B) -- (C) -- cycle;
\draw[thick,] (C) -- ($1/3*(C) + 1/3*(b)$);
\draw[thick,dashed] (b) -- ($1/2*(C) + 1/2*(b)$);
\draw[thick, dashed]  (b) -- (A);
\draw[thick,dashed] (B) -- (b);

\node[below right] at (b) {$\aKii$};
\node[above] at (C) {$\aKi$};
\node[below] at ($1/2*(A) + 1/2*(B)$) {$\Fi$};
\node[above right] at ($1/2*(b1) + 1/2*(B)$) {\large$\orange{\lambdai}$};

\end{scope}
\end{tikzpicture}
\caption{\footnotesize Subdivision of a macroelement into the three subelements~$\Ki \in \CcalhK$,
$i=1,2,3$ (left panel),
and visual representation of the barycentric coordinate~$\lambdai$ (right panel).}
\label{fig:one-element-subdivision}
\end{figure}

Applying this construction to every element of~$\Tauh$
yields a conforming, shape-regular, and quasi-uniform
simplicial partition~$\Ccalh$ of~$\Omega$,
which is a refinement of~$\Tauh$.
We denote by~$\Fcalhc$ the set of edges of~$\Ccalh$
and by $\Fcalhd := \Fcalhc \setminus \Fcalh$
the set of new edges introduced by the submesh.
In Figure~\ref{fig:mesh-subdivision},
we depict
an example of mesh~$\Tauh$,
its submesh~$\Ccalh$,
and the corresponding notation.
Throughout,
elements of both~$\Tauh$ and~$\Ccalh$ will appear simultaneously;
we reserve the notation~$\K$ for denoting a \emph{macro}element of~$\Tauh$,
whereas any additional subscript attached to~$\K$, e.g., $\Ki$,
will denote a \emph{sub}element of~$\Ccalh$.

\begin{figure}[tbp]
\centering
\begin{tikzpicture}[scale=1.8]
\begin{scope}[rotate=90]

\coordinate (A1) at (-1,0);
\coordinate (B1) at (1,0);
\coordinate (C1) at ($(0,{sqrt(3)})$);
\coordinate (D1) at ($(0,-{sqrt(3)})$);
\coordinate (bp1) at ($1/3*(A1) + 1/3*(B1) + 1/3*(C1)$);
\coordinate (bp2) at ($1/3*(A1) + 1/3*(B1) + 1/3*(D1)$);

\node[] at (bp1) {$\K$};
\node[] at (bp2) {$\K'$};

\draw (A1)--(B1)--(C1)--cycle;
\draw[thick,] (A1)--(D1)--(B1);
\draw[thick,] (A1)--(B1)--(C1)--cycle;

\coordinate (shift) at (0,-4);

\coordinate (A2) at ($(A1)+(shift)$);
\coordinate (B2) at ($(B1)+(shift)$);
\coordinate (C2) at ($(C1)+(shift)$);
\coordinate (D2) at ($(D1)+(shift)$);
\coordinate (b1) at ($1/3*(A2) + 1/3*(B2) + 1/3*(C2)$);
\coordinate (b2) at ($1/3*(A2) + 1/3*(B2) + 1/3*(D2)$);
\coordinate (bsub) at ($1/3*(A2) + 1/3*(b1) + 1/3*(C2)$);

\fill[cyan, opacity=0.2] (A2) -- (B2) -- (C2) -- cycle;
\draw[thick] (A2)--(D2)--(B2);

\draw[thick,] (A2) -- (B2) -- (C2) -- cycle;
\draw[thick,] (A2) -- (D2) -- (B2);
\draw[thick, dashed, red] (A2) -- (b2) -- (B2);
\draw[thick, dashed, red] (A2) -- (b1) -- (B2);
\draw[thick, dashed, red] (b1) -- (C2);
\draw[thick, dashed, red] (b2) -- (D2);

\node[] at (b1) {$\bullet$};
\node[] at (b2) {$\bullet$};
\node[above left] at (b1) {$\bbf_{K}$};
\node[above right] at (b2) {$\bbf_{K'}$};

\node[above left, cyan] at ($1/2*(C2) + 1/2*(B2)$) {$\Ccalh(\K)$};

\end{scope}
\end{tikzpicture}

\caption{\footnotesize
Example of a macroelement mesh~$\Tauh := \{\K, \K'\}$ (left panel)
and a corresponding submesh~$\Ccalh$ (right panel)
obtained by connecting the barycenters~$\bbf_{\K}$
and~$\bbf_{\K'}$ to the vertices of each element.
The new interior edges ($\red{\boldsymbol{- -}}$) are collected in~$\Fcalhd$.}
\label{fig:mesh-subdivision}
\end{figure} 

For non-negative integers~$s$ and~$k$,
we define the broken Sobolev and polynomial spaces associated with~$\Ccalh$ by
\begin{align*}
H^s(\Omega,\Ccalh) &:= \{ q \in \LtwoOmega \, | \, q|_{\K} \in H^s(\K)\,\, \forall\K\in\Ccalh \}, \\
\Pbbk(\Omega,\Ccalh) &:= \{ q \in \LtwoOmega \, | \, q|_{\K} \in \Pbbk(\K)\,\, \forall\K\in\Ccalh \}.
\end{align*}
Analogous definitions apply to the mesh~$\Tauh$,
and for vector- and tensor-valued functions.

For positive $A$ and $B$,
we shall denote with
$A\lesssim B$ and $A \gtrsim B$
the existence of positive constants~$C$ and~$c$
independent of the mesh size~$h$ and the viscosity~$\nu$,
such that
$c A \leq B$ and $B \leq C A$, respectively;
if both $A\lesssim B$ and $A \gtrsim B$ hold,
we use $A \simeq B$.

Given an edge~$\F$ in~$\Fcalh$,
we denote by~$\jump{\cdot}_{\F}$ the jump operator across~$\F$,
which coincides with the identity operator if~$\F$ belongs to~$\FcalhB$.
When no confusion occurs, we shall omit the subscript~$\F$,
and write~$\jump{\cdot}$ instead.
Given an element~$\K$ in~$\Tauh$ (resp. $\Ccalh$)
and an edge~$\F$ in~$\FcalK$,
then,
for all~$\qh$ in~$\Pbbk(\Tauh)$ (resp. $\Pbbk(\Ccalh)$),
the following discrete trace inequality holds true~\cite[Lem.~1.46]{DiPietro-Ern:2012}:
\begin{equation} \label{trace-inequality}
\Norm{\qh}_{\Ltwo(\F)} \lesssim \hK^{-\frac12} \Norm{\qh}_{\Ltwo(\K)}.
\end{equation}
Let~$\Fj$ be fixed as one of the two edges of an element~$\Ki$ in~$\CcalhK$
that is not the external edge~$\Fi$,
and let~$\lambda_j$
denote the barycentric coordinate associated with
the vertex of~$\Ki$ opposite to~$\Fj$.
Consider now the constant extension operator
$\Lcalk:\Pbbk(\Fi) \to \Pbbk(\Ki)$
defined by
extending functions constantly along the~$\lambdaj$-direction, i.e.,
given a point~$\xbf$ in~$\Ki$,
\begin{equation} \label{Lcal}
\Lcalk(\qkF)(\xbf) := \qkF(\xbf_{\Fi}),
\end{equation}
where $\xbf_{\Fi}$ in $\Fi$
is the unique point satisfying
 $\lambdaj(\xbf_{\Fi}) = \lambdaj(\xbf)$.
Then, a standard scaling argument yields
the following stability estimate for the extension operator~$\Lcalk$:
\begin{equation} \label{stability-Lcal}
\Norm{\Lcalk(\qkF)}_{\Ltwo(\Ki)}
\lesssim \hKi^{\frac12} \Norm{\qkF}_{\Ltwo(\Fi)}.
\end{equation}
We refer to Figure~\ref{fig:extansion-opeartor}
for a graphic representation of such an extension operator.
On each edge $\F$ in $\Fcalh$
with corresponding normal vector~$\nbfF$,
let~$\PizF$ be the~$\Ltwo$-orthogonal projection
with respect to the inner product~$(\cdot,\cdot)_{\Ltwo(\F)}$
into the space of constant vector fields.

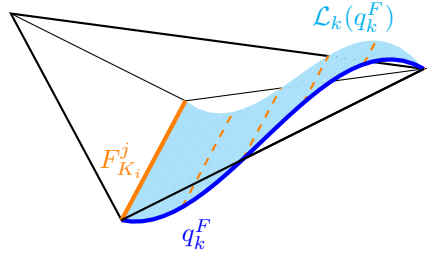
\begin{figure}[htbp]
\centering
\begin{tikzpicture}[scale=2]
\begin{scope}[
  x={(0.5cm,0.25cm)},
  y={(-0.5cm,0.25cm)},
  z={(0cm,0.3cm)},
  xshift=4cm,
  yshift=.5cm,
  scale=2]
\coordinate (A) at (-1,0);
\coordinate (B) at (1,0);
\coordinate (C) at ($(0,{sqrt(3)})$);
\coordinate (b) at ($1/3*(A) + 1/3*(B) + 1/3*(C)$);

\coordinate (A0) at ($(A)+(0,0,0)$);
\coordinate (B0) at ($(B)+(0,0,0)$);
\coordinate (b1) at ($(b)+(0,0,0.5)$);

\draw[thick,] (A) -- (B) -- (C) -- cycle;
\draw[] (B) -- (b) -- (C);
\node[left] at ($1/2*(A) + 1/2*(b)$) {$\orange{\Fj}$};

\node[below] at ($0.75*(A) + 0.25*(B) + (0,0,-0.3)$) {$\blue{\qkF}$};

\node[right] at ($0.15*(A) + 0.85*(B) + (0,+0.5,0.4)$) {$\cyan{\Lcalk(\qkF)}$};

\foreach \i in {0,...,19}{
    \pgfmathsetmacro{\ta}{\i/20}
    \pgfmathsetmacro{\tb}{(\i+1)/20}
    \pgfmathsetmacro{\ha}{-6*\ta*(\ta-0.4)*(\ta-1)}
    \pgfmathsetmacro{\hb}{-6*\tb*(\tb-0.4)*(\tb-1)}
    \fill[cyan!30,opacity=0.7]
        ({-1+2*\ta},{0},{\ha}) --
        ({-1+2*\tb},{0},{\hb}) --
        ({\tb},{0.57735*(1-\tb)},{\hb}) --
        ({\ta},{0.57735*(1-\ta)},{\ha}) -- cycle;
}

        \foreach \t in {0.2,0.4,0.6,0.8,1}{
            \pgfmathsetmacro{\hh}{-6*\t*(\t-0.4)*(\t-1)}
            \draw[orange,thick, dashed]
                ({-1+2*\t},{0},{\hh}) -- ({\t},{0.57735*(1-\t)},{\hh});
        }



 \draw[ultra thick,blue]
            plot[domain=0:1,samples=30,variable=\t]
            ({-1+2*\t},{0},{-6*\t*(\t-0.4)*(\t-1)});
\draw[orange, ultra thick] (A) -- (b);
\draw[thick] (A) -- (B);

\end{scope}
\end{tikzpicture}
\caption{\footnotesize Graphic representation of the constant extension operator~$\Lcalk$ of~\eqref{Lcal}.}
\label{fig:extansion-opeartor}
\end{figure}

\section{A symmetric MCS method} \label{section:MCS-symmetric}
We are now in a position to introduce a discrete formulation of~\eqref{variational-formulation}.
In particular,
the discrete stress finite element space
is constructed on the Clough--Tocher submesh
introduced in the previous section.

For an integer~$k \geq 1$,
we define the following spaces
\begin{subequations}  \label{discrete-spaces}
\begin{equation}
\Sigmahkmo := \{ \taubfh \in \Pbbkmo(\Ccalh,\Sbb) \, | \, \tr(\taubfh) = 0, \jump{(\taubfh)_{\normtang}}_{\F} = 0,\, \forall \F \in \Fcalh \},
\end{equation}
\begin{equation}
\Vhk     := \{ \vbfh \in \Pbbk(\Tauh, \Rbbd) \, | \, \jump{(\vbfh)_{\norm}}_{\F} = 0,\,  \forall \F \in \Fcalh \} = \Pbbk(\Tauh, \Rbbd) \cap \Vbf,
\end{equation}
\begin{equation}
\Qhkmo     := \Pbbkmo(\Tauh) \cap Q.
\end{equation}
\end{subequations}
The velocity and pressure spaces are chosen as in~\cite{Gopalakrishnan-Lederer-Schoberl:2020}, i.e.,
Brezzi--Douglas--Marini ($\BDM$) \cite{Boffi-Brezzi-Fortin:2013} finite elements of order $k$,
and discontinuous piecewise polynomials of degree $k-1$, respectively,
both defined on the mesh~$\Tauh$.
A key property of this pair of spaces is the relation
\begin{equation*}
\dive \Vhk = \Qhkmo, 
\end{equation*}
which yields the exact divergence-free property of the discrete velocity field:
\begin{equation} \label{exact-divergence-free}
(\dive \ubfh, \qh) = 0 \text{ for all } \qh \in \Qhkmo \iff \dive \ubfh = 0 \quad \text{in } \Omega.
\end{equation}
The choice of the discrete stress space, however,
differs from that adopted in~\cite{Gopalakrishnan-Lederer-Schoberl:2020-b}.
Here,
we consider piecewise tensor-valued polynomials of degree $k-1$
on the \emph{sub}mesh~$\Ccalh$,
which are \emph{symmetric} and trace-free.
Notice that the normal-tangential continuity of the stress space~$\Sigmahkmo$
is enforced only across the external edges $\Fi$ of~$\Fcalh$,
whereas no continuity is imposed across
the interior Clough--Tocher edges in~$\Fcalhd$
(i.e., across the edges denoted by $({\color{red}- -})$ in Figure~\ref{fig:mesh-subdivision}).
As discussed in~\cite[Sect.~4]{Gopalakrishnan-Lederer-Schoberl:2020},
requiring normal-tangential continuity
yields a ``slightly'' nonconforming space;
to achieve $H(\curlbf\divebf)$-conformity,
additional (and restrictive)
regularity assumptions on the stress should be imposed.
Finally,
we point out that both the discrete stresses on the interior
and their traces on the edges
are polynomials of degree $k-1$,
in contrast to the original MCS stress space,
where interior bubbles of order~$k$ were required
to guarantee stability.

To formulate the discrete problem,
in accordance with the symmetric MCS discrete formulation of~\cite{Gopalakrishnan-Lederer-Schoberl:2020-b},
we define the following bilinear forms:
\begin{itemize}
\item the bilinear form $a:\Ltwo(\Omega,\Sbb)\times\Ltwo(\Omega,\Sbb) \to \Rbb$ defined by
\begin{equation} \label{a}
a(\sigmabf,\taubf) := (\numo \dev\sigmabf, \dev\taubf );
\end{equation}
\item the bilinear form $\bone: \Vbf \times Q \to \Rbb$ defined by
\begin{equation} \label{bone}
\bone(\ubf,q) := (\dive\ubf,q);
\end{equation}
\item the bilinear form
$\btwo : \{ \taubf \in H^1(\Ccalh, \Rbbdd) \, | \, \jump{\taubf_{\normtang}} = 0\}
\times \{\vbf \in H^1(\Tauh,\Rbbd) \, | \, \jump{\vbf_{\norm}} = 0 \}
\to \Rbb$
defined by
\begin{equation} \label{btwo}
\btwo(\taubf,\ubf)
:= - \sum_{\K \in \Tauh} \sum_{i=1}^3 \int_{\Ki} \taubf : \epsbfun(\ubf)\dx 
+ \sum_{\F \in \Fcalh} \int_{\F} \taubf_{\normtang} \jump{\ubf_t} \ds.
\end{equation}
\end{itemize}
The discrete formulation reads as follows:
find $(\sigmabfh, \ubfh,\ph)$ in $\Sigmahkmo \times \Vhk \times \Qhkmo$ such that
\begin{subequations} \label{variational-formulation:discrete}
\begin{alignat}{2}
a(\sigmabfh,\taubfh) + \btwo(\taubfh, \ubfh) &= 0           \quad&&\forall \taubfh \in \Sigmahkmo, \\
\btwo(\sigmabfh, \vbfh) + \bone(\vbfh, \ph)        &= -(\fbf,\vbfh)  \quad&&\forall \vbfh \in \Vhk, \\
\bone(\ubfh, \qh)                               &= 0          \quad&&\forall \qh \in \Qhkmo.
\end{alignat}
\end{subequations}

The remainder of this section is devoted to the study of the discrete stress space~$\Sigmahkmo$,
and it is organized as follows.
In Section~\ref{subsection:dofs},
we introduce a unisolvent set of degrees of freedom for the corresponding local space
defined on a \emph{sub}element~$\Ki$,
which is the local stress discretization space,
whereas,
in Section~\ref{subsection:macro-local-space},
we discuss some properties of
the local space defined on a \emph{macro}element~$\K$,
which will be useful for the analysis of the method.
A hybrid version of
\eqref{variational-formulation:discrete}
is briefly discussed in Section~\ref{subsection:hybrid},
together with its relation to the original formulation
in terms of computational cost.

\subsection{Local stress space and degrees of freedom} \label{subsection:dofs}
Here,
we introduce a set of unisolvent degrees of freedom (DoFs)
for the global stress space~$\Sigmahkmo$.
With this aim,
we first construct a set of unisolvent degrees of freedom
for the corresponding local discrete stress space~$\Sigmahkmo(\Ki)$
defined below.

Let~$\K$ be an element of~$\Tauh$
and denote by~$\Ki$, $i=1,2,3$,
the elements in~$\CcalhK$.
We consider the local stress space
\begin{equation*}
\begin{split}
\Sigmahkmo(\Ki) := \{ \taubfh \in \Pbbkmo(\Ki,\Sbb)
\, | \,  \tr(\taubfh) = 0 \}.
\end{split}
\end{equation*}
The symmetry and trace-free constraints
imply that
the dimension of such a space can be computed as follows:
\begin{equation} \label{dim-sigmakmo}
\dim \Sigmahkmo(\Ki)
= \dim \Pbbkmo(\Ki,\Rbbdd)- 2\dim\Pbbkmo(\Ki)
= 2\dim\Pbbkmo(\Ki) = k(k+1).
\end{equation}
Let~$\Fi$ be the external edge of~$\Ki$,
and consider
the associated unit normal and tangential vectors,
which we recall have components
$(\ncompo,\ncompt)$
and
$(\tcompo,\tcompt) = (\ncompt,-\ncompo)$,
respectively.
We preliminarily collect some identities
involving the tensors
$\dev(\normi\otimes\normi)$
and
$\sym(\normi\otimes\tangi)$,
which will be used repeatedly in the sequel.
Their explicit matrix representations are:
\begin{equation*}
\begin{split}
\dev(\normi\otimes\normi)
&= \dev
\begin{pmatrix}
    \ncompo\ncompo  & \ncompo\ncompt   \\
    \ncompo\ncompt  & \ncompt\ncompt   
\end{pmatrix}
= \begin{pmatrix}
    \ncompo^2 - \frac12 & \ncompo\ncompt   \\
    \ncompo\ncompt  & \ncompt^2 - \frac12  
\end{pmatrix}
\end{split}
\end{equation*}
and
\begin{equation*}
\begin{split}
\sym(\normi\otimes\tangi)
&= \sym
\begin{pmatrix}
    \ncompo\tcompo  & \ncompo\tcompt   \\
    \ncompt\tcompo  & \ncompt\tcompt   
\end{pmatrix} 
= \begin{pmatrix}
    \ncompo\ncompt  & \frac12 (\ncompt^2 -\ncompo^2)   \\
    \frac12 ( \ncompt^2 - \ncompo^2)  & -\ncompt\ncompo
\end{pmatrix};
\end{split}
\end{equation*}
notice that in particular both tensors are symmetric and trace-free.
A direct computation shows that
their squared Frobenius norms
are
\begin{equation} \label{dev:frob-norm}
\begin{split}
|\dev(\normi\otimes\normi)|^2
&= \dev(\normi\otimes\normi):\dev(\normi\otimes\normi) \\
&=  (\ncompo^2 - \frac12)^2 + 2\ncompo^2\ncompt^2 + (\ncompt^2 - \frac12)^2
=  (\ncompo^2 + \ncompt^2)^2 - \ncompo^2 - \ncompt^2 + \frac12 = \frac12 
\end{split}
\end{equation}
and
\begin{equation} \label{sym:frob-norm}
\begin{split}
|\sym(\normi\otimes\tangi)|^2
&=\sym(\normi\otimes\tangi):\sym(\normi\otimes\tangi) \\
&=  2\ncompo^2\ncompt^2 + 2 \big(\frac12 (\ncompt^2 -\ncompo^2)\big)^2   
=  \frac12\big(\ncompo^2 + \ncompt^2\big)^2 = \frac12.
\end{split}
\end{equation}
Since $\sym(\normi\otimes\tangi)$ is symmetric,
we have
\begin{equation*}
\sym(\normi\otimes\tangi):\sym(\normi\otimes\tangi)
=\sym(\normi\otimes\tangi) : (\normi\otimes\tangi)
= \tangi^T \sym(\normi\otimes\tangi) \normi,
\end{equation*}
which,
combined with~\eqref{sym:frob-norm},
yields 
\begin{equation*} \label{nt-component-sym}
    (\sym(\normi\otimes\tangi))_{\normtang} = \frac12.
\end{equation*}
Moreover,
these two tensors satisfy the following orthogonality property:
\begin{equation} \label{orthogonality:dev-sym}
\begin{split}
\dev(\normi\otimes\normi) : \sym(\normi\otimes\tangi)
&= \dev(\normi\otimes\normi) : (\normi\otimes\tangi) \\
&= (\normi\otimes\normi) : (\normi\otimes\tangi)
    - \frac{1}{2}\tr(\normi\otimes\normi)\Idbf:(\normi\otimes\tangi) \\
&= (\normi\cdot\normi)(\normi\cdot \tangi) - \frac{1}{2}(\normi\cdot \tangi)
= 0.
\end{split}
\end{equation}
Since~$\dev(\normi\otimes\normi)$ is symmetric,
we have
\begin{equation*}
\dev(\normi\otimes\normi) : \sym(\normi\otimes\tangi)
= \dev(\normi\otimes\normi) : (\normi\otimes\tangi)
= \tangi^T \dev(\normi\otimes\normi) \normi,
\end{equation*}
which,
combined with~\eqref{orthogonality:dev-sym},
yields
\begin{equation*} \label{nt-component-dev}
(\dev(\normi\otimes\normi))_{\normtang} = 0.
\end{equation*}

The next result shows that these two tensors
form a basis for the lowest-order local discrete stress space~$\Sigmahz(\Ki)$
and
that they are involved in a representation of arbitrary order stress tensors
that will be crucial throughout the analysis.
\begin{lemma} \label{lemma:dev-sym-basis-Ki}
Let~$\Ki$ be an element of~$\Ccalh$.
Then,
the set $\{\dev(\normi\otimes\normi),\sym(\normi\otimes\tangi)\}$
is a basis for~$\Sigmahz(\Ki)$.
In particular,
for any given~$\taubfh$ in~$\Sigmahkmo(\Ki)$,
there exist polynomials~$\rkmoi$ and~$\skmoi$ in~$\Pbbkmo(\Ki)$
such that~$\taubfh$ admits the following representation:
\begin{equation*}
\taubfh = \rkmoi\dev(\normi\otimes\normi) + \skmoi\sym(\normi\otimes\tangi).
\end{equation*}
\end{lemma}
\begin{proof}
Since the set $\{\dev(\normi\otimes\normi),\sym(\normi\otimes\tangi)\}$
has two elements,
it suffices to show that they are linearly independent,
which directly follows from
the orthogonality property~\eqref{orthogonality:dev-sym}.
Using that
$\{\dev(\normi\otimes\normi),\sym(\normi\otimes\tangi)\}$
is a basis for~$\Sigmahz(\Ki)$,
writing
$\Sigmahkmo(\Ki) = \Pbbkmo(\Ki)\otimes\Sigmahz(\Ki)$
yields the desired decomposition.
\end{proof}

We now aim at introducing a set of degrees of freedom
for~$\Sigmahkmo(\Ki)$.
For a given~$\Ki$ in~$\Ccalh$
with external edge~$\Fi$,
consider the following three groups of linear functionals:
\begin{subequations} \label{dofs}
\begin{itemize}
\item
edge moments~$\psiEi(\taubfh)$ on~$\Fi$
associated with the normal-tangential component,
up to order~$k-1$:
    \begin{equation} \label{dofs:edge}
    \psiEi(\taubfh):= \Big\{
            \int_{\Fi} (\taubfh)_{\normtang}  \piF \ds,\,
            \forall \piF \in \Pbbkmo(\Fi)
                    \Big\};
    \end{equation}
\item bulk moments~$\psiKio(\taubfh)$
associated with the deviatoric component,
up to order~$k-1$:
    \begin{equation} \label{dofs:bulk-k}
    \psiKio(\taubfh) := \Big\{ \int_{\Ki} \taubfh : \dev(\normi\otimes\normi) \qi \dx,\,
    \forall\qi \in \Pbbkmo(\Ki) \Big\};
    \end{equation}
\item bulk moments~$\psiKit(\taubfh)$
associated with the normal-tangential component weighted by the barycenter coordinate,
up to order~$k-2$:
    \begin{equation} \label{dofs:bulk-kmo-lambda}
    \psiKit(\taubfh) := \Big\{ \int_{\Ki} \taubfh :
    \sym(\normi\otimes\tangi)
    \wi \lambdai \dx,\,
    \forall\wi \in \Pbbkmt(\Ki) \Big\};
\end{equation}
\end{itemize}
\end{subequations}

\begin{remark}
A natural alternative choice for the degrees of freedom
for the local discrete stress space~$\Sigmahkmo(\Ki)$
would be  to consider only two families of bulk moments
of degree up to order~$k-1$,
associated with the normal-normal
and normal-tangential components, respectively, i.e.,
\begin{subequations}
\begin{equation}
\int_{\Ki} \taubfh : \dev(\normi\otimes\normi) \qi \dx\qquad\qquad
    \forall\qi \in \Pbbkmo(\Ki),
\end{equation}
and
\begin{equation} \label{dofs:natural-bulk}
\int_{\Ki} \taubfh : \sym(\normi\otimes\tangi) \qi \dx\qquad\qquad
    \forall\qi \in \Pbbkmo(\Ki).
\end{equation}
\end{subequations}
The unisolvence of these functionals follows immediately
from the same arguments used in the proof of Lemma~\ref{lemma:unisolvence},
since $\dev(\normi\otimes\normi)$ and $\sym(\normi\otimes\tangi)$ form a basis for~$\Sigmahz(\Ki)$.
However,
we decompose the second family of bulk functionals~\eqref{dofs:natural-bulk}
into the two groups~\eqref{dofs:bulk-k} and~\eqref{dofs:bulk-kmo-lambda}
because the second term on the right-hand side in the definition of~$\btwo$ in~\eqref{btwo}
requires control of the normal-tangential component on the edges.
\eremk
\end{remark}

Throughout,
we will frequently use the following polynomial decomposition,
which allows us to split a polynomial on~$\Ki$
into an (external) edge and a bulk contribution.

\begin{lemma} \label{lemma:decomposition-polynomial-spaces}

The space of polynomials~$\Pbbkmo(\Ki)$ admits the following decomposition
\begin{equation} \label{decomposition-polynomial-spaces}
\Pbbkmo(\Ki) = {\Lcalkmo(\Pbbkmo(\Fi))}\oplus\lambdai\Pbbkmt(\Ki).
\end{equation}
\end{lemma}
\begin{proof}
We first show that the sum on the right-hand side of~\eqref{decomposition-polynomial-spaces} is direct:
let~$\pkmo$ belong to
$\Lcalkmo(\Pbbkmo(\Fi))\cap\lambdai\Pbbkmt(\Ki)$.
Then,
there exist~$\qkmoF$ in~$\Pbbkmo(\Fi)$
and~$\qkmt$ in~$\Pbbkmt(\Ki)$ such that
$\pkmo = \Lcalkmo(\qkmoF) = \lambdai\qkmt$.
Restricting the second identity to~$\Fi$
and using the definition of~$\Lcalkmo$
together with the fact that~$\lambdai$ vanishes on~$\Fi$,
we obtain
\begin{equation*}
\qkmo|_\Fi  = \Lcalkmo(\qkmoF){|_\Fi} = (\lambdai\qkmt)|_\Fi = 0.
\end{equation*}
Hence~$\pkmo = \Lcalkmo(0) = 0$,
the sum is direct,
and we obtain
\begin{equation*}
\dim(\Lcalkmo(\Pbbkmo(\Fi))\oplus\lambdai\Pbbkmt(\Ki))
= \dim(\Lcalkmo(\Pbbkmo(\Fi)))+\dim(\lambdai\Pbbkmt(\Ki)).
\end{equation*}
It remains to show that the sum space has the same dimension as~$\Pbbkmo(\Ki)$. By construction,
the extension operator~$\Lcalkmo$ of~\eqref{Lcal}
is injective,
since it is a right inverse of the trace map from~$\Ki$ to~$\Fi$.
Hence,
$\dim\Lcalkmo(\Pbbkmo(\Fi)) = \dim\Pbbkmo(\Fi)$.
Moreover,
$
\dim(\Pbbkmo(\Ki))
= {k(k+1)}/{2}
= k + {(k-1)k}/{2}.
$
Since both spaces in the direct sum are trivially contained in~$\Pbbkmo(\Ki)$,
the sum space is a subspace of~$\Pbbkmo(\Ki)$ with the same dimension.
Hence,
\eqref{decomposition-polynomial-spaces} follows.
\end{proof}

In particular,
Lemma~\ref{lemma:decomposition-polynomial-spaces} allows us to decompose the coefficient~$\skmoi$ of the normal-tangential component of a tensor,
introduced in the representation of Lemma~\ref{lemma:dev-sym-basis-Ki},
into an edge contribution and a bulk contribution, i.e.,
\begin{equation*}
\skmoi = \Lcalkmo(\skmoFi) + \lambdai \stildekmti,
\end{equation*}
where $\skmoFi$ is in~$\Pbbkmo(\Fi)$ and~$\stildekmti$ is in~$\Pbbkmt(\Ki)$.
These two terms naturally correspond to
the two groups of bulk functionals~\eqref{dofs:bulk-k} and~\eqref{dofs:bulk-kmo-lambda},
thereby allowing us to control the edge and interior contributions separately.
This additional structure will be crucial
for establishing both the stability properties and the optimal error estimates.

In the next result we show that the functionals in~\eqref{dofs}
can be taken as unisolvent set of degrees of freedom for
the local discrete stress space~$\Sigmahkmo(\Ki)$.
\begin{lemma} \label{lemma:unisolvence}
Let~$\Ki$ be an element of~$\Ccalh$.
Then,
the set of functionals $\psiEi$, $\psiKio$, and~$\psiKit$ of~\eqref{dofs}
is a unisolvent set of degrees of freedom for the space~$\Sigmahkmo(\Ki)$.
\end{lemma}
\begin{proof}
Let~$\taubfh$ in~$\Sigmahkmo(\Ki)$
be such that the moments~\eqref{dofs} vanish.
The set of functionals in~\eqref{dofs}
has the same cardinality as $\dim \Sigmahkmo(\Ki)$.
Indeed,
a direct computation
and~\eqref{dim-sigmakmo}
show
\begin{equation} \label{dimensions-check}
\begin{split}
\dim \Pbbkmo(\Fi) + \dim \Pbbkmo(\Ki) + \dim \Pbbkmt(\Ki)
&= k + \frac{k(k+1)}{2} + \frac{(k-1)k}{2} \\
&= k(k+1)
\overset{\eqref{dim-sigmakmo}}{=} \dim \Sigmahkmo(\Ki).
\end{split}
\end{equation}
Therefore,
it suffices to prove that~$\taubfh$ is identically zero over~$\Ki$.

We detail the proof for $k-1 \geq 1$.
The lowest-order case $k-1 = 0$ can be treated similarly
and is simpler,
since we no longer have the bulk moment of~\eqref{dofs:bulk-kmo-lambda}
among the degrees of freedom,
and the polynomial coefficients appearing in~\eqref{decomposition-polynomial-spaces} are constants.

Let $k-1 \geq 1$
and let~$\taubfh$ in~$\Sigmahkmo(\Ki)$ be such that
the moments~\eqref{dofs} vanish.
By Lemma~\ref{lemma:dev-sym-basis-Ki},
there exist polynomials~$\rkmoi$ and~$\skmoi$ in~$\Pbbkmo(\Ki)$
such that the restriction of~$\taubfh$ on each~$\Ki$
admits the following representation
\begin{equation} \label{decomposition-taubfh}
\taubfh{}|_{\Ki} = \rkmoi\dev(\normi\otimes\normi) + \skmoi\sym(\normi\otimes\tangi).
\end{equation}
We proceed by proving that both polynomial coefficients vanish.

We first show that~$\rkmoi$ vanishes.
Since the bulk moments~\eqref{dofs:bulk-k} vanish,
by applying the orthogonality property~\eqref{orthogonality:dev-sym}
and the identity~\eqref{dev:frob-norm},
we deduce that
for all~$\qi$ in~$\Pbbkmo(\Ki)$
the following identity holds true
\begin{equation*} 
\begin{split}
0 
\overset{\eqref{dofs:bulk-k}}{=}
\int_\Ki  \taubfh : \dev(\normi\otimes\normi) \qi \dx
&\overset{\eqref{decomposition-taubfh}}{=}
\int_\Ki  \rkmoi \dev(\normi\otimes\normi) : \dev(\normi\otimes\normi) \qi \dx \\
&\qquad+ \int_\Ki \skmoi \sym(\normi\otimes\tangi) :  \dev(\normi\otimes\normi) \qi \dx \\
&\overset{\eqref{dev:frob-norm},\eqref{orthogonality:dev-sym}}{=}
\frac12 \int_\Ki  \rkmoi \qi \dx.
\end{split}
\end{equation*}
Since~$\rkmoi$ is in~$\Pbbkmo(\Ki)$,
choosing~$\qi = \rkmoi$ in
the above display yields~$\rkmoi = 0$.
Hence,
the decomposition~\eqref{decomposition-taubfh} reads as
\begin{equation} \label{decomposition-taubfh-rkmoizero}
\taubfh|_{\Ki} = \skmoi\sym(\normi\otimes\tangi).
\end{equation}

We next show that~$\skmoi$ vanishes.
Since $\taubfh$ is symmetric,
its normal-tangential component on~$\Fi$ is
\begin{equation} \label{nt-tauh-on-FKi}
\begin{split}
(\taubfh)_\normtang{}|_{\Fi}
&
= \tangi^T(\taubfh|_{\Fi}) \normi
= \taubfh|_{\Fi} : (\tangi\otimes\normi)
= \taubfh|_{\Fi} : \sym(\normi\otimes\tangi)
\\
&\!\!\!\!\overset{\eqref{decomposition-taubfh}}{=}
(\rkmoi{}|_{\Fi}\dev(\normi\otimes\normi) + \skmoi{}|_{\Fi}\sym(\normi\otimes\tangi)): \sym(\normi\otimes\tangi) \\
&\overset{\eqref{orthogonality:dev-sym},\eqref{sym:frob-norm}}{=}
0 + \frac{1}{2}\skmoi{}|_{\Fi}.
\end{split}
\end{equation}
Since the edge moments~\eqref{dofs:edge} vanish 
and $2(\taubfh)_\normtang|_{\Fi} = \skmoi|_{\Fi}$ belongs to $\Pbbkmo(\Fi)$,
using the display above,
we deduce
\begin{equation*}
\begin{split}
0
\overset{\eqref{dofs:edge}}{=}
(\taubfh)_\normtang|_{\Fi}
\overset{\eqref{nt-tauh-on-FKi}}{=} 
\frac12\skmoi{}|_{\Fi},
\end{split}
\end{equation*}
i.e., $\skmoi$ vanishes on the edge~$\Fi$.
In particular,
due to the decomposition~\eqref{decomposition-polynomial-spaces},
there exists a polynomial~$\stildekmti$ in~$\Pbbkmt(\Ki)$
such that
\begin{equation} \label{skmoi-decomposition}
\skmoi = \lambdai \stildekmti,
\end{equation}
and the decomposition~\eqref{decomposition-taubfh-rkmoizero} reads as
\begin{equation} \label{tau-lambda-stildekmti}
\taubfh|_{\Ki}
= \lambdai \stildekmti\sym(\normi\otimes\tangi).
\end{equation}

Since the bulk moments~\eqref{dofs:bulk-kmo-lambda} vanish as well,
by applying the identity~\eqref{sym:frob-norm}
we deduce that,
for all~$\qkmt$ in~$\Pbbkmt(\Ki)$,
the following identity holds true
\begin{equation*} 
\begin{split}
0 
\overset{\eqref{dofs:bulk-kmo-lambda}}{=}
\int_\Ki  \taubfh :  \sym(\normi\otimes\tangi) \lambdai\wi \dx
&\overset{\eqref{tau-lambda-stildekmti}}{=}
\int_\Ki \lambdai \stildekmti \sym(\normi\otimes\tangi) :  \sym(\normi\otimes\tangi) \lambdai\wi \dx \\
&\overset{\eqref{sym:frob-norm}}{=}
\frac12 \int_\Ki \lambdai^2 \stildekmti \wi \dx.
\end{split}
\end{equation*}
Since~$\stildekmti$ is in~$\Pbbkmt(\Ki)$,
and the barycentric coordinate~$\lambdai$ is positive on~$\Ki$,
taking~$\wi = \stildekmti$ in the above display yields~$\stildekmti = 0$.
Using~\eqref{skmoi-decomposition},
we deduce that~$\skmoi = 0$.
\end{proof}

A set of global degrees of freedom for the discrete space~$\Sigmahkmo$
is obtained by a coupling of the degrees of freedom in~\eqref{dofs}.

\begin{remark}
In the standard MCS method~\cite{Gopalakrishnan-Lederer-Schoberl:2020-b},
a suitable combination of the covariant and contravariant Piola transformations
is employed to map the stress space from the reference element to the physical element,
while preserving the normal-tangential component
and the trace-free property.
In the present setting,
we would further need to preserve the symmetry of the stress tensor,
but considering the symmetric version of the above mapping is not sufficient:
a direct computation shows that,
although it preserves symmetry,
it no longer preserves the normal-tangential component.
Instead of pursuing the construction of a completely different mapping,
we rely on alternative arguments in the analysis,
which therefore does not require any specific mapping for the space~$\Sigmahkmo$.
\eremk
\end{remark}

\subsection{A technical tool: a local discrete space on~$\K$} \label{subsection:macro-local-space}

Given an element~$\K$ of the mesh~$\Tauh$
and one of its subelements~$\Ki$ in~$\CcalhK$,
both Lemma~\ref{lemma:dev-sym-basis-Ki}
and Lemma~\ref{lemma:unisolvence}
in the previous section
rely on the basis
$\{\dev(\normi\otimes\normi),\sym(\normi\otimes\tangi)\}$
of~$\Sigmahz(\Ki)$.
This space coincides with the lowest-order local discrete stress space~$\Sigmahkmo(\Ki)$,
which is defined on one \emph{sub}element.
For the stability analysis,
however,
it is essential to also exploit the structure of a local discrete space
defined on the element~$\K$.
To this aim,
we introduce the following local space
\begin{equation*}
\Sigmahkmo(K) := \{ \taubfh \in \Pbbkmo(K,\Sbb) \, | \, \tr(\taubfh) = 0 \},
\end{equation*}
to which the deviatoric part of
the local restriction of the symmetric gradient
of the discrete velocity field belongs.
Below,
we discuss some structural properties of this space.
We begin by showing that every tensor in~$\Sigmahkmo(K)$
admits a representation in terms of basis functions of
the lowest-order local space~$\Sigmahz(K)$
with polynomial coefficients in~$\Pbbkmo(K)$.
This representation
will play a key role in the proof of the discrete inf-sup condition.

\begin{lemma} \label{lemma:dev-n1-n2-basis}
Given~$\K$ in~$\Tauh$,
there exist two distinct indices $\itt,\jtt$ in $\{1,2,3\}$
such that the set
$\{\dev(\normitt\otimes\normitt),\dev(\normjtt\otimes\normjtt)\}$
forms a basis for~$\Sigmahz(\K)$.
In particular,
for any given~$\taubfh$ in~$\Sigmahkmo(\K)$,
there exist polynomials~$\lkmo$ and~$\mkmo$ in~$\Pbbkmo(\K)$
such that~$\taubfh$ admits the following representation:
\begin{equation*}
\taubfh = \lkmo \dev(\normitt\otimes\normitt) + \mkmo \dev(\normjtt\otimes\normjtt).
\end{equation*}
Moreover,
the following estimate holds true
\begin{equation} \label{strict-CS}
|\dev(\normitt\otimes\normitt):\dev(\normjtt\otimes\normjtt)|
< |\dev(\normitt\otimes\normitt)||\dev(\normjtt\otimes\normjtt)|
= \frac12.
\end{equation}
\end{lemma}
\begin{proof}
The shape-regularity assumption guarantees that
there exists at least a pair of distinct indices $(\itt,\jtt)$ in $\{1,2,3\}$,
such that the corresponding normal vectors $\normitt$ and $\normjtt$
are neither parallel nor orthogonal.
Consider the lowest-order local space $\Sigmahz(\K)$.
Using the same dimension count as in~\eqref{dimensions-check},
we obtain $\dim \Sigmahz(\K) = 2$.
Since the set $\{\dev(\normitt\otimes\normitt),\dev(\normjtt\otimes\normjtt)\}$
comprises two elements,
it suffices to show that they are linearly independent.

Let~$\alpha$ and~$\beta$ be in~$\Rbb$.
Assume that
\begin{equation*}
\zerobf
= \alpha\dev(\normitt\otimes\normitt) + \beta \dev(\normjtt\otimes\normjtt)
= \alpha(\normitt\otimes\normitt - \frac{1}{d}\Idbf) + \beta(\normjtt\otimes\normjtt - \frac{1}{d}\Idbf),
\end{equation*}
where we have applied the definition of the deviatoric operator~$\dev$.
Applying the above identity to the tangential vector~$\tangjtt$,
and using the orthogonality property between $\normjtt$ and $\tangjtt$
together with the identity
$(\abf\otimes\bbf)\cbf = (\bbf \cdot \cbf)\abf$
for all $\abf$, $\bbf$, $\cbf$ in $\Rbbd$,
gives
\begin{equation*} \label{alpha-beta-basisn1n2}
\begin{split}
0
= \alpha(\normitt\otimes\normitt - \frac{1}{d}\Idbf)\tangjtt + \beta(\normjtt\otimes\normjtt - \frac{1}{d}\Idbf)\tangjtt
= \alpha(\normitt \cdot \tangjtt)\normitt - \frac{\alpha+\beta}{d}\tangjtt.
\end{split}
\end{equation*}
Since~$\normitt$ and~$\tangjtt$ are not parallel,
the above identity is satisfied if and only if both the coefficients
$\alpha(\normitt \cdot \tangjtt)$
and
$(\alpha+\beta)/d$
vanish;
since~$\normitt$ and~$\tangjtt$ are not orthogonal,
$\normitt \cdot \tangjtt \neq 0$,
which,
combined with $\alpha(\normitt \cdot \tangjtt)=0$,
yields~$\alpha = 0$.
Using this identity in $(\alpha+\beta)/d = 0$
in turn implies~$\beta = 0$.
Therefore,
the tensors $\dev(\normitt\otimes\normitt)$ and $\dev(\normjtt\otimes\normjtt)$
are linearly independent and form a basis of~$\Sigmahz(\K)$.

The existence of the representation
$
\taubfh = \lkmo \dev(\normitt\otimes\normitt) + \mkmo \dev(\normjtt\otimes\normjtt),
$
for~$\lkmo$ and~$\mkmo$ in~$\Pbbkmo(\K)$,
follows from the decomposition
$
\Sigmahkmo(\K) = \Pbbkmo(\K)\otimes\Sigmahz(\K).
$
The Cauchy--Schwarz inequality yields
$
|\dev(\normitt\otimes\normitt):\dev(\normjtt\otimes\normjtt)|
\leq |\dev(\normitt\otimes\normitt)||\dev(\normjtt\otimes\normjtt)|.
$
Since
$\dev(\normitt\otimes\normitt)$ and $\dev(\normjtt\otimes\normjtt)$
are linearly independent,
equality cannot hold in the above display,
which combined with the identity~\eqref{dev:frob-norm},
yields the assertion.
\end{proof}

Using the indices~$\itt$ and~$\jtt$
from the previous Lemma,
we further deduce the following decompositions,
which will be crucial
in the proof of the discrete inf-sup condition
in Section~\ref{section:a-priori}.

\begin{corollary} \label{corollary:decomposition-epsvh}
Given~$\K$ in~$\Tauh$
and~$\vbfh$ in~$\Vhk$,
there exist polynomials~$\lkmo$ and~$\mkmo$ in~$\Pbbkmo(\K)$
such that the restriction of~$\dev(\epsbfun(\vbfh))$ to~$\K$
admits the following representation:
\begin{equation} \label{eps-vh:dev-dev-representation-dev}
\dev(\epsbfun(\vbfh)|_{\K}) = \lkmo \dev(\normitt\otimes\normitt) + \mkmo \dev(\normjtt\otimes\normjtt).
\end{equation}
Moreover,
if~$\dive(\vbfh|_{\K})$ vanishes,
the same representation holds also for the restriction of~$\epsbfun(\vbfh)$ to~$\K$, i.e.,
\begin{equation} \label{eps-vh:dev-dev-representation}
\epsbfun(\vbfh)|_{\K} = \lkmo \dev(\normitt\otimes\normitt) + \mkmo \dev(\normjtt\otimes\normjtt).
\end{equation}
\end{corollary}
\begin{proof}
Since for every element~$\K$ in~$\Tauh$,
$\dev(\epsbfun(\vbfh))|_{\K}$
belongs to~$\Sigmahkmo(\K)$ by construction,
identity~\eqref{eps-vh:dev-dev-representation-dev} follows by applying Lemma~\ref{lemma:dev-n1-n2-basis}.

In the particular case in which $\dive(\vbfh|_{\K}) = 0$,
$\epsbfun(\vbfh)$ is trace-free and
we have
\begin{equation} \label{deveps-equal-eps}
\dev(\epsbfun(\vbfh))
= \epsbfun(\vbfh) - \frac{1}{d}\tr(\epsbfun(\vbfh)) \Idbf 
= \epsbfun(\vbfh) - \frac{1}{d}\dive(\vbfh) \Idbf 
= \epsbfun(\vbfh).
\end{equation}
Hence,
decomposition~\eqref{eps-vh:dev-dev-representation}
follows from combining the display above with~\eqref{eps-vh:dev-dev-representation-dev}.
\end{proof}

Throughout,
the notation~$\lkmo$ and~$\mkmo$
will always refer to the polynomial coefficients
appearing
in the decomposition~\eqref{eps-vh:dev-dev-representation-dev} of~$\dev(\epsbfun(\vbfh))$.


\subsection{An alternative hybrid variational formulation} \label{subsection:hybrid}
As an equivalent alternative formulation for the mixed method~\eqref{strong-mixed},
one can consider a hybrid discretization
in which an interface variable is introduced to approximate
the tangential component of the trace of the stress.
In particular,
we introduce the discrete space
\begin{align*}
\SigmaHDG := \{ \taubfh  \in \Pbbkmo(\Ccalh,\Sbb)\,| \, \tr(\taubfh) = 0 \},
\end{align*}
together with the interface space
\begin{align*}
\VhatHDG := \{ \vhhat \in \Pbbkmo(\Fcalh,\Rbbd) \, | \, (\vhhat)_{\norm}|_{\F} = 0,\, \forall \F \in \Fcalh,
                                                      \, \vhhat|_{\Gamma} = 0 \}. 
\end{align*}
The tangential facet space~$\VhatHDG$
is introduced to enforce continuity of the normal-tangential component of the stress.
We denote the corresponding product space by
\begin{equation*}
\UHDG := \SigmaHDG \times \Vhk \times \Qhkmo \times \VhatHDG.
\end{equation*}
Given the bilinear form
\begin{equation*} 
\btwohyb(\taubfh;(\ubfh,\uhhat))
:= - \sum_{\K \in \Tauh} \sum_{i=1}^3 \int_{\Ki} \taubfh : \epsbfun(\ubfh)\dx 
+ \int_{\partial\K} (\taubfh)_{\normtang} ((\ubfh-\uhhat)_{t}) \ds,
\end{equation*}
we consider the following \emph{hybrid} MCS formulation of~\eqref{variational-formulation:discrete}:
find $(\sigmabfh, \ubfh,\ph, \uhhat)$ in $\UHDG$ such that
\begin{subequations} \label{variational-formulation:discrete-hybrid}
\begin{alignat}{2}
a(\sigmabfh,\taubfh) + \btwohyb(\taubfh;(\ubfh,\uhhat))      &= 0                    \quad&&\forall \taubfh \in \SigmaHDG \\
\btwohyb(\sigmabfh;(\vbfh,\vhhat)) + \bone(\vbfh, \ph)        &= -(\fbf,\vbfh)        \quad&&\forall (\vbfh, \vhhat) \in \Vhk \times \VhatHDG \\
\bone(\ubfh, \qh)                                                                               &= 0                    \quad&&\forall \qh \in \Qhkmo.
\end{alignat}
\end{subequations}

Since hybridization yields a formulation that is equivalent
to the original mixed method;
see, e.g., \cite[Ch.~7.2.2]{Boffi-Brezzi-Fortin:2013},
the hybrid formulation~\eqref{variational-formulation:discrete-hybrid}
and the (non-hybridized) formulation~\eqref{variational-formulation:discrete}
produce the same discrete solution.
For the sake of clarity,
we briefly detail the steps of the proof. 
Let $(\sigmabfh,\ubfh,\ph, \uhhat)$ be the solution to~\eqref{variational-formulation:discrete-hybrid}.
Testing the second equation of~\eqref{variational-formulation:discrete-hybrid}
with $(\vbfh, \vhhat) = (0, \jump{(\sigmabfh)_{\normtang}})$,
which is a possible choice since
$(\sigmabfnt)_{t}$ is a tangential polynomial of order~$k-1$ on each~$\F$ in~$\Fcalh$,
we obtain
\begin{equation*}
\sum_{\F \in \Fcalh} \int_{\F} \jump{(\sigmabfh)_{\normtang}}^2 \ds = 0,
\end{equation*}
thus $\sigmabfh$ is $\normtang$-continuous and~$\sigmabfh$ belongs to~$\Sigmahkmo$.
Moreover,
testing the first equation of~\eqref{variational-formulation:discrete-hybrid}
with a stress tensor~$\taubfh$ in~$\Sigmahkmo$,
we obtain the first equation of~\eqref{variational-formulation:discrete}
since~$\btwohyb(\taubfh; (\ubfh,\uhhat))$ boils down to~$\btwo(\taubfh; \ubfh)$.
Hence,
$(\sigmabfh,\ubfh,\ph)$ solves also~\eqref{variational-formulation:discrete},
and the claimed equivalence follows from the uniqueness of the solution to such a formulation.

In particular,
we will employ the hybrid version solely as an implementation tool
for efficiently solving the resulting linear system
through static condensation as detailed below.
For this reason,
we omit the details concerning the solvability
and analysis of~\eqref{variational-formulation:discrete-hybrid},
which follows from the same arguments used in~\cite{Gopalakrishnan-Kogler-Lederer-Schoberl:2023}
for the analysis of a different hybridized method
corresponding to the standard MCS formulation.
An alternative hybridization for~\eqref{variational-formulation:discrete} would be to
also approximate the normal component of the velocity by an additional interface variable.
In this case,
one can introduce a second facet space
$\QhatHDG := \Pbbk(\Fcalh)$,
set $\VHDG := \Pbbk(\Tauh,\Rbbd)$,
and modify the discrete formulation accordingly.
We refer to~\cite{Rhebergen-Wells:2018,Rhebergen-Wells:2022}
for a detailed analysis of the linear system stemming from the hybridization procedure.

Using the hybrid formulation~\eqref{variational-formulation:discrete-hybrid}
and the above comments,
we can compare the computational cost of~\eqref{variational-formulation:discrete}
with the standard MCS method~\cite{Gopalakrishnan-Lederer-Schoberl:2020},
for which details on hybridization and
static condensation can be found in~\cite[Sect.~4.2]{Kogler-Lederer-Schoberl:2023}.
Since the velocity and pressure discrete spaces coincide in both methods,
the resulting globally coupled unknowns after static condensation are the same.
The stress space,
instead,
is now defined on Clough--Tocher macroelements,
which may introduce additional coupling.
However,
a crucial feature of the space~$\Sigmahkmo$ of~\eqref{discrete-spaces}
is that no continuity is imposed
across the interior Clough--Tocher edges in~$\Fcalhd$;
hence,
the only coupled degrees of freedom
are those associated with the external edges in~$\Fcalh$,
whereas all the remaining are purely local.
After eliminating all interior degrees of freedom by static condensation,
the dimension of the global linear system associated with~\eqref{variational-formulation:discrete-hybrid}
is therefore the same as that of the standard MCS method~\cite{Gopalakrishnan-Lederer-Schoberl:2020}
and
the computational cost of the two methods,
as well as that of other standard mixed methods
for the discretization of the Stokes problem,
is comparable.
At the same time,
the novel method of~\eqref{variational-formulation:discrete} enforces the symmetry of the stress tensor exactly.

\section{A priori error analysis} \label{section:a-priori}
In this section,
we establish the stability and convergence properties
of method~\eqref{variational-formulation:discrete}.
In line with the a priori analysis of~\cite{Gopalakrishnan-Lederer-Schoberl:2020,Gopalakrishnan-Lederer-Schoberl:2020-b},
we equip the discrete spaces in~\eqref{discrete-spaces}
with the following norms:
\begin{align*}
&\Norm{\taubfh}_{\Sigmah}^2
    := \Norm{\taubfh}_{\LtwoOmega}^2
    =\sum_{\K \in \Tauh} \NormLtwoK{\taubfh}^2
    = \sum_{\K \in \Tauh} \sum_{i=1}^3 \Norm{\dev(\taubfh)}^2_{\Ltwo(\Ki)},  \\
&\Norm{\vbfh}_{\Vh}^{2}
    := \Norm{\vbfh}_{1,h}^{2}
    := \sum_{\K \in \Tauh} \Norm{\epsbfun(\vbfh)}_{\Ltwo(\K)}^2 + \sum_{\F \in \Fcalh}\frac{1}{h}\Norm{\jump{(\vbfh)_{\tangscalar}}}_{\Ltwo(\F)}^2, \\
&\Norm{\qh}_{\Qh}^2 := \Norm{\qh}_{\LtwoOmega}^2.
\end{align*}

The a priori analysis is organized as follows.
In Section~\ref{subsection:basis-norm-equivalence},
we establish preliminary norm equivalences involving the norms introduced above.
These tools are then used
in Section~\ref{subsection:stability}
to prove continuity,
kernel coercivity,
and discrete inf-sup stability
for the bilinear forms arising in the proposed method.
Finally,
in Section~\ref{subsection:error-estimates},
we construct an interpolation operator for the discrete stress space
and derive optimal error estimates for~\eqref{variational-formulation:discrete}.

Throughout,
two different types of indices will be used.
For the reader's convenience,
we briefly recall their meaning and role.
The indices~$i$ and~$j$ are associated with subelements~$\Ki$ and~$\Kj$ in~$\Ccalh$,
and are used to locally decompose a stress tensor in~$\Sigmahkmo(\Ki)$ or~$\Sigmahkmo(\Kj)$,
as in Lemma~\ref{lemma:dev-sym-basis-Ki}.
The indices~$\itt$ and~$\jtt$ are instead associated with a macroelement~$\K$ in~$\Tauh$,
and are used together to locally decompose the symmetric gradient
(or its deviatoric part)
of a velocity field in~$\Vhk$,
as in Corollary~\ref{corollary:decomposition-epsvh}.

\subsection{Norm equivalences} \label{subsection:basis-norm-equivalence}
Here,
we are interested in deriving some norm equivalences
involving a discrete velocity~$\vbfh$ in~$\Vhk$
and its symmetric gradient~$\epsbfun(\vbfh)$,
with a particular emphasis on divergence-free functions.

Consider the discrete divergence-free subspace
$\Vhkz := \{ \vbfh \in \Vhk \, | \, \dive(\vbfh) = 0 \}$.
Then, given a discrete velocity field~$\vbfh$ in~$\Vhkz$,
the tensor $\epsbfun(\vbfh)|_{\K}$
is trace-free, cf.~\eqref{deveps-equal-eps},
symmetric,
and belongs to~$\Pbbkmo(\K,\Rbbdd)$
for every~$\K$ in~$\Tauh$ by construction.
Therefore,
\begin{equation*}
\epsbfun(\vbfh)|_{\K} \in \Sigmahkmo(\K)
\qquad
\forall \vbfh \in \Vhkz, \quad \forall \K \in \Tauh.
\end{equation*}
In particular,
for every~$\vbfh$ in~$\Vhkz$,
$\epsbfun(\vbfh)|_{\K}$ admits the representation in~\eqref{eps-vh:dev-dev-representation}.

Introduce now the mesh-dependent norm
\begin{equation*} 
\Norm{\vbfh}_{1,\dev,h}^{2} := \sum_{\K \in \Tauh} \Norm{\dev\epsbfun(\vbfh)}_{\Ltwo(\K)}^2 + \sum_{\F \in \Fcalh}\frac{1}{h}\Norm{\PizF\jump{(\vbfh)_{\tangscalar}}}_{\Ltwo(\F)}^2.
\end{equation*}
By combining 
the identity~\eqref{deveps-equal-eps}
with the stability and approximation properties of the projector operator~$\PizF$
and a Korn-type inequality; see, e.g., \cite{Brenner:2004},
we obtain
\begin{equation} \label{norm-equivalence-h-dev}
\Norm{\vbfh}_{\Vh} \simeq \Norm{\vbfh}_{1,\dev,h},
\end{equation}
for all~$\vbfh$ in~$\Vhkz$.
We refer the reader to~\cite[Lem.~6.2]{Gopalakrishnan-Lederer-Schoberl:2020} and \cite[Sect.~3.2]{Gopalakrishnan-Kogler-Lederer-Schoberl:2023}
for details regarding the proof of the above norm equivalence.

We next compare suitable norms involving
the symmetric gradient of a divergence-free discrete velocity field and
the polynomial coefficients appearing in the decomposition~\eqref{eps-vh:dev-dev-representation-dev}.
\begin{lemma} \label{lemma:norm-equivalence}
Let~$\K$ be an element in~$\Tauh$
and~$\vbfh$ be in~$\Vhkz$.
Consider its symmetric gradient~$\epsbfun(\vbfh)$
and let~$\lkmo$ and~$\mkmo$ be the polynomials in~$\Pbbkmo(\K)$
of its representation on~$\K$; cf.~\eqref{eps-vh:dev-dev-representation-dev}.
Then,
the following norm equivalence holds true
\begin{equation} \label{norm-equivalence:statement}
\Norm{\epsbfun(\vbfh)}_{\Ltwo(\K)}^2
\simeq \sum_{i = 1}^3\big(\Norm{\lkmo}_{\Ltwo(\Ki)}^2 + \Norm{\mkmo}_{\Ltwo(\Ki)}^2\big).
\end{equation}
\end{lemma}
\begin{proof}
We first show the upper bound.
The decomposition~\eqref{eps-vh:dev-dev-representation-dev},
the identity~\eqref{dev:frob-norm},
and
the triangle inequality,
give
\begin{equation} \label{norm-equivalence:leq}
\begin{split}
\Norm{\epsbfun(\vbfh)}_{\Ltwo(\K)}^2
&= \Norm{\lkmo\dev(\normitt\otimes\normitt) + \mkmo\dev(\normjtt\otimes\normjtt)}_{\Ltwo(\K)}^2 \\
&\lesssim \Norm{\lkmo}_{\Ltwo(\K)}^2 + \Norm{\mkmo}_{\Ltwo(\K)}^2
= \sum_{i=1}^3 \big(\Norm{\lkmo}_{\Ltwo(\Ki)}^2 + \Norm{\mkmo}_{\Ltwo(\Ki)}^2\big).
\end{split}
\end{equation}

We now show the lower bound.
Using the decomposition~\eqref{eps-vh:dev-dev-representation-dev},
expanding the square,
applying identity~\eqref{dev:frob-norm},
the Cauchy--Schwarz inequality,
and the Young inequality,
we obtain
\begin{equation*}
\begin{split}
\Norm{\epsbfun(\vbfh)}_{\Ltwo(\K)}^2
&= \Norm{\lkmo\dev(\normitt\otimes\normitt) + \mkmo\dev(\normjtt\otimes\normjtt)}_{\Ltwo(\K)}^2 \\
&\geq \frac12\Norm{\lkmo}_{\Ltwo(\K)}^2 + \frac12\Norm{\mkmo}_{\Ltwo(\K)}^2 \\
&\qquad        - 2  |\dev(\normitt\otimes\normitt):\dev(\normjtt\otimes\normjtt)| \Norm{\lkmo}_{\Ltwo(\K)}\Norm{\mkmo}_{\Ltwo(\K)} \\
&\geq \frac12\Norm{\lkmo}_{\Ltwo(\K)}^2 + \frac12\Norm{\mkmo}_{\Ltwo(\K)}^2 \\
&\qquad        - 2|\dev(\normitt\otimes\normitt):\dev(\normjtt\otimes\normjtt)| \Big(\frac12\Norm{\lkmo}_{\Ltwo(\K)}^2
                    + \frac12\Norm{\mkmo}_{\Ltwo(\K)}^2\Big) \\
&= \big(\frac12-|\dev(\normitt\otimes\normitt):\dev(\normjtt\otimes\normjtt)|\big)
    \sum_{i=1}^3 \big(\Norm{\lkmo}_{\Ltwo(\Ki)}^2 + \Norm{\mkmo}_{\Ltwo(\Ki)}^2 \big) .
\end{split}
\end{equation*}
Since by estimate~\eqref{strict-CS} we have
$|\dev(\normitt\otimes\normitt):\dev(\normjtt\otimes\normjtt)| < 1/2$,
the coefficient in the above display
is positive,
and we deduce
\begin{equation} \label{norm-equivalence:geq}
\Norm{\epsbfun(\vbfh)}_{\Ltwo(\K)}^2 \gtrsim  \sum_{i=1}^3 \big(\Norm{\lkmo}_{\Ltwo(\Ki)}^2 + \Norm{\mkmo}_{\Ltwo(\Ki)}^2\big).
\end{equation}
Combining the bounds~\eqref{norm-equivalence:leq} and~\eqref{norm-equivalence:geq},
the assertion follows.
\end{proof}

When needed,
we explicitly denote the constants hidden
in the norm equivalence~\eqref{norm-equivalence:statement} by
\begin{equation} \label{norm-equivalence:explicit-constants}
\ceqone\Norm{\epsbfun(\vbfh)}_{\Ltwo(\K)}
\leq \Big(\sum_{i = 1}^3\big(\Norm{\lkmo}_{\Ltwo(\Ki)}^2 + \Norm{\mkmo}_{\Ltwo(\Ki)}^2\big)\Big)^{\frac12}
\leq \Ceqone\Norm{\epsbfun(\vbfh)}_{\Ltwo(\K)}.
\end{equation}
This equivalence will be used in the stability analysis below.

\subsection{Stability analysis} \label{subsection:stability}
We now address the stability analysis of the method~\eqref{variational-formulation:discrete}.
We begin by establishing appropriate continuity and coercivity properties
of the bilinear forms
and then prove discrete inf-sup conditions.
For simplicity,
throughout we assume that the viscosity~$\nu$ is constant.

In the next result,
we show the continuity of the bilinear forms.
\begin{lemma} \label{lemma:continuity}
The bilinear forms~$a(\cdot,\cdot)$, $\bone(\cdot,\cdot)$, and~$\btwo(\cdot,\cdot)$
of~\eqref{a}--\eqref{btwo} 
are continuous.
\end{lemma}
\begin{proof}
As for the bilinear forms~$a(\cdot,\cdot)$ of~\eqref{a} and~$\bone(\cdot,\cdot)$ of~\eqref{bone},
their continuity follows from the Cauchy--Schwarz inequality.

As for the bilinear form~$\btwo(\cdot,\cdot)$ of~\eqref{btwo},
bounding the normal-tangential component of the stress tensor by the stress tensor itself,
and applying
the trace-free property of the functions in~$\Sigmahkmo$,
we have
\begin{equation*}
\Norm{(\sigmabfh)_{\normtang}}_{\Ltwo(\F)}
\lesssim \Norm{\dev(\sigmabfh)}_{\Ltwo(\F)}.
\end{equation*}
Using the definition~\eqref{btwo} of~$\btwo(\cdot,\cdot)$,
the Cauchy--Schwarz inequality,
the above display,
the identity~\eqref{sym:frob-norm},
the discrete trace inequality~\eqref{trace-inequality},
and the shape-regularity of the mesh,
we obtain
\begin{equation*}
\begin{split}
\btwo(\sigmabfh,\vbfh)
&= - \sum_{\K \in \Tauh} \int_{\K} \sigmabfh : \epsbfun(\vbfh) \dx
    + \sum_{\F \in \Fcalh} \int_{\F} (\sigmabfh)_{\normtang} \jump{(\vbfh)_{\tangscalar}} \ds \\
&\lesssim \sum_{\K \in \Tauh} \Norm{\sigmabfh}_{\Ltwo(\K)} \Norm{\epsbfun(\vbfh)}_{\Ltwo(\K)}
    + \sum_{\F \in \Fcalh} \Norm{\dev(\sigmabfh)}_{\Ltwo(\F)} \Norm{\jump{(\vbfh)_{\tangscalar}}}_{\Ltwo(\F)} \\
&\lesssim  \sum_{\K \in \Tauh} \Norm{\sigmabfh}_{\Ltwo(\K)} \Norm{\epsbfun(\vbfh)}_{\Ltwo(\K)}
    + \sum_{\Ki \in \Ccalh} \Norm{\dev(\sigmabfh)}_{\Ltwo(\Ki)} h^{-\frac12} \Norm{\jump{(\vbfh)_{\tangscalar}}}_{\Ltwo(\Fi)} \\
&\lesssim \Norm{\sigmabfh}_{\Sigmah} \Norm{\vbfh}_{\Vh}
\end{split}
\end{equation*}
for all~$\vbfh$ in~$\Vhk$ and~$\sigmabfh$ in~$\Sigmahkmo$.
\end{proof}

We next introduce the discrete kernel
associated with the velocity and pressure discrete spaces
\begin{equation*}
\Zh := \{ (\taubfh, \qh) \in \Sigmahkmo \times \Qhkmo \, | \, \bone(\vbfh,\qh) + \btwo(\taubfh,\vbfh) = 0,
            \, \forall \vbfh \in \Vhk \},
\end{equation*}
and discuss the coercivity of the bilinear form~$a(\cdot,\cdot)$ on~$\Zh$
in the next result.
\begin{lemma} \label{lemma_coercivity-kernel}
The bilinear form~$a(\cdot,\cdot)$ of~\eqref{a}
is coercive on~$\Zh$, i.e.,
for all $(\taubfh,\qh)$ in~$\Zh$,
the following estimate holds true
\begin{equation*}
a(\taubfh,\taubfh) \gtrsim \frac{1}{\nu}(\Norm{\taubfh}_{\Sigmah} + \Norm{\qh}_{\Qh})^2.
\end{equation*}
\end{lemma}
\begin{proof}
The proof follows along the same lines as those of~\cite[Lemma~6.4]{Gopalakrishnan-Lederer-Schoberl:2020},
once we apply the continuity of the bilinear form~$\btwo(\cdot,\cdot)$ shown in Lemma~\ref{lemma:continuity}.
\end{proof}

We show an inf-sup condition for the bilinear form~$\btwo$ of~\eqref{btwo} on~$\Vhkz$,
by constructing a specific stress function~$\taubfh$ in~$\Sigmahkmo$
which only depends on~$\dev(\epsbfun(\vbfh))$ for any given~$\vbfh$ in~$\Vhkz$.
\begin{lemma} \label{lemma:discrete-inf-sup-bt}
For any~$\vbfh$ in~$\Vhkz$,
there exists~$\taubfh$ in~$\Sigmahkmo$ such that
\begin{subequations} \label{inf-sup:i-ii}
\begin{equation} \label{inf-sup:i}
\btwo(\taubfh,\vbfh) \gtrsim \Norm{\vbfh}_{1,\dev,h}^2
\end{equation}
\begin{equation} \label{inf-sup:ii}
\Norm{\taubfh}_{\Sigmah} \lesssim \Norm{\vbfh}_{1,\dev,h}.
\end{equation}
\end{subequations}
These estimates imply
the discrete inf-sup stability of the bilinear form~$\btwo(\cdot,\cdot)$ of~\eqref{btwo} on~$\Vhkz$,
i.e.,
for all~$\vbfh$ in $\Vhkz$,
\begin{equation} \label{discrete-inf-sup:btwo}
\sup_{\taubfh \in \Sigmahkmo} \frac{\btwo(\taubfh,\vbfh)}{\Norm{\taubfh}_{\Sigmah}}
    \gtrsim\Norm{\vbfh}_{1,h}.
\end{equation}
\end{lemma}
\begin{proof}
We detail the proof for $k-1 \geq 1$.
The lowest-order case $k-1 = 0$ can be treated similarly
and is simpler,
since all the estimates below that rely on the combination of
the degrees of freedom~\eqref{dofs:edge} and~\eqref{dofs:bulk-kmo-lambda}
can be directly obtained by using just the edge moments~\eqref{dofs:edge}.

Let~$\vbfh$ be any given function in~$\Vhkz$;
we proceed by constructing a stress function~$\taubfh$ in~$\Sigmahkmo$
that satisfies the estimates~\eqref{inf-sup:i}--\eqref{inf-sup:ii}.
Let~$\gamma > 0$ be a positive constant to be fixed later.
We define~$\taubfh$
by fixing the corresponding degrees of freedom in~\eqref{dofs} as follows.
We impose:
\begin{itemize}
\item on every edge~$\F$ in~$\Fcalh$, for all~$\piF$ in~$\Pbbkmo(\F)$,
\begin{equation}  \label{stability-dofs:1}
\int_{\F} (\taubfh)_{\normtang} \piF \ds = \frac{1}{h} \int_{\F} \PizF\jump{(\vbfh)_{\tangscalar}}\piF \ds;
\end{equation}
\item on every subelement~$\Ki$ in~$\Ccalh$, for all~$\qi$ in~$\Pbbkmo(\Ki)$,
\begin{equation} \label{stability-dofs:2}
\int_{\Ki} \taubfh : \dev(\normi\otimes\normi) \qi \dx = - \gamma\int_{\Ki} \epsbfun(\vbfh) : \dev(\normi\otimes\normi) \qi \dx;
\end{equation}
\item on every subelement~$\Ki$ in~$\Ccalh$, for all~$\wi$ in~$\Pbbkmt(\Ki)$,
\begin{equation} \label{stability-dofs:3}
\int_{\Ki} \taubfh : \sym(\normi\otimes\tangi)\wi\lambdai \dx = 0.
\end{equation}
\end{itemize}
Now let~$\K$ be any fixed element in~$\Tauh$,
and let~$\Ki$ be in~$\CcalhK$.
Since~$\taubfh$ belongs to~$\Sigmahkmo$,
as discussed in Lemma~\ref{lemma:dev-sym-basis-Ki},
there exist polynomials~$\rkmoi$ and~$\skmoi$ in~$\Pbbkmo(\Ki)$
such that the restriction of~$\taubfh$ on~$\Ki$
admits the following representation:
\begin{equation} \label{tau-decomposition}
\taubfh|_{\Ki} = \rkmoi\dev(\normi\otimes\normi) + \skmoi\sym(\normi\otimes\tangi).
\end{equation} 
Similarly,
since~$\epsbfun(\vbfh)$ belongs to~$\Sigmahkmo(\K)$,
as discussed in Corollary~\ref{corollary:decomposition-epsvh},
there exist polynomials~$\lkmo$ and $\mkmo$ in~$\Pbbkmo(\K)$
such that the restriction of~$\epsbfun(\vbfh)$ on~$\K$
admits the following representation:
\begin{equation} \label{eps-decomposition}
\epsbfun(\vbfh)|_{\K} = \lkmo\dev(\normitt\otimes\normitt) + \mkmo\dev(\normjtt\otimes\normjtt).
\end{equation}
Since the subelements~$\Ki \in \CcalhK$ are subsets of $\K$,
we can restrict the representation~\eqref{eps-decomposition}
to each~$\Ki$.

We now give alternative representations
for the coefficients of the decomposition~\eqref{tau-decomposition} above,
starting from~$\skmoi$.
Since $(\taubfh)_{\normtang}|_{\F}$
belongs to~$\Pbbkmo(\F)$,
identity~\eqref{stability-dofs:1},
entails that,
for all~$\F$ in~$\Fcalh$,
\begin{equation*}
\begin{split}
\frac{1}{h}\PizF\jump{(\vbfh)_{\tangscalar}}
\overset{\eqref{stability-dofs:1}}{=}
(\taubfh)_{\normtang}|_{\F} 
&\overset{\eqref{tau-decomposition}}{=}
\big(\rkmoi\dev(\normi\otimes\normi) + \skmoi\sym(\normi\otimes\tangi)\big)_{\normtang}|_{\F} 
\overset{\eqref{orthogonality:dev-sym},\eqref{sym:frob-norm}}{=}
\frac{1}{2}\skmoi|_{\F}.
\end{split}
\end{equation*}
In particular,
for each~$\Ki$ in~$\CcalhK$
with corresponding external edge~$\Fi$ in~$\Fcalh$,
it holds
\begin{equation}  \label{skmoi-value-on-edge}
\frac{1}{h}\PizFi\jump{(\vbfh)_{\tangscalar}}
= \frac{1}{2}\skmoi|_{\Fi}.
\end{equation}
Since~$\skmoi$ belongs to~$\Pbbkmo(\Ki)$,
the above display combined with the decomposition~\eqref{decomposition-polynomial-spaces},
yield the existence of a polynomial~$\stildekmti$ in~$\Pbbkmt(\Ki)$
such that
\begin{equation} \label{skmoi-stildekmti}
\skmoi = \frac{2}{h}\Lcalzero(\PizFi\jump{(\vbfh)_{\tangscalar}}) + \stildekmti\lambdai,
\end{equation}
where
$\Lcalzero(\PizFi\jump{(\vbfh)_{\tangscalar}})$ denotes
the constant extension of $\PizFi\jump{(\vbfh)_{\tangscalar}}$ from~$\Fi$ over~$\Ki$
as defined in~\eqref{Lcal}.

We now also give a representation for the polynomial coefficient~$\rkmoi$ of~\eqref{tau-decomposition}.
Inserting the \eqref{tau-decomposition}
in~\eqref{stability-dofs:2},
and applying the orthogonality property~\eqref{orthogonality:dev-sym},
we obtain that,
for all~$\qi$ in $\Pbbkmo(\Ki)$,
the following identity holds true
\begin{equation*} \label{rkmoi-integral-value}
\int_{\Ki} \rkmoi \dev(\normi\otimes\normi) : \dev(\normi\otimes\normi) \qi \dx = - \gamma\int_{\Ki} \epsbfun(\vbfh) : \dev(\normi\otimes\normi) \qi \dx.
\end{equation*}
Since the polynomial coefficient $\rkmoi$
belongs to~$\Pbbkmo(\Ki)$
and the symmetric gradient $\epsbfun(\vbfh)|_{\Ki}$
and the tensor $\dev(\normi\otimes\normi)$
have entries in~$\Pbbkmo(\Ki)$, and~$\Pbbz(\Ki)$,
respectively,
the above display
combined with the identity~\eqref{dev:frob-norm}
yields
\begin{equation} \label{rkmoi-value}
\begin{split}
\rkmoi
&= -2\gamma \dev(\normi\otimes\normi) : \epsbfun(\vbfh)|_{\Ki}.
\end{split}
\end{equation}

We are now ready to prove the estimates in~\eqref{inf-sup:i-ii},
starting with~\eqref{inf-sup:i}.
Plugging~$\taubfh$ and~$\vbfh$
in the bilinear form~$\btwo(\cdot,\cdot)$ of~\eqref{btwo},
we have
\begin{equation*} \label{bt-Ione-Itwo}
\begin{split}
\btwo(\taubfh,\vbfh) 
= -\sum_{\K \in \Tauh} \sum_{i=1}^3 \int_{\Ki} \taubfh : \epsbfun(\vbfh) \dx
    + \sum_{\F \in \Fcalh} \int_{\F} (\taubfh)_{\normtang}\jump{(\vbfh)_{\tangscalar}} \ds 
=: \Ione + \Itwo.
\end{split}
\end{equation*}

We start by estimating~$\Ione$.
Consider the constants
$B_{ij} := \dev(\normi\otimes\normi):\sym(\normj\otimes\tangj)$;
in particular,
Cauchy--Schwarz inequality gives~$|B_{ij}| \leq 1$.
Note that below,
one of the two indices will always be fixed to~$\itt$ or~$\jtt$
of Lemma~\ref{lemma:dev-n1-n2-basis}.
Using the representations~\eqref{tau-decomposition} of~$\taubfh$
and~\eqref{eps-decomposition} of~$\epsbfun(\vbfh)$  on each~$\Ki$,
the orthogonality~\eqref{orthogonality:dev-sym},
and the identity~\eqref{rkmoi-value},
we have
\begin{equation} \label{Ione}
\begin{split}
-\sum_{i=1}^3& \int_{\Ki} \taubfh : \epsbfun(\vbfh) \dx 
\overset{\eqref{tau-decomposition}}{=}
-\sum_{i=1}^3 \int_{\Ki} (\rkmoi\dev(\normi\otimes\normi)+\skmoi\sym(\normi\otimes\tangi)) : \epsbfun(\vbfh) \dx \\
&\overset{\eqref{orthogonality:dev-sym},\eqref{eps-decomposition}}{=} 
-\sum_{i=1}^3 \int_{\Ki} 
   \rkmoi\,(\dev(\normi\otimes\normi):\epsbfun(\vbfh))
   + \skmoi\,(\lkmo B_{i\itt} + \mkmo B_{i\jtt}) \dx \\
&\overset{\eqref{rkmoi-value}}{=}
\sum_{i=1}^3 \int_{\Ki} 
   2\gamma(\dev(\normi\otimes\normi):\epsbfun(\vbfh))^2 
   \dx
   -\sum_{i=1}^3 \int_{\Ki} \skmoi\,(\lkmo B_{i\itt} + \mkmo B_{i\jtt}) \dx\\
&=: {\IKone} + {\IKtwo}.
\end{split}
\end{equation}
We estimate the two terms on the right-hand side separately.

We first estimate~$\IKone$.
Since~$\dev(\normi\otimes\normi):\epsbfun(\vbfh)$ is a polynomial,
a scaling argument combined with~$|\Ki|=|K|/3$
and the shape-regularity assumption
yield
that there exists a positive constant~$\Ceqtwo$ such that
\begin{equation*}
\Norm{\dev(\normi\otimes\normi):\epsbfun(\vbfh)}_{\Ltwo(\K)}
\leq \Ceqtwo \Norm{\dev(\normi\otimes\normi):\epsbfun(\vbfh)}_{\Ltwo(\Ki)},
\end{equation*}
which implies
\begin{equation} \label{IKone:first-estimate}
\IKone 
= \sum_{i=1}^3 2\gamma
  \Norm{\dev(\normi\otimes\normi):\epsbfun(\vbfh)}_{\Ltwo(\Ki)}^2
\geq 2\gamma\Ceqtwo^2 \sum_{i=1}^3  
  \Norm{\dev(\normi\otimes\normi):\epsbfun(\vbfh)}_{\Ltwo(\K)}^2.
\end{equation}
Consider the constants
$D_{ij} := \dev(\normi\otimes\normi):\dev(\normj\otimes\normj)$;
with this notation,
identity~\eqref{dev:frob-norm} reads as $|D_{ii}|=1/2$.
Note that below,
the two indices will always be fixed to~$\itt$ and~$\jtt$ of Lemma~\ref{lemma:dev-n1-n2-basis}.
Neglecting a non-negative term,
using the representation~\eqref{eps-decomposition} of~$\epsbfun(\vbfh)$ on~$\K$
and the definition of~$D_{ij}$,
and expanding the squares,
we obtain
\begin{equation} \label{sumdev-eps}
\begin{split}
\sum_{i = 1}^3&\Norm{\dev(\normi\otimes\normi):\epsbfun(\vbfh)}_{\Ltwo(\K)}^2 
\geq \sum_{i \in\itt,\jtt}\Norm{\dev(\normi\otimes\normi):\epsbfun(\vbfh)}_{\Ltwo(\K)}^2 \\
&\overset{\eqref{eps-decomposition}}{=} 
\sum_{i \in\itt,\jtt}\Norm{\dev(\normi\otimes\normi):(\lkmo\dev(\normitt\otimes\normitt)+\mkmo\dev(\normjtt\otimes\normjtt))}_{\Ltwo(\K)}^2 \\
&=\Norm{\lkmo D_{\itt\itt} + \mkmo D_{\itt\jtt}}_{\Ltwo(\K)}^2 
     + \Norm{\lkmo D_{\itt\jtt}+ \mkmo D_{\jtt\jtt}}_{\Ltwo(\K)}^2 \\
&=\Norm{\frac12\lkmo + \mkmo D_{\itt\jtt}}_{\Ltwo(\K)}^2 
     + \Norm{\lkmo D_{\itt\jtt}+ \frac12\mkmo}_{\Ltwo(\K)}^2 \\
&= \frac14\Norm{\lkmo}^2_{\Ltwo(\K)} + |D_{\itt\jtt}|^2\Norm{\mkmo}^2_{\Ltwo(\K)} + \int_\K\lkmo\mkmo D_{\itt\jtt}\dx \\
        &\quad +|D_{\itt\jtt}|^2\Norm{\lkmo}^2_{\Ltwo(\K)} + \frac14\Norm{\mkmo}^2_{\Ltwo(\K)} +\int_\K\lkmo\mkmo D_{\itt\jtt}\dx \\
&=\Big(\frac14+|D_{\itt\jtt}|^2\Big)(\Norm{\lkmo}^2_{\Ltwo(\K)} + \Norm{\mkmo}^2_{\Ltwo(\K)})+ 2\int_\K\lkmo\mkmo D_{\itt\jtt}\dx \\
\end{split}
\end{equation}
We estimate the last term on the right-hand side.
The Cauchy--Schwarz inequality and
the Young inequality
yield
\begin{equation*}
\begin{split}
2\int_\K\lkmo\mkmo D_{\itt\jtt}\dx
&\geq - 2\int_\K|\lkmo\mkmo D_{\itt\jtt}|\dx 
\geq - 2|D_{\itt\jtt}|\Norm{\lkmo}_{\Ltwo(\K)}\Norm{\mkmo}_{\Ltwo(\K)} \\
&\geq - 2|D_{\itt\jtt}|\big(\frac12\Norm{\lkmo}^2_{\Ltwo(\K)} + \frac12\Norm{\mkmo}^2_{\Ltwo(\K)}\big).
\end{split}
\end{equation*}
Using the estimate above in~\eqref{sumdev-eps}
and applying Lemma~\ref{lemma:norm-equivalence}
with the notation of~\eqref{norm-equivalence:explicit-constants},
gives
\begin{equation*}
\begin{split}
&\sum_{i = 1}^3\Norm{\dev(\normi\otimes\normi):\epsbfun(\vbfh)}_{\Ltwo(\K)}^2 
\geq \Big(\frac14+|D_{\itt\jtt}|^2-|D_{\itt\jtt}|\Big)(\Norm{\lkmo}^2_{\Ltwo(\K)} + \Norm{\mkmo}^2_{\Ltwo(\K)}) \\
&\qquad= \Big(\frac12-|D_{\itt\jtt}|\Big)^2(\Norm{\lkmo}^2_{\Ltwo(\K)} + \Norm{\mkmo}^2_{\Ltwo(\K)})
\geq \Big(\frac12-|D_{\itt\jtt}|\Big)^2\ceqone^2\Norm{\epsbfun(\vbfh)}^2_{\Ltwo(\K)} \\
&\qquad=: \Ceqthree^2\Norm{\epsbfun(\vbfh)}^2_{\Ltwo(\K)}.
\end{split}
\end{equation*}
Note that due to~\eqref{strict-CS}, $|D_{\itt\jtt}|<1/2$ and
consequently, $\Ceqthree$ is a positive constant.
Inserting the above display in~\eqref{IKone:first-estimate},
we obtain
\begin{equation} \label{IKone}
\IKone \geq 2\gamma\Ceqtwo^2\Ceqthree^2\Norm{\epsbfun(\vbfh)}_{\Ltwo(\K)}^2.
\end{equation}

We now estimate~$\IKtwo$.
Inserting the representations~\eqref{tau-decomposition} of~$\taubfh$
and~\eqref{skmoi-stildekmti} of~$\skmoi$ on~$\Ki$
in identity~\eqref{stability-dofs:3},
we obtain that,
for all~$\wi$ in~$\Pbbkmt(\Ki)$
it holds
\begin{equation} \label{zero-sum-Lcal-skmoi}
\begin{split}
0
&\overset{\eqref{stability-dofs:3}}{=}
\int_{\Ki} \taubfh : \sym(\normi\otimes\tangi)\wi\lambdai \dx \\
&\overset{\eqref{tau-decomposition}}{=}
\int_{\Ki} (\rkmoi\dev(\normi\otimes\normi) + \skmoi\sym(\normi\otimes\tangi)) : \sym(\normi\otimes\tangi)\wi\lambdai \dx \\
&\overset{\eqref{orthogonality:dev-sym},\eqref{skmoi-stildekmti}}{=}
\int_{\Ki} 0 + \big(\frac{2}{h}\Lcalzero(\PizFi\jump{(\vbfh)_{\tangscalar}}) + \stildekmti\lambdai\big) \sym(\normi\otimes\tangi) : \sym(\normi\otimes\tangi)\wi\lambdai \dx \\
&\overset{\eqref{sym:frob-norm}}{=}
\frac12 \int_{\Ki} \big(\frac{2}{h}\Lcalzero(\PizFi\jump{(\vbfh)_{\tangscalar}}) + \stildekmti\lambdai\big)\wi\lambdai \dx.
\end{split}
\end{equation}
Choosing
$\wi = \stildekmti$ in the above display
and applying
the Cauchy--Schwarz inequality
and the stability estimate~\eqref{stability-Lcal} of the constant extension operator~$\Lcalzero$,
give
\begin{equation*}
\begin{split}
\Norm{\stildekmti\lambdai}_{\Ltwo(\Ki)}^2
&\overset{\eqref{zero-sum-Lcal-skmoi}}{\lesssim}
\frac{1}{h}\int_{\Ki}|\Lcalzero(\PizFi\jump{(\vbfh)_{\tangscalar}})\stildekmti\lambdai| \dx \\
&\leq \frac{1}{h}\Norm{\Lcalzero(\PizFi\jump{(\vbfh)_{\tangscalar}})}_{\Ltwo(\Ki)}\Norm{\stildekmti\lambdai}_{\Ltwo(\Ki)} \\
&\overset{\eqref{stability-Lcal}}{\lesssim}
h^{-\frac12}\Norm{\PizFi\jump{(\vbfh)_{\tangscalar}}}_{\Ltwo(\Fi)}\Norm{\stildekmti\lambdai}_{\Ltwo(\Ki)}.
\end{split} 
\end{equation*}
Then,
taking the $\Ltwo$ norm on both sides of the identity~\eqref{skmoi-stildekmti}
and applying 
the triangle inequality,
the above display,
the stability estimate~\eqref{stability-Lcal},
and the shape-regularity assumption,
yield
the existence of a positive constant~$\Ctilde$
such that
\begin{equation} \label{skmoi-leq-PizF}
\begin{split}
\Norm{\skmoi}_{\Ltwo(\Ki)}
&\leq\Norm{\stildekmti\lambdai}_{\Ltwo(\Ki)} + \frac{2}{h}\Norm{\Lcalzero(\PizFi\jump{(\vbfh)_{\tangscalar}})}_{\Ltwo(\Ki)} \\
&\leq \Ctilde h^{-\frac12}\Norm{\PizFi\jump{(\vbfh)_{\tangscalar}}}_{\Ltwo(\Fi)}.
\end{split}
\end{equation}
Considering now the term~$\IKtwo$ in~\eqref{Ione},
applying the Cauchy--Schwarz inequality,
the triangle inequality,
the estimates~\eqref{skmoi-leq-PizF}
and~$|B_{ij}| \leq 1$ for~$ij = \{i\itt,i\jtt\}$,
the weighted Young inequality with parameter~$\delta > 0$,
and Lemma~\ref{lemma:norm-equivalence}
with the notation of~\eqref{norm-equivalence:explicit-constants},
we have
\begin{equation} \label{IKtwo}
\begin{split}
|\IKtwo|
&\leq \sum_{i=1}^3 \Norm{\skmoi}_{\Ltwo(\Ki)} \Norm{\lkmo B_{i\itt} + \mkmo B_{i\jtt}}_{\Ltwo(\Ki)} \\
&\leq \sum_{i=1}^3 \Norm{\skmoi}_{\Ltwo(\Ki)} \big(\Norm{\lkmo B_{i\itt}}_{\Ltwo(\Ki)} + \Norm{\mkmo B_{i\jtt}}_{\Ltwo(\Ki)}\big) \\
&\leq \sum_{i=1}^3 \Ctilde \frac{1}{h^{\frac12}} \Norm{\PizFi\jump{(\vbfh)_{\tangscalar}}}_{\Ltwo(\Fi)} \big(\Norm{\lkmo}_{\Ltwo(\Ki)} + \Norm{\mkmo}_{\Ltwo(\Ki)}\big) \\
&\leq \Ctilde\frac{1}{2\delta}\sum_{i=1}^3 \frac{1}{h} \Norm{\PizFi\jump{(\vbfh)_{\tangscalar}}}^2_{\Ltwo(\Fi)}
        + \delta\sum_{i=1}^3 \big(\Norm{\lkmo}_{\Ltwo(\Ki)}^2 + \Norm{\mkmo}_{\Ltwo(\Ki)}^2\big) \\
&= \frac{\Ctilde}{2\delta}\sum_{i=1}^3 \frac{1}{h} \Norm{\PizFi\jump{(\vbfh)_{\tangscalar}}}^2_{\Ltwo(\Fi)}
        + \delta\Ceqone^2\Norm{\epsbfun(\vbfh)}_{\Ltwo(\K)}^2.
\end{split}
\end{equation}
Summing over the elements~$\K$ in~$\Tauh$
and plugging~\eqref{IKone} and~\eqref{IKtwo} in~\eqref{Ione},
since one edge in~$\Fcalh$ is shared at most by two elements in~$\Tauh$,
gives
\begin{equation} \label{Ione-final}
\begin{split}
\Ione
&\geq \sum_{\K\in \Tauh}(\IKone - |\IKtwo|) \\
&\geq \Big(2\gamma\Ceqtwo^2\Ceqthree^2-\delta\Ceqone^2\Big) \sum_{\K\in \Tauh} \Norm{\epsbfun(\vbfh)}_{\Ltwo(\K)}^2
            - \frac{\Ctilde}{\delta}\sum_{\F\in\Fcalh}\frac{1}{h}{\Norm{\PizF\jump{(\vbfh)_{\tangscalar}}}}_{\Ltwo(\F)}^2.
\end{split}
\end{equation}

We now estimate~$\Itwo$.
Using the identity~\eqref{skmoi-value-on-edge}
and the~$\Ltwo$-orthogonality of~$\PizF$,
we have
\begin{equation} \label{Itwo}
\Itwo 
= \sum_{\F \in \Fcalh} \frac{1}{h}\int_\F \PizF\jump{(\vbfh)_{\tangscalar}} 
  \cdot \jump{(\vbfh)_{\tangscalar}} \ds
= \sum_{\F \in \Fcalh} \frac{1}{h}
  \Norm{\PizF\jump{(\vbfh)_{\tangscalar}}}_{\Ltwo(\F)}^2.
\end{equation}

Combining the estimates~\eqref{Ione-final} and~\eqref{Itwo}
for~$\Ione$ and~$\Itwo$,
gives
\begin{equation*}
\begin{split}
\btwo(\taubfh,&\vbfh)
= \Ione + \Itwo \geq \sum_{\K\in\Tauh}(\IKone - |\IKtwo|) + \Itwo \\
&\geq \Big(2\gamma\Ceqtwo^2\Ceqthree^2-\delta\Ceqone^2\Big)\sum_{\K \in \Tauh}\Norm{\epsbfun(\vbfh)}_{\Ltwo(\K)}^2
  + \Big(1 - \frac{\Ctilde}{\delta}\Big)
  \sum_{\F \in \Fcalh}
  \frac{1}{h}\Norm{\PizF\jump{(\vbfh)_{\tangscalar}}}_{\Ltwo(\F)}^2.
\end{split}
\end{equation*}
By choosing the positive parameters~$\delta$ and~$\gamma$ as
\begin{equation} \label{delta-gamma}
\delta = 2\Ctilde,
\quad \text{and}\quad
\gamma = 2\Ctilde\Big(\frac{\Ceqone}{\Ceqtwo\Ceqthree}\Big)^2,
\end{equation}
and using the identity~\eqref{deveps-equal-eps},
we conclude
\begin{equation*}
\btwo(\taubfh,\vbfh)
\gtrsim  \sum_{\K \in \Tauh} \Norm{\dev(\epsbfun(\vbfh))}_{\Ltwo(\K)}^2
  + \sum_{\F \in \Fcalh}
  \frac{1}{h}\Norm{\PizF\jump{(\vbfh)_{\tangscalar}}}_{\Ltwo(\F)}^2,
\end{equation*}
which yields estimate~\eqref{inf-sup:i}.

We next show~\eqref{inf-sup:ii}.
Using the representation~\eqref{tau-decomposition} of~$\taubfh$
and the orthogonality property~\eqref{orthogonality:dev-sym} twice and
choosing~$\qi = \rkmoi$ in~\eqref{stability-dofs:2},
applying
the estimate~\eqref{skmoi-leq-PizF},
the Cauchy--Schwarz inequality,
the weighted Young inequality with positive parameter~$\eta >0$,
and adding some non-negative terms,
yield
\begin{equation} \label{intermediate-tauhnorm}
\begin{split}
&\Norm{\taubfh}_{\Ltwo(\K)}^2
\overset{\eqref{tau-decomposition}}{\leq}
\sum_{i=1}^3 \Norm{\rkmoi\dev(\normi\otimes\normi)}_{\Ltwo(\Ki)}^2 + \sum_{i=1}^3\Norm{\skmoi\sym(\normi\otimes\tangi)}_{\Ltwo(\Ki)}^2 \\
&= \sum_{i=1}^3 \int_{\Ki}\rkmoi\dev(\normi\otimes\normi):\dev(\normi\otimes\normi)\rkmoi\dx + \sum_{i=1}^3\Norm{\skmoi\sym(\normi\otimes\tangi)}_{\Ltwo(\Ki)}^2 \\
&\overset{\eqref{stability-dofs:2},\eqref{sym:frob-norm}}{=}
\sum_{i=1}^3 - \gamma\int_{\Ki} \epsbfun(\vbfh):\dev(\normi\otimes\normi)\rkmoi \dx
    + \frac12\sum_{i=1}^3 \Norm{\skmoi}_{\Ltwo(\Ki)}^2 \\
&\overset{\eqref{skmoi-value-on-edge}}{\leq}
\sum_{i=1}^3 \gamma\Norm{\epsbfun(\vbfh)}_{\Ltwo(\K)} \Norm{\rkmoi \dev(\normi\otimes\normi)}_{\Ltwo(\Ki)}
    +  \frac{\Ctilde}{2}\sum_{i=1}^3\frac{1}{h} \Norm{\PizFi\jump{(\vbfh)_{\tangscalar}}}^2_{\Ltwo(\Fi)} \\
&{\leq}
\frac{\eta}{2}\gamma^2\Norm{\epsbfun(\vbfh)}^2_{\Ltwo(\K)}
    +  \frac{\Ctilde}{2}\sum_{i=1}^3 \frac{1}{h} \Norm{\PizFi\jump{(\vbfh)_{\tangscalar}}}^2_{\Ltwo(\Fi)} \\
&\qquad    + \sum_{i=1}^3 \frac{1}{2\eta}(\Norm{\rkmoi \dev(\normi\otimes\normi)}^2_{\Ltwo(\Ki)}
        + \Norm{\skmoi \sym(\normi\otimes\tangi)}^2_{\Ltwo(\Ki)}).
\end{split}
\end{equation}
Then,
combining the above display
with
$\dev(\normi\otimes\normi) : \skmoi \sym(\normi\otimes\tangi)= 0$
due to the orthogonality property~\eqref{orthogonality:dev-sym},
we obtain
\begin{equation*}
\begin{split}
\Norm{\taubfh}_{\Ltwo(\K)}^2
&\overset{\eqref{intermediate-tauhnorm},\eqref{orthogonality:dev-sym}}{\leq}
\frac{\eta}{2}\gamma^2\Norm{\epsbfun(\vbfh)}^2_{\Ltwo(\K)}
    +  \frac{\Ctilde}{2}\sum_{i=1}^3 \frac{1}{h} \Norm{\PizFi\jump{(\vbfh)_{\tangscalar}}}^2_{\Ltwo(\Fi)} \\
&\qquad    + \sum_{i=1}^3 \frac{1}{2\eta}\Big(
            \Norm{\rkmoi \dev(\normi\otimes\normi)}^2_{\Ltwo(\Ki)}
          + \Norm{\skmoi \sym(\normi\otimes\tangi)}^2_{\Ltwo(\Ki)} \\
&\qquad   + 2\int_{\Ki} \rkmoi \dev(\normi\otimes\normi) : \skmoi \sym(\normi\otimes\tangi) \dx\Big) \\ 
&= \frac{\eta}{2}\gamma^2\Norm{\epsbfun(\vbfh)}^2_{\Ltwo(\K)}
    + \frac{\Ctilde}{2} \sum_{i=1}^3 \frac{1}{h} \Norm{\PizFi\jump{(\vbfh)_{\tangscalar}}}^2_{\Ltwo(\Fi)}
    + \sum_{i=1}^3 \frac{1}{2\eta}\Norm{\taubfh}^2_{\Ltwo(\Ki)},
\end{split}
\end{equation*}
where the constant~$\gamma$ is fixed as in~\eqref{delta-gamma}.
Using the identity~\eqref{deveps-equal-eps}
in the above display
yields
\begin{equation*} \label{upper-bound-taubfh}
\big(1-\frac{1}{2\eta}\big) \Norm{\taubfh}_{\Ltwo(\K)}^2
\leq \frac12(\eta\gamma^2 + \Ctilde) \Big(\Norm{\dev(\epsbfun(\vbfh))}_{\Ltwo(\K)}^2 + \sum_{i=1}^3 \frac{1}{h}\Norm{\PizFi\jump{(\vbfh)_{\tangscalar}}}^2_{\Ltwo(\Fi)}\Big).
\end{equation*}
Choosing~$\eta = 1$,
both the coefficients
$(1-1/(2\eta))$ and $\eta\gamma^2 + \Ctilde$ are positive;
hence,
summing over the elements~$\K$ in~$\Tauh$,
yields~\eqref{inf-sup:ii}.

The inf-sup estimate~\eqref{discrete-inf-sup:btwo}
follows from~\eqref{inf-sup:i-ii} and
the norm equivalence~\eqref{norm-equivalence-h-dev}.
This concludes the proof.
\end{proof}

\begin{remark} \label{remark:inf-sup}
For any discrete velocity~$\vbfh$ in~$\Vhk$,
the estimates~\eqref{inf-sup:i-ii} of Lemma~\ref{lemma:discrete-inf-sup-bt}
remain valid.
Indeed,
the same arguments as in the proof of Lemma~\ref{lemma:discrete-inf-sup-bt}
can be applied
using the decomposition~\eqref{eps-vh:dev-dev-representation-dev} of~$\dev(\epsbfun(\vbfh))$
on each~$\K$ in~$\Tauh$.
This decomposition no longer coincides with the one of~$\epsbfun(\vbfh)$;
however,
since the norm appearing in~\eqref{inf-sup:i-ii}
involves only the deviatoric part of the symmetric gradient,
the same estimates follow.
\eremk
\end{remark}

We are now in a position to prove the full inf-sup stability condition.
To this end,
we construct an appropriate pressure test function
and combine it with the stress test function obtained in the previous result.
\begin{theorem} \label{theorem:inf-sup-bone-btwo}
For any~$\vbfh$ in~$\Vhk$,
there holds the following discrete inf-sup stability
\begin{equation*}
\sup_{(\taubfh,\qh)\in\Sigmahkmo\times\Qhkmo} \frac{\bone(\vbfh,\qh) + \btwo(\taubfh,\vbfh)}{\Norm{\taubfh}_{\Sigmah} + \Norm{\qh}_{\Qh}}
    \gtrsim \Norm{\vbfh}_{\Vh}.
\end{equation*}
\end{theorem}
\begin{proof}
The proof is as that of~\cite[Thm~6.1]{Gopalakrishnan-Lederer-Schoberl:2020};
we report some details for the sake of completeness.
Consider the stress~$\taubfh$ of Lemma~\ref{lemma:discrete-inf-sup-bt}.
Since $\dive\Vhk = \Qhkmo$,
there exists a pressure~$\qh$ in~$\Qhkmo$ such that
$\qh = \dive\vbfh$;
with this choice,
the bilinear form~$\bone(\cdot,\cdot)$ of~\eqref{bone} satisfies
$\bone(\vbfh,\qh) = \Norm{\dive\vbfh}_{\Qh}^2 = \Norm{\qh}_{\LtwoOmega}^2$.
Hence,
using the estimates~\eqref{inf-sup:i-ii}
as discussed in Remark~\ref{remark:inf-sup},
and that
$\Norm{\epsbfun(\vbfh)}^2_{\Ltwo(\K)} \simeq \Norm{\dev(\epsbfun(\vbfh))}^2_{\Ltwo(\K)} + \Norm{\dive(\vbfh)}^2_{\Ltwo(\K)}$
on each~$\K$ in~$\Tauh$,
we obtain
\begin{equation*}
\frac{\bone(\vbfh,\qh)+\btwo(\taubfh,\vbfh)}{\Norm{\taubfh}_{\Sigmah} + \Norm{\qh}_{\Qh}}
\gtrsim \frac{\Norm{\dive\vbfh}_{\Qh}^2+\Norm{\vbfh}_{1,\dev,h}^2}{\Norm{\taubfh}_{\Sigmah} + \Norm{\qh}_{\Qh}}
\gtrsim \Norm{\vbfh}_{\Vh}.
\end{equation*}
Taking the supremum over~$\Sigmahkmo\times\Qhkmo$ on both sides yields the assertion.
\end{proof}

We further have a consistency property of the MCS formulation,
as stated in the next result.
\begin{lemma} \label{lemma:consistency}
Let $(\sigmabf,\ubf,p)$ be the solution to~\eqref{strong-mixed}
and assume $\sigmabf \in H^1(\Omega,\Rbbdd)$, $\ubf \in H^1(\Omega,\Rbbd)$, and $p \in H^1(\Omega,\Rbb)$.
Then,
for all $\taubfh \in \Sigmahkmo$, $\vbfh \in \Vhk$, and $\qh \in \Qhkmo$,
the following identity holds true
\begin{equation*}
a(\sigmabf,\taubfh) + \bone(\ubf,\qh) + \bone(\vbfh,p) + \btwo(\sigmabf,\vbfh) + \btwo(\taubfh,\ubf)
    = (-\fbf,\vbfh)_{\Omega}.
\end{equation*}
\end{lemma}
\begin{proof}
The proof follows along the same lines as those of~\cite[Thm~6.2]{Gopalakrishnan-Lederer-Schoberl:2020},
using the symmetric gradient~$\epsbfun(\cdot)$ rather than~$\nablabfun\cdot$.
\end{proof}

Consider now the bilinear form
$B: (\Sigmahkmo\times\Vhk\times\Qhkmo) \times (\Sigmahkmo\times\Vhk\times\Qhkmo) \to \Rbb$
defined by
\begin{equation} \label{bilinear-form:B}
B(\sigmabfh,\ubfh,\ph;\taubfh,\vbfh,\qh)
:= a(\sigmabfh,\taubfh) + \bone(\ubfh,\qh) + \bone(\vbfh,\ph) + \btwo(\sigmabfh,\vbfh) + \btwo(\taubfh,\ubfh),
\end{equation}
and the norm~$\Normthreebars{\cdot}{\cdot}{\cdot}$ on the product space $(\Sigmahkmo\times\Vhk\times\Qhkmo)$
defined by
\begin{equation} \label{norm-three-bars}
\Normthreebars{\sigmabfh}{\ubfh}{\ph}
:= \sqrt{\nu}\Norm{\ubfh}_{\Vh}
    + \frac{1}{\nu}\big(\Norm{\sigmabfh}_{\Sigmah} + \Norm{\ph}_{\Qh}\big).
\end{equation}

\subsection{Error estimates} \label{subsection:error-estimates}
We are now in a position to derive
optimal a priori error estimates
for the discrete formulation~\eqref{variational-formulation:discrete}.

Let $m \geq 1$ and assume that the solution $(\sigmabf,\ubf,p)$ to~\eqref{strong-mixed} satisfies
$\ubf$ in $H^1(\Omega,\Rbbd)\cap H^{m}(\Tauh,\Rbbd)$,
$p$ in $L^2_0(\Omega)\cap H^{m-1}(\Tauh)$,
and
$\sigmabf$ in $\Sigmabf\cap H^1(\Omega,\Rbbdd)\cap H^{m-1}(\Ccalh,\Rbbdd)$.
We first introduce the interpolation operators for the velocity and pressure spaces.
Let~$\IcalVhk$ denote the standard $H(\dive)$-conforming interpolation operator;
see, e.g., \cite[Eq.~(2.5.26)]{Boffi-Brezzi-Fortin:2013},
and~$\IcalQhkmo$ be the $\Ltwo$-projection into~$\Qhkmo$;
see, e.g., \cite[Eq.~(18.30)]{Ern-Guermond:1}.
Set $s := \min\{m-1,\,k\}$,
combining the estimate
$|\epsbfun(\vbf)|_{\K}| \leq |\nablabfun(\vbf)|_{\K}|$ for all~$\vbf$ in $H^1(\Tauh,\Rbbd)$ and~$\K$ in~$\Tauh$,
the multiplicative trace inequality,
and standard approximation results for the interpolation operator above; see, e.g.,
\cite[Prop.~2.5.4]{Boffi-Brezzi-Fortin:2013},
\cite[Thm.~18.16]{Ern-Guermond:1},
we obtain the following estimates
\begin{equation} \label{approximation:V-Q}
\Norm{\ubf-\IcalVhk\ubf}_{\Vh} \lesssim h^s\Norm{\ubf}_{H^{s+1}(\Tauh)}
\qquad\text{and}\qquad
\Norm{p - \IcalQhkmo p}_{\Qh} \lesssim h^s\Norm{p}_{H^s(\Tauh)}.
\end{equation}

We next introduce a suitable interpolation operator
for the stress space and establish its approximation properties.
Given an element~$\Ki$ in~$\Ccalh$,
let $\IcalKi: \Sigmabf(\Ki)\to\Sigmahkmo(\Ki)$
be the local interpolation operator
defined by imposing the degrees of freedom \eqref{dofs}
of~$\sigmabf$ in~$\Sigmabf(\Ki)$ for~$\IcalKi \sigmabf$, i.e.,
\begin{subequations} \label{local-interpolator-stress}
    \begin{align} 
        \int_{\Fi} (\IcalKi\sigmabf)_{\normtang} \piF \ds
        &\!= \int_{\Fi} (\sigmabf)_{\normtang} \piF \ds 
            &&\forall\piF \in \Pbbkmo(\Fi), \label{interpolation:edge} \\ 
        \int_{\Ki} (\IcalKi\sigmabf) : \dev(\normi\otimes\normi) \qi \dx
        &\!=\int_{\Ki} \sigmabf : \dev(\normi\otimes\normi) \qi \dx  
            &&\forall\qi \in \Pbbkmo(\Ki), \label{interpolation:bulk-kmo} \\
        \int_{\Ki} (\IcalKi\sigmabf) : \sym(\normi\otimes\tangi) \wi \lambdai \dx
        &=\int_{\Ki} \sigmabf : \sym(\normi\otimes\tangi) \wi \lambdai \dx
            &&\forall\wi \in \Pbbkmt(\Ki) \label{interpolation:bulk-kmt-lambda}.
\end{align}
\end{subequations}
Let
$\IcalSigmahkmo : \Sigmabf \to \Sigmahkmo$
be the corresponding global interpolation operator for the stress space, i.e.,
given~$\Ki$ in~$\Ccalh$ and~$\sigmabf$ in~$\Sigmabf$,
\begin{equation} \label{global-interpolation:def}
    (\IcalSigmahkmo\sigmabf)|_{\Ki} = \IcalKi(\sigmabf|_{\Ki}) \in \Sigmahkmo(\Ki).
\end{equation}
We begin by establishing stability and approximation properties
for the interpolation operator $\IcalSigmahkmo$.
\begin{lemma}
Let $\sigmabf$ be in $\Sigmabf\,\cap\,H^1(\Omega,\Rbbdd) \cap H^{m-1}(\Ccalh,\Rbbdd)$.
Then,
the local interpolation operator $\IcalKi$ defined by~\eqref{local-interpolator-stress}
satisfies the stability estimate
\begin{equation*}
\Norm{\IcalKi\sigmabf}_{\Ltwo(\Ki)}
\lesssim \Norm{\sigmabf}_{\Ltwo(\Ki)} + \hK\SemiNorm{\sigmabf}_{H^1(\Ki)}.
\end{equation*}
Moreover,
for $s:= \min(m-1,k)$,
for every element~$\Ki$ in~$\Ccalh$
and corresponding external edge~$\Fi$,
the following approximation estimates hold true:
\begin{equation} \label{local-approximation:sigma}
\Norm{\sigmabf-\IcalKi\sigmabf}_{\Ltwo(\Ki)}
\lesssim
h^s\Norm{\sigmabf}_{H^s(\Ki)}
\qquad\text{and}\qquad
\Norm{(\sigmabf - \IcalKi\sigmabf)_{\normtang}}_{\Ltwo(\Fi)}
\lesssim h^{s-\frac12}\Norm{\sigmabf}_{H^s(\Ki)}.
\end{equation}

Furthermore,
given the global interpolation operator~$\IcalSigmahkmo$ of~\eqref{global-interpolation:def},
it holds that
\begin{equation} \label{global-approximation:sigma}
\Norm{\sigmabf-\IcalSigmahkmo\sigmabf}_{\Sigmah}
\lesssim
h^s\Norm{\sigmabf}_{H^s(\Ccalh)}.
\end{equation}
\end{lemma}
\begin{proof}
We detail the proof for $k-1 \geq 1$.
The lowest-order case $k-1 = 0$ can be treated similarly
and is simpler.
We split the proof in three steps.
\paragraph*{Stability estimates.}
We first prove the local stability estimate.
Let~$\Ki$ be an element of~$\Ccalh$
and~$\sigmabf$ be a stress function in~$\Sigmabf$.
Lemma~\ref{lemma:dev-sym-basis-Ki} entails that
there exist polynomials~$\rIi$ and~$\sIi$ in~$\Pbbkmo(\Ki)$
such that
$\IcalKi(\sigmabf|_{\Ki})$ admits the following representation
\begin{equation} \label{IcalKi-representation}
\IcalKi(\sigmabf|_{\Ki}) = \rIi\dev(\normi\otimes\normi) + \sIi\sym(\normi\otimes\tangi).
\end{equation}
Hence,
the triangle inequality
and identities~\eqref{dev:frob-norm}-\eqref{sym:frob-norm}
imply that
\begin{equation} \label{triangle-IcalKi}
\begin{split}
\Norm{\IcalKi(\sigmabf|_{\Ki})}_{\Ltwo(\Ki)}
&\leq \Norm{\rIi\dev(\normi\otimes\normi)}_{\Ltwo(\Ki)}
    + \Norm{\sIi\sym(\normi\otimes\tangi)}_{\Ltwo(\Ki)} \\
&\leq\frac{1}{\sqrt{2}}
    (\Norm{\rIi}_{\Ltwo(\Ki)}
    +\Norm{\sIi}_{\Ltwo(\Ki)}).
\end{split}
\end{equation}
We estimate the two terms on the right-hand side separately.

As for the first term,
inserting the representation~\eqref{IcalKi-representation}
in~\eqref{interpolation:bulk-kmo}
and applying the orthogonality property~\eqref{orthogonality:dev-sym},
we obtain that,
for all~$\qi$ in $\Pbbkmo(\Ki)$,
the following identity holds true
\begin{equation*} \label{rIi-integral-value}
\int_{\Ki} \rIi \dev(\normi\otimes\normi) : \dev(\normi\otimes\normi) \qi \dx
= \int_{\Ki} \sigmabf : \dev(\normi\otimes\normi) \qi \dx.
\end{equation*}
The above display
and the Cauchy--Schwarz inequality
yield that
\begin{equation*} 
\begin{split}
&\frac12\Norm{\rIi}_{\Ltwo(\Ki)}^2
= \int_{\Ki}\rIi\dev(\normi\otimes\normi):\dev(\normi\otimes\normi)\rIi\dx \\
&\qquad= \int_{\Ki} \sigmabf:\dev(\normi\otimes\normi)\rIi \dx 
\lesssim \Norm{\sigmabf}_{\Ltwo(\Ki)} \Norm{\rIi}_{\Ltwo(\Ki)}.
\end{split}
\end{equation*}
Therefore,
\begin{equation} \label{rIi-estimate}
\Norm{\rIi}_{\Ltwo(\Ki)} \lesssim \Norm{\sigmabf}_{\Ltwo(\Ki)}.
\end{equation}

As for the second term,
since $(\IcalKi(\sigmabf|_{\Ki}) )_{\normtang}|_{\Fi}$
belongs to~$\Pbbkmo(\Fi)$,
the identity~\eqref{interpolation:edge},
the orthogonality property~\eqref{orthogonality:dev-sym},
the identity~\eqref{sym:frob-norm},
the symmetry of~$\sigmabf$,
and the fact that $\sym(\normi\otimes\tangi)$ is constant on~$\Fi$,
entail
\begin{equation} \label{sI-value-on-edge}
\frac{1}{2}\sIi|_{\Fi}
= (\IcalKi(\sigmabf|_{\Ki}))_{\normtang}|_{\Fi}
= \PikmoFi(\sigmabf_{\normtang})
=  \PikmoFi(\sigmabf):\sym(\normi\otimes\tangi).
\end{equation}
Since~$\sIi$ belongs to~$\Pbbkmo(\Ki)$,
the above display yields the existence of~$\stildeIi$ in~$\Pbbkmt(\Ki)$
such that
\begin{equation} \label{sI-stildeI}
\sIi = 2\Lcalkmo(\PikmoFi(\sigmabf_{\normtang})) + \stildeIi\lambdai,
\end{equation}
where
$\Lcalkmo(\PikmoFi(\sigmabf_{\normtang}))$ denotes
the constant extension of $\PikmoFi(\sigmabf_{\normtang})$ from~$\Fi$ over~$\Ki$
as defined in~\eqref{Lcal}.
Then,
the triangle inequality yields
\begin{equation} \label{triangle-sIi}
\begin{split}
\Norm{\sIi}_{\Ltwo(\Ki)}
\lesssim \Norm{\Lcalkmo(\PikmoFi\sigmabf_{\normtang})}_{\Ltwo(\Ki)}
    +\Norm{\stildeIi\lambdai}_{\Ltwo(\Ki)}
\end{split}
\end{equation}
Using the stability of the extension operator~$\Lcalkmo$ of~\eqref{stability-Lcal},
the stability of the $\Ltwo$-projection on~$\Fi$, cf., e.g.~\cite[Thm.~18.16]{Ern-Guermond:1}, 
the Cauchy--Schwarz inequality,
identity~\eqref{sym:frob-norm},
the multiplicative trace inequality; see, e.g., \cite[Lem.~12.15]{Ern-Guermond:1},
and the shape-regularity assumption,
give
\begin{equation} \label{Lcal-estimate}
\begin{split}
\Norm{\Lcalkmo(\PikmoFi\sigmabf)}_{\Ltwo(\Ki)}
&\lesssim h^{\frac12}\Norm{\PikmoFi\sigmabf_{\normtang}}_{\Ltwo(\Fi)} 
\lesssim h^{\frac12}\Norm{\sigmabf_{\normtang}}_{\Ltwo(\Fi)}
\lesssim h^{\frac12}\Norm{\sigmabf}_{\Ltwo(\Fi)} \\
&\lesssim \Norm{\sigmabf}_{\Ltwo(\Ki)} + \hK\SemiNorm{\sigmabf}_{H^1(\Ki)}.
\end{split}
\end{equation}
Inserting the representations~\eqref{IcalKi-representation} of~$\IcalKi\sigmabf$
and~\eqref{sI-stildeI} of~$\sIi$ on~$\Ki$
in identity~\eqref{interpolation:bulk-kmt-lambda},
we obtain that,
for all~$\wi$ in~$\Pbbkmt(\Ki)$
it holds
\begin{equation*}
\begin{split}
\int_{\Ki} \sigmabf &: \sym(\normi\otimes\tangi)\wi\lambdai\dx
=\int_{\Ki} \IcalKi(\sigmabf) : \sym(\normi\otimes\tangi)\wi\lambdai \dx \\
&= \int_{\Ki} (\rIi\dev(\normi\otimes\normi) + \sIi\sym(\normi\otimes\tangi)) : \sym(\normi\otimes\tangi)\wi\lambdai \dx \\
&= \int_{\Ki} \big(2\Lcalkmo(\PikmoFi(\sigmabf_{\normtang})) + \stildeIi\lambdai\big) \sym(\normi\otimes\tangi) : \sym(\normi\otimes\tangi)\wi\lambdai \dx \\
&= \frac12\int_{\Ki} \big(2\Lcalkmo(\PikmoFi(\sigmabf_{\normtang})) + \stildeIi\big)\lambdai\wi\lambdai \dx.
\end{split}
\end{equation*}
Choosing~$\wi = \stildeIi$ in~\eqref{interpolation:bulk-kmt-lambda} and
using the representation~\eqref{sI-stildeI},
the identity~\eqref{sym:frob-norm},
and the estimate~\eqref{Lcal-estimate},
give
\begin{equation*} 
\begin{split}
\Norm{\stildeIi\lambdai}^2_{\Ltwo(\Ki)}
&= 2\int_{\Ki}\sigmabf:\sym(\normi\otimes\tangi)\,\stildeIi\lambdai\dx
    - \int_{\Ki}\Lcalkmo(\PikmoFi(\sigmabf_{\normtang}))\,\stildeIi\lambdai\dx \\
&\lesssim \Big(\Norm{\sigmabf}_{\Ltwo(\Ki)}
    + \Norm{\Lcalkmo(\PikmoFi(\sigmabf_{\normtang}))}_{\Ltwo(\Ki)}\Big)
    \Norm{\stildeIi\lambdai}_{\Ltwo(\Ki)} \\
&\lesssim\Big(\Norm{\sigmabf}_{\Ltwo(\Ki)} + h\SemiNorm{\sigmabf}_{H^1(\Ki)}\Big)
     \Norm{\stildeIi\lambdai}_{\Ltwo(\Ki)},
\end{split}
\end{equation*}
hence,
the following estimate holds true
\begin{equation} \label{stildeIi-estimate}
\Norm{\stildeIi\lambdai}_{\Ltwo(\Ki)} \lesssim \Norm{\sigmabf}_{\Ltwo(\Ki)} + \hKi\SemiNorm{\sigmabf}_{H^1(\Ki)}.
\end{equation}
Inserting estimates~\eqref{Lcal-estimate} and~\eqref{stildeIi-estimate} in~\eqref{triangle-sIi}
gives
\begin{equation} \label{sIi-estimate}
\Norm{\sIi}_{\Ltwo(\Ki)} \lesssim \Norm{\sigmabf}_{\Ltwo(\Ki)} + \hKi\SemiNorm{\sigmabf}_{H^1(\Ki)}.
\end{equation}

Combining the estimates~\eqref{rIi-estimate} for~$\rIi$
and~\eqref{sIi-estimate} for~$\sIi$
in~\eqref{triangle-IcalKi}
with the quasi-uniform assumption,
we obtain the stability estimate
\begin{equation} \label{IcalKi-stability}
\Norm{\IcalKi(\sigmabf)}_{\Ltwo(\Ki)}
\lesssim \Norm{\sigmabf}_{\Ltwo(\Ki)} + \hK\SemiNorm{\sigmabf}_{H^1(\Ki)}.
\end{equation}

\paragraph*{Approximation estimates.}
We first derive the edge approximation estimate.
By~\eqref{sI-value-on-edge},
we obtain
\begin{equation*}
\Norm{(\sigmabf - \IcalKi\sigmabf)_{\normtang}}_{\Ltwo(\Fi)}
=\Norm{(\Id - \PikmoFi)(\sigmabf_{\normtang})}_{\Ltwo(\Fi)}.
\end{equation*}
Standard approximation properties of the $\Ltwo$-projection on edges; see, e.g., \cite[Rem.~18.17]{Ern-Guermond:1}
and regularity assumptions on~$\Tauh$,
yields
\begin{equation*}
\Norm{q - \PikmoFi q}_{\Ltwo(\Fi)} \lesssim h^{s-\frac12}\SemiNorm{q}_{H^s(\Ki)}
\qquad\qquad \forall q \in H^s(\Ki).
\end{equation*}
Combining the two displays above
with the regularity assumptions on~$\Tauh$ gives
\begin{equation*}
\Norm{(\sigmabf - \IcalKi\sigmabf)_{\normtang}}_{\Ltwo(\Fi)}
\lesssim h^{s-\frac12}\SemiNorm{\sigmabf}_{H^s(\Ki)},
\end{equation*}
which proves the local approximation estimate on~$\Fi$.

We next derive the volume approximation estimate.
By estimate~\eqref{IcalKi-stability},
the interpolation operator~$\IcalKi$ is continuous and
by the unisolvency of the degrees of freedom, cf. Lemma~\ref{lemma:unisolvence},
it preserves polynomials in~$\Sigmahkmo(\Ki)$, i.e.,
\begin{equation} \label{preservation-polynomials}
\sigmabfh - \IcalKi\sigmabfh = 0
    \qquad \forall \sigmabfh \in \Sigmahkmo(\Ki).
\end{equation}
Since the $\Ltwo$-projection preserves both the symmetric and trace-free properties,
given~$\sigmabf$ in~$\Sigmabf(\Ki)$
its $\Ltwo$-projection $\qbfunderkmo:=\Pikmo(\sigmabf)$ belongs to~$\Sigmahkmo(\Ki)$.
Then,
the triangle inequality,
the preservation property~\eqref{preservation-polynomials} and
the stability property~\eqref{IcalKi-stability} of the interpolation operator $\IcalKi$,
and the approximation properties of the $\Ltwo$-projection $\Pikmo$, cf., e.g., \cite[Thm.~18.16]{Ern-Guermond:1},
yield
\begin{equation*}
\begin{split}
\Norm{\sigmabf - \IcalKi\sigmabf}_{\Ltwo(\Ki)}
&\leq \Norm{\sigmabf - \qbfunderkmo}_{\Ltwo(\Ki)}
    + \Norm{\qbfunderkmo - \IcalKi\sigmabf}_{\Ltwo(\Ki)} \\
&= \Norm{\sigmabf - \qbfunderkmo}_{\Ltwo(\Ki)}
    + \Norm{\IcalKi(\qbfunderkmo - \sigmabf)}_{\Ltwo(\Ki)} \\
&\lesssim \Norm{\sigmabf - \qbfunderkmo}_{\Ltwo(\Ki)} + \hKi\SemiNorm{\sigmabf - \qbfunderkmo}_{H^1(\Ki)}
\lesssim h^s \SemiNorm{\sigmabf}_{H^s(\Ki)},
\end{split}
\end{equation*}
which proves the local approximation estimate on~$\Ki$.

The global estimate~\eqref{global-approximation:sigma}
follows by summing the local estimate
over all elements~$\Ki$ in~$\Ccalh$.
\end{proof}

We preliminarily show an inf-sup stability property
of the bilinear form~$B$ of~\eqref{bilinear-form:B}
with respect to the norm $\Normthreebars{\cdot}{\cdot}{\cdot}$ of~\eqref{norm-three-bars}.
\begin{lemma} \label{lemma:inf-sup-product-space}
The following inf-sup stability estimate holds true
\begin{equation*}
\begin{split}
\sup_{(\taubfh,\vbfh,\qh)\in\Sigmahkmo\times\Vhk\times\Qhkmo}
    &\frac{B(\IcalSigmahkmo\sigmabf - \sigmabfh,\IcalVhk\ubf - \ubfh,\IcalQhkmo p - \ph; \taubfh,\vbfh,\qh)}{\Normthreebars{\taubfh}{\vbfh}{\qh}} \\
&\qquad\qquad\qquad\geq \Normthreebars{\IcalSigmahkmo\sigmabf - \sigmabfh}{\IcalVhk\ubf - \ubfh}{\IcalQhkmo p - \ph}.
\end{split}
\end{equation*}
\end{lemma}
\begin{proof}
For the bilinear forms appearing in~\eqref{bilinear-form:B} we have the following properties:
the continuity of the bilinear forms~$a$, $\bone$, and~$\btwo$, cf. Lemma~\ref{lemma:continuity},
the coercivity on the kernel of the bilinear form~$a$, cf. Lemma~\ref{lemma_coercivity-kernel},
and the inf-sup condition for the bilinear forms~$\bone$ and~$\btwo$, cf. Theorem~\ref{theorem:inf-sup-bone-btwo}.
Hence,
the assertion follows by applying 
Brezzi's theorem; see, e.g., \cite{Boffi-Brezzi-Fortin:2013}.
\end{proof}

We now establish optimal convergence error estimates.
\begin{theorem} \label{theorem:error-estimates}
Let~$(\sigmabf,\ubf,\p)$ and $(\sigmabfh,\ubfh,\ph)$
be the solutions to~\eqref{strong-mixed} and~\eqref{variational-formulation:discrete}.
Assume that $\sigmabf$ belongs to $H^1(\Omega,\Rbbdd)\cap H^{m-1}(\Ccalh,\Rbbdd)$,
$\ubf$ to $H^1(\Omega,\Rbbd)\cap H^{m}(\Tauh,\Rbbd)$,
and $\p$ to $L^2_0(\Omega,\Rbb)\cap H^{m-1}(\Tauh,\Rbb)$.
Then,
for~$s := \min(m-1,k)$
the following estimate holds true
\begin{equation*}
 \frac{1}{\nu}\Norm{\sigmabf-\sigmabfh}_{\Sigmah} + \Norm{\ubf-\ubfh}_{\Vh} + \frac{1}{\nu}\Norm{\p-\ph}_{\Qh}
    \lesssim h^s\Big(\Norm{\sigmabf}_{H^{s}(\Ccalh)} + \Norm{\ubf}_{H^{s+1}(\Tauh)} +\Norm{\p}_{H^{s}(\Tauh)} \Big)
\end{equation*}
\end{theorem}
\begin{proof}
The proof follows along the same lines as those of~\cite[Thm.~6.3]{Gopalakrishnan-Lederer-Schoberl:2020};
we include it here for completeness.

The triangle inequality gives
\begin{equation} \label{error-estimate:triangle}
\begin{split}
&\frac{1}{\nu}\Norm{\sigmabf-\sigmabfh}_{\Sigmah} + \Norm{\ubf-\ubfh}_{\Vh} + \frac{1}{\nu}\Norm{\p-\ph}_{\Qh} \\
&\qquad\qquad\lesssim \frac{1}{\nu}\Norm{\sigmabf-\IcalSigmahkmo\sigmabf}_{\Sigmah} + \Norm{\ubf-\IcalVhk\ubf}_{\Vh} + \frac{1}{\nu}\Norm{\p-\IcalQhkmo p}_{\Qh} \\
&\qquad\qquad\qquad +\frac{1}{\nu}\Norm{\IcalSigmahkmo\sigmabf-\sigmabfh}_{\Sigmah} + \Norm{\IcalVhk\ubf-\ubfh}_{\Vh} + \frac{1}{\nu}\Norm{\IcalQhkmo p-\ph}_{\Qh} \\
&\qquad\qquad\lesssim \frac{1}{\nu}\Norm{\sigmabf-\IcalSigmahkmo\sigmabf}_{\Sigmah} + \Norm{\ubf-\IcalVhk\ubf}_{\Vh} + \frac{1}{\nu}\Norm{\p-\IcalQhkmo p}_{\Qh} \\
&\qquad\qquad\qquad +\frac{1}{\sqrt{\nu}}\Normthreebars{\IcalSigmahkmo\sigmabf-\sigmabfh}{\IcalVhk\ubf-\ubfh}{\IcalQhkmo p-\ph}.
\end{split}
\end{equation}
Using the approximation estimates~\eqref{approximation:V-Q} and~\eqref{global-approximation:sigma},
the first three terms on the right-hand side of the above display can be bounded as needed.
As for the latter term,
using Lemmata~\ref{lemma:inf-sup-product-space}
and~\ref{lemma:consistency},
we deduce
\begin{equation} \label{sup:error-estimates}
\begin{split}
&\Normthreebars{\IcalSigmahkmo\sigmabf - \sigmabfh}{\IcalVhk\ubf - \ubfh}{\IcalQhkmo p - \ph} \\
&\qquad\qquad \leq
    \sup_{(\taubfh,\vbfh,\qh)\in\Sigmahkmo\times\Vhk\times\Qhkmo}
    \frac{B(\IcalSigmahkmo\sigmabf - \sigmabf,\IcalVhk\ubf - \ubf,\IcalQhkmo p - p; \taubfh,\vbfh,\qh)}{\Normthreebars{\taubfh}{\vbfh}{\qh}}.
\end{split}
\end{equation}
We now estimate the numerator.
As for the first three terms on the right-hand side of~\eqref{bilinear-form:B},
the Cauchy--Schwarz inequality yields  
\begin{equation} \label{a-bone:estimate}
\begin{split}
&a(\IcalSigmahkmo\sigmabf - \sigmabf,\taubfh) + \bone(\IcalVhk\ubf - \ubf,\qh) + \bone(\vbfh,\IcalQhkmo p - p) \\
&\quad\lesssim \Big(\frac{1}{\sqrt{\nu}}\Norm{\IcalSigmahkmo\sigmabf - \sigmabf}_{\Sigmah}\frac{1}{\sqrt{\nu}}\Norm{\taubfh}_{\Sigmah} \Big)
    + \Big(\sqrt{\nu}\Norm{\IcalVhk\ubf - \ubf}_{\Vh}\frac{1}{\sqrt{\nu}}\Norm{\qh}_{\Qh}\Big) \\
    &\qquad+ \Big(\sqrt{\nu}\Norm{\vbfh}_{\Vh}\frac{1}{\sqrt{\nu}}\Norm{\IcalQhkmo p - p}_{\Qh}\Big) \\
&\quad\leq \Normthreebars{\IcalSigmahkmo\sigmabf - \sigmabf}{\IcalVhk\ubf -\ubf}{\IcalQhkmo p - p}\Normthreebars{\taubfh}{\vbfh}{\qh}.
\end{split}
\end{equation}
As for the latter two terms on the right-hand side of~\eqref{bilinear-form:B},
the Cauchy--Schwarz inequality
and
the discrete trace inequality~\eqref{trace-inequality},
give
\begin{equation*} \label{btwo:estimate}
\begin{split}
&\btwo(\IcalSigmahkmo\sigmabf - \sigmabf,\vbfh) + \btwo(\taubfh,\IcalVhk\ubf - \ubf) \\
&\lesssim \sum_{\F \in \Fcalh} \sqrt{h}\Norm{(\IcalSigmahkmo\sigmabf - \sigmabf)_{\normtang}}_{\Ltwo(\F)}\frac{1}{\sqrt{h}}\Norm{\jump{(\vbfh)_{\tang}}}_{\Ltwo(\F)} 
    +\sum_{\Ki \in \Ccalh} \Norm{\IcalSigmahkmo\sigmabf - \sigmabf}_{\Ltwo(\Ki)}\Norm{\epsbfun(\vbfh)}_{\Ltwo(\Ki)} \\
&\quad+ \sum_{\F \in \Fcalh} \sqrt{h}\Norm{(\taubfh)_{\normtang}}_{\Ltwo(\F)}\frac{1}{\sqrt{h}}\Norm{\jump{(\IcalVhk\ubf - \ubf)_{\tang}}}_{\Ltwo(\F)} 
    +\sum_{\Ki \in \Ccalh} \Norm{\taubfh}_{\Ltwo(\Ki)}\Norm{\epsbfun(\IcalVhk\ubf - \ubf)}_{\Ltwo(\Ki)} \\
&\lesssim \left(\sum_{\F \in \Fcalh} \frac{h}{\nu}\Norm{(\IcalSigmahkmo\sigmabf - \sigmabf)_{\normtang}}^2_{\Ltwo(\F)}\right)^{\frac12}
          \left(\sum_{\F \in \Fcalh} \frac{\nu}{h}\Norm{\jump{(\vbfh)_{\tang}}}^2_{\Ltwo(\F)}\right)^{\frac12} \\
    &\quad+ \left(\sum_{\Ki \in \Ccalh} \frac{1}{\nu}\Norm{\IcalSigmahkmo\sigmabf - \sigmabf}^2_{\Ltwo(\Ki)}\right)^{\frac12}
             \left(\sum_{\K \in \Tauh}\nu\Norm{\epsbfun(\vbfh)}^2_{\Ltwo(\K)}\right)^{\frac12} \\
    &\quad+ \left(\sum_{\Ki \in \Ccalh} \frac{1}{\nu}\Norm{\taubfh}^2_{\Ltwo(\Ki)}\right)^{\frac12}
             \left(\sum_{\F \in \Fcalh} \frac{\nu}{{h}}\Norm{\jump{(\IcalVhk\ubf - \ubf)_{\tang}}}^2_{\Ltwo(\F)}\right)^{\frac12} \\
    &\quad+ \left(\sum_{\Ki \in \Ccalh} \frac{1}{\nu}\Norm{\taubfh}^2_{\Ltwo(\Ki)}\right)^{\frac12}
             \left(\sum_{\K \in \Tauh}\Norm{\epsbfun(\IcalVhk\ubf - \ubf)}^2_{\Ltwo(\K)}\right)^{\frac12}.
\end{split}
\end{equation*}
Starting from the right-hand side of the above display,
the definition of the norm~$\Normthreebars{\cdot}{\cdot}{\cdot}$ of~\eqref{norm-three-bars}
and the estimate
$\Normthreebars{\taubfh}{\vbfh}{0} \leq \Normthreebars{\taubfh}{\vbfh}{\ph}$,
yield
\begin{equation*}
\begin{split}
\btwo(\IcalSigmahkmo\sigmabf - \sigmabf,\vbfh) &+ \btwo(\taubfh,\IcalVhk\ubf - \ubf) \\
&\lesssim \frac{1}{\sqrt{\nu}}\left(\left(\sum_{\F \in \Fcalh} {h}\Norm{(\IcalSigmahkmo\sigmabf - \sigmabf)_{\normtang}}^2_{\Ltwo(\F)}\right)^{\frac12}
                +\Norm{\IcalSigmahkmo\sigmabf - \sigmabf}_{\Sigmah} \right)
                \Normthreebars{\taubfh}{\vbfh}{0} \\
    &\quad+ \sqrt{\nu}\Norm{\IcalVhk\ubf - \ubf}_{\Vh} 
                \Normthreebars{\taubfh}{\vbfh}{0} \\
&\lesssim \frac{1}{\sqrt{\nu}}\left(\left(\sum_{\F \in \Fcalh} {h}\Norm{(\IcalSigmahkmo\sigmabf - \sigmabf)_{\normtang}}^2_{\Ltwo(\F)}\right)^{\frac12}
                +\Norm{\IcalSigmahkmo\sigmabf - \sigmabf}_{\Sigmah} \right)
                \Normthreebars{\taubfh}{\vbfh}{\ph} \\
    &\quad+ \sqrt{\nu}\Norm{\IcalVhk\ubf - \ubf}_{\Vh} 
                \Normthreebars{\taubfh}{\vbfh}{\ph}.
\end{split}
\end{equation*}
Combining the above display with estimates~\eqref{a-bone:estimate} in~\eqref{sup:error-estimates},
we obtain
\begin{equation*}
\begin{split}
&\Normthreebars{\IcalSigmahkmo\sigmabf - \sigmabfh}{\IcalVhk\ubf - \ubfh}{\IcalQhkmo p - \ph} \\
&\qquad\qquad\lesssim \Normthreebars{\IcalSigmahkmo\sigmabf - \sigmabf}{\IcalVhk\ubf - \ubf}{\IcalQhkmo p - p}
    + \frac{1}{\sqrt{\nu}}\left(\sum_{\Fi \in \Fcalh} {h}\Norm{(\IcalSigmahkmo\sigmabf - \sigmabf)_{\normtang}}^2_{\Ltwo(\Fi)}\right)^{\frac12}.
\end{split}
\end{equation*}
Inserting the estimate above in~\eqref{error-estimate:triangle},
and applying the approximation
estimates~\eqref{approximation:V-Q}, \eqref{local-approximation:sigma}, and~\eqref{global-approximation:sigma},
the assertion follows.
\end{proof}

Due to the choice of the spaces~$\Vhk$ and~$\Qhkmo$,
the proposed method~\eqref{variational-formulation:discrete}
produces exactly divergence-free velocity fields,
cf.~\eqref{exact-divergence-free}.
In particular,
this leads to pressure robustness,
which is the subject of the next result.
\begin{theorem}
Under the same assumptions as Theorem~\ref{theorem:error-estimates},
the following estimate holds true
\begin{equation*}
\frac{1}{\nu}\Norm{\sigmabf-\sigmabfh}_{\Sigmah}
    + \Norm{\ubf - \ubfh}_{\Vh}
\lesssim
h^s \Norm{\ubf}_{H^{s+1}(\Tauh)}.
\end{equation*}
\end{theorem}
\begin{proof}
The proof follows along the same lines as those of Theorem~\ref{theorem:error-estimates}
and of~\cite[Thm.~6.4]{Gopalakrishnan-Lederer-Schoberl:2020}.
Omitting the pressure term,
estimate~\eqref{error-estimate:triangle} reads as
\begin{equation*}
\frac{1}{\nu}\Norm{\sigmabf-\sigmabfh}_{\Sigmah} + \Norm{\ubf-\ubfh}_{\Vh}
\lesssim
\frac{h^s}{\nu}\Norm{\sigmabf}_{H^s(\Ccalh)}
    + h^s\Norm{\ubf}_{H^{s+1}(\Tauh)}
    + \frac{1}{\sqrt{\nu}}\Normthreebars{\IcalSigmahkmo\sigmabf-\sigmabfh}{\IcalVhk\ubf-\ubfh}{0}.
\end{equation*}
Using the identity $\sigmabf = \nu \epsbfun(\ubf)$,
the assertion follows proceeding as before.
\end{proof}

\section{Numerical results} \label{section:numerics}
In this section, we present numerical results
to verify the performance of the proposed methods
and validate the theoretical analysis.
All numerical experiments are performed using the finite element library
\texttt{NGSolve}/\texttt{Netgen}~\cite{Schoberl:1997,Schoberl:2014}.

Given the exact solution~$(\sigmabf,\ubf,\p)$ to the strong formulation~\eqref{strong-mixed}, 
let $(\sigmabfh,\ubfh,\ph)$ denote the discrete solution
to the formulation~\eqref{variational-formulation:discrete},
for a given polynomial degree~$k\geq 1$.
We consider the following error measurements:
\begin{itemize}
\item as for the fluxes,
the $\Ltwo$-norm of the stress error
and the $\Ltwo$-norm of the symmetric part of the stress error
\begin{equation*}
\Ekmosigma := \Norm{\sigmabf - \sigmabfh}_{\Ltwo(\Omega)},
\qquad\text{and}\qquad
\Ekmosigmasym := \Norm{\sym(\sigmabf) - \sym(\sigmabfh)}_{\Ltwo(\Omega)};
\end{equation*}
\item as for the velocity,
the $\Ltwo$-norm of the symmetric gradient error
and the $\Ltwo$-norm of the velocity error
\begin{equation*}
\Ekepsu := \Norm{\epsbfun(\ubf) - \epsbfun(\ubfh)}_{\Ltwo(\Omega)}
\qquad\text{and}\qquad
\Eku := \Norm{\ubf - \ubfh}_{\Ltwo(\Omega)};
\end{equation*}
\item as for the pressure,
the $\Ltwo$-norm of the pressure error
$
\Ekmop := \Norm{p - \ph}_{\Ltwo(\Omega)}.
$
\end{itemize}
In the following,
we validate the proposed methods on two numerical examples:
a Stokes problem with a smooth, divergence-free exact velocity
to assess the convergence rates of the errors and the pressure robustness of the method,
and a non-Newtonian channel flow problem governed by a power-law model
to assess the performance of the methods for non-Newtonian fluids.

\subsection{Curl-generated velocity field}  \label{subsection:stokes}
On the two-dimensional domain $\Omega = (0,1)^2$,
we consider the following prescribed exact solution to the Stokes problem~\eqref{strong-mixed}
\begin{equation*}
\ubf = \curlbf(x^2(x-1)^2y^2(y-1)^2), \qquad
\p = x^5 + y^5 - \frac{1}{3}, \qquad
\sigmabf = \nu \epsbfun(\ubf),
\end{equation*}
and compute the corresponding forcing term~$\fbf = -\divebf \sigmabf + \nablabf p$.

In Figure~\ref{fig:stokes},
we show the convergence plots of the errors introduced above
with respect to $h$-refinement,
for a fixed viscosity $\nu = 10^{-3}$
and polynomial degrees $k=1,2,3,4$.
As expected from the theoretical analysis,
see Theorem~\ref{theorem:error-estimates},
the error of the symmetric velocity gradient,
as well as the~$\Ltwo$-norm errors
of the stress and pressure,
converge with order~$k$.
The $\Ltwo$-norm error of the velocity instead converges with order~$k+1$,
as expected from the standard Aubin--Nitsche argument
under sufficiently smooth regularity requirements on~$\ubf$
and under full elliptic regularity assumptions,
which hold also for the dual problem
since the problem is symmetric.

A detailed summary of the velocity and pressure errors,
together with the corresponding experimental orders of convergence (eoc)
for polynomial degrees $k=1,2$,
is reported in Table~\ref{tab:errors-u-p}.
In Table~\ref{tab:errors-sigma},
we compare the error of the stress and of its symmetric part
for polynomial degrees $k=1,2$.
The two quantities coincide up to machine precision,
confirming that the discrete stress solution is exactly symmetric.

Notice that the velocity field~$\ubf$ is the $\curlbf$ of a scalar potential,
and is therefore divergence-free;
this exact solution thus allows us to numerically validate
the pressure robustness of the method.
In Table~\ref{tab:nu},
we report the $\Ltwo$-norm velocity error
with respect to the viscosity~$\nu$
on a fixed mesh and
for each polynomial degree~$k=1,2,3$.  
We observe that the velocity error~$\Eku$ is independent of the viscosity~$\nu$,
even across several orders of magnitude,
thus confirming the pressure robustness.

\begin{figure}[htb]
    \centering
    \includegraphics[width=0.99\linewidth]{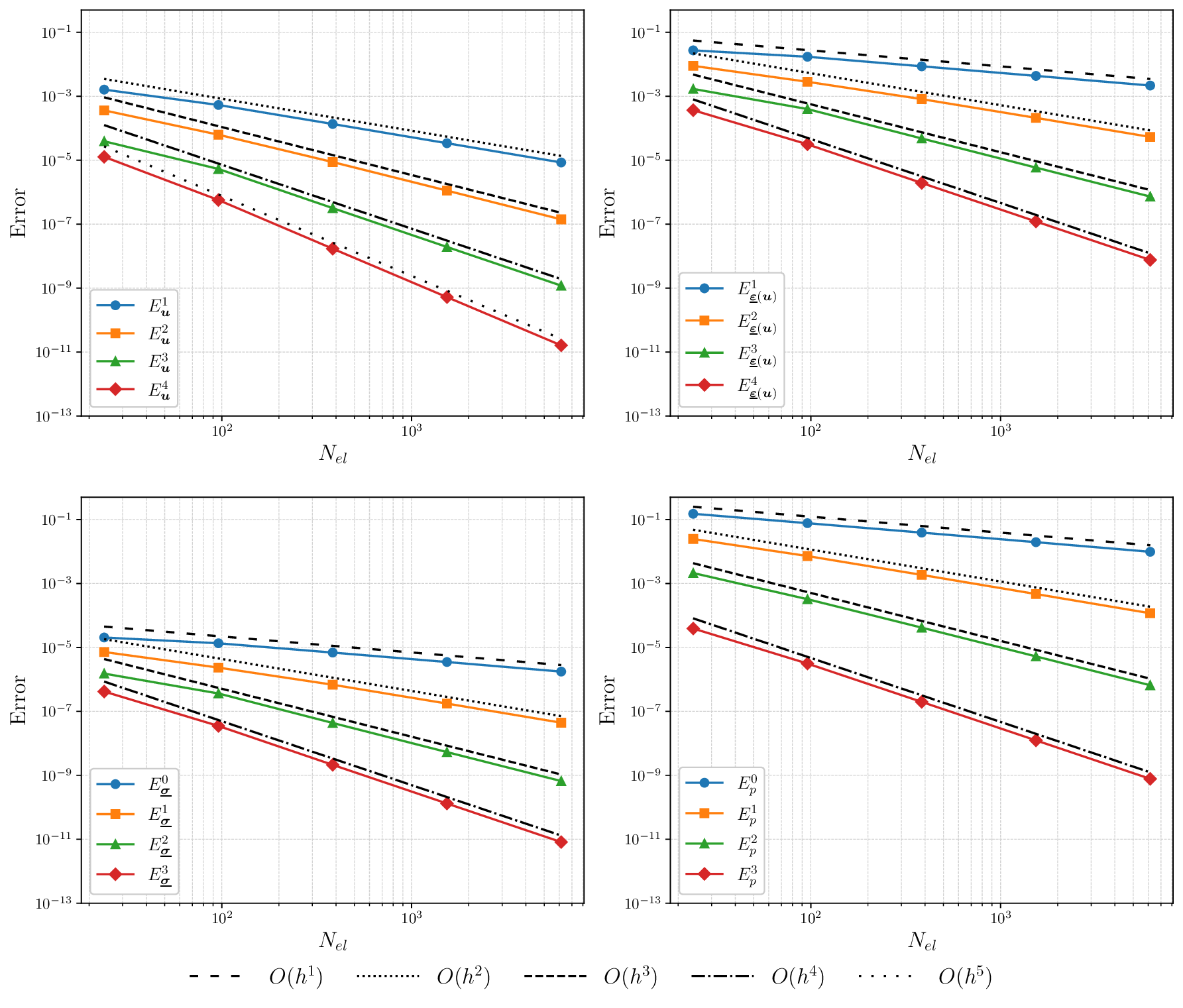}
    \caption{ \footnotesize Problem of Section~\ref{subsection:stokes}:
    convergence plot with respect to $h$-refinement for
    the discrete solutions to the exact symmetric formulation~\eqref{variational-formulation:discrete}
    and polynomial degrees $k=1,2,3,4$.}
     \label{fig:stokes}
\end{figure}

\begin{table}[htbp]
\centering
\begin{tabular}{c|cc cc cc}
\toprule
$\Nel$
& $\Eku$ & (eoc)
& $\Ekepsu$ & (eoc)
& $\Ekmop$ & (eoc) \\
\midrule
\multicolumn{7}{c}{$k = 1$} \\
24    & $1.6\cdot 10^{-3}$ & (--) & $2.7\cdot 10^{-2}$ & (--) & $1.5\cdot 10^{-1}$ & (--) \\
96    & $5.3\cdot 10^{-4}$ & (1.6) & $1.7\cdot 10^{-2}$ & (0.7) & $7.7\cdot 10^{-2}$ & (1.0) \\
384   & $1.4\cdot 10^{-4}$ & (2.0) & $8.7\cdot 10^{-3}$ & (1.0) & $3.9\cdot 10^{-2}$ & (1.0) \\
1536  & $3.4\cdot 10^{-5}$ & (2.0) & $4.3\cdot 10^{-3}$ & (1.0) & $2.0\cdot 10^{-2}$ & (1.0) \\
6144  & $8.5\cdot 10^{-6}$ & (2.0) & $2.2\cdot 10^{-3}$ & (1.0) & $9.8\cdot 10^{-3}$ & (1.0) \\
\midrule
\multicolumn{7}{c}{$k = 2$} \\
24    & $3.6\cdot 10^{-4}$ & (--) & $9.0\cdot 10^{-3}$ & (--) & $2.5\cdot 10^{-2}$ & (--) \\
96    & $6.3\cdot 10^{-5}$ & (2.5) & $2.8\cdot 10^{-3}$ & (1.7) & $7.2\cdot 10^{-3}$ & (1.8) \\
384   & $8.8\cdot 10^{-6}$ & (2.8) & $8.2\cdot 10^{-4}$ & (1.8) & $1.9\cdot 10^{-3}$ & (2.0) \\
1536  & $1.1\cdot 10^{-6}$ & (3.0) & $2.1\cdot 10^{-4}$ & (1.9) & $4.7\cdot 10^{-4}$ & (2.0) \\
6144  & $1.4\cdot 10^{-7}$ & (3.0) & $5.4\cdot 10^{-5}$ & (2.0) & $1.2\cdot 10^{-4}$ & (2.0) \\
\bottomrule
\end{tabular}
\caption{\footnotesize Problem of Section~\ref{subsection:stokes}:
        velocity and pressure errors under $h$-refinement and
        corresponding experimental order of convergence (eoc),
        for fixed polynomial degree $k=1,2$.}
\label{tab:errors-u-p}
\end{table}

\begin{table}[htbp]
\centering
\begin{tabular}{c|ccc|ccc}
\toprule
& \multicolumn{3}{c|}{$k=1$}
& \multicolumn{3}{c}{$k=2$} \\
$\Nel$
& $\Ekmosigma$
& $\Ekmosigmasym$
& (eoc)
& $\Ekmosigma$
& $\Ekmosigmasym$
& (eoc) \\
\midrule
24
& $2.0\cdot10^{-5}$ & $2.0\cdot10^{-5}$ & (--)
& $7.2\cdot10^{-6}$ & $7.2\cdot10^{-6}$ & (--) \\

96
& $1.3\cdot10^{-5}$ & $1.3\cdot10^{-5}$ & (0.6)
& $2.3\cdot10^{-6}$ & $2.3\cdot10^{-6}$ & (1.6) \\

384
& $6.9\cdot10^{-6}$ & $6.9\cdot10^{-6}$ & (1.0)
& $6.8\cdot10^{-7}$ & $6.8\cdot10^{-7}$ & (1.8) \\

1536
& $3.5\cdot10^{-6}$ & $3.5\cdot10^{-6}$ & (1.0)
& $1.8\cdot10^{-7}$ & $1.8\cdot10^{-7}$ & (2.0) \\

6144
& $1.8\cdot10^{-6}$ & $1.8\cdot10^{-6}$ & (1.0)
& $4.4\cdot10^{-8}$ & $4.4\cdot10^{-8}$ & (2.0) \\
\bottomrule
\end{tabular}
\caption{\footnotesize Problem of Section~\ref{subsection:stokes}:
        stress errors under $h$-refinement and
        corresponding experimental order of convergence (eoc),
        for fixed polynomial degree $k=1,2$.}
\label{tab:errors-sigma}
\end{table}

\begin{table}[htbp]
\centering
\begin{tabular}{c c c c c c c c}
\toprule
& & \multicolumn{6}{c}{$\nu$} \\
& & $10^{-7}$ & $10^{-5}$ & $10^{-3}$ & $10^{-1}$ & $10^1$ & $10^3$  \\
\midrule 
\multirow{3}{*}{$\Eku$}
& \multicolumn{1}{c|}{$k = 1$}
& $2.7\cdot 10^{-2}$ & $2.7\cdot 10^{-2}$ & $2.7\cdot 10^{-2}$ & $2.7\cdot 10^{-2}$ & $2.7\cdot 10^{-2}$ & $2.7\cdot 10^{-2}$\\
& \multicolumn{1}{c|}{$k = 2$}
& $8.9\cdot 10^{-3}$ & $8.9\cdot 10^{-3}$ & $8.9\cdot 10^{-3}$ & $8.9\cdot 10^{-3}$ & $8.9\cdot 10^{-3}$ & $8.9\cdot 10^{-3}$\\
& \multicolumn{1}{c|}{$k = 3$}
& $1.7\cdot 10^{-3}$ & $1.7\cdot 10^{-3}$ & $1.7\cdot 10^{-3}$ & $1.7\cdot 10^{-3}$ & $1.7\cdot 10^{-3}$ & $1.7\cdot 10^{-3}$\\
\bottomrule
\end{tabular}
\caption{\footnotesize Problem of Section~\ref{subsection:stokes}:
                        velocity error $\Eku$ with respect to the viscosity~$\nu$,
                        on a fixed mesh and for fixed polynomial degree~$k=1,2,3$.}
\label{tab:nu}
\end{table}
 
\subsection{Non-Newtonian channel flow with power-law fluid model} \label{subsection:channel-flow}

In this second experiment,
we consider a two-dimensional channel flow problem
as an illustrative benchmark for non-Newtonian fluids.
The direct access to the symmetric gradient of the velocity field~$\epsbfun(\ubf)$
is particularly relevant in this context,
since the constitutive relations of non-Newtonian fluids
are typically expressed in terms of~$\epsbfun(\ubf)$.
A well-known example is the power-law fluid model considered in this section.

A rigorous analysis of the resulting nonlinear discretization
goes beyond the scope of this work;
here,
we only aim to show the applicability
of the proposed method beyond the Newtonian setting
analyzed in the previous sections.

We consider the domain $\Omega = (0,2)\times(-1/2,1/2)$,
with no-slip boundary conditions on the walls, i.e. $(0,2) \times \{ -1/2, 1/2 \}$,
and periodic boundary conditions at the inlet $\{ 0 \} \times (-1/2,1/2)$ and outlet $\{ 2 \} \times (-1/2,1/2)$.
For the non-Newtonian fluid model we adopt a power-law model,
also known in the literature as the Ostwald-de Waele model;
see, e.g.~\cite{Glowinski-Wachs:2011},
where,
given a consistency value~$K>0$ and a power-law index~$r>1$,
the stress tensor~$\sigmabf$ and the velocity field~$\ubf$ 
satisfy the following constitutive relation
\begin{equation} \label{power-law}
\epsbfun(\ubf) - (2K)^{\frac{1}{1-r}}(\sqrt{2}|\sigmabf|)^{\frac{r-2}{1-r}}\sigmabf = 0.
\end{equation}
Observe that the case $r=2$ corresponds to the classical Newtonian fluid model.
We fix the forcing term $\fbf = (2,0)$
and the consistency value~$K=1$,
and we consider the following field as the exact velocity solution 
\begin{equation*}
\ubf =
\left(
\begin{array}{c}
\frac{r-1}{r}\Big(\frac{2}{K}\Big)^{\frac{1}{r-1}}\Big(\frac12\Big)^{\frac{r}{r-1}}\Big(1-2^{\frac{r}{r-1}}|y|^{\frac{r}{r-1}}\Big)  \\
0
\end{array}
\right)
;
\end{equation*}
the exact stress tensor~$\sigmabf$ is computed accordingly via~\eqref{power-law}.

We introduce a suitable modification
of the exactly symmetric MCS discrete formulation~\eqref{variational-formulation:discrete},
where the term $a(\sigmabfh,\taubfh)$ in the first equation is replaced by
\begin{equation*} 
\widetilde{a}(\sigmabfh,\taubfh)
:= \sum_{\K \in \Tauh} \sum_{i=1}^3 \int_{\Ki}
    (2K)^{\frac{1}{1-r}}(\sqrt{2}|\sigmabfh|)^{\frac{r-2}{1-r}} \sigmabfh : \taubfh\dx.
\end{equation*}
To address the nonlinearity of the power-law model,
we employ a fixed-point iteration scheme.
We consider power-law indices $r\in \{1.5, 2, 2.5\}$
and compare the results obtained with the aforementioned formulation.
Figure~\ref{fig:non-newtonian-channel-flow}
reports the corresponding convergence plots under~$h$-refinement.
In all considered cases,
we observe optimal convergence rates:
order~$k$ for the errors of symmetric gradient of the velocity, the stress, and the pressure,
and order~$k+1$ for the $\Ltwo$-norm error of the velocity.
Moreover,
the discrete stress solution is exactly symmetric also in this test case,
as shown in Table~\ref{tab:symmetry-error-channel}
for~$k=1$ and~$r=1.5$.
Similar observations hold for other values of $k$ and $r$.

\begin{figure}[htb]
    \centering
    \includegraphics[width=0.99\linewidth]{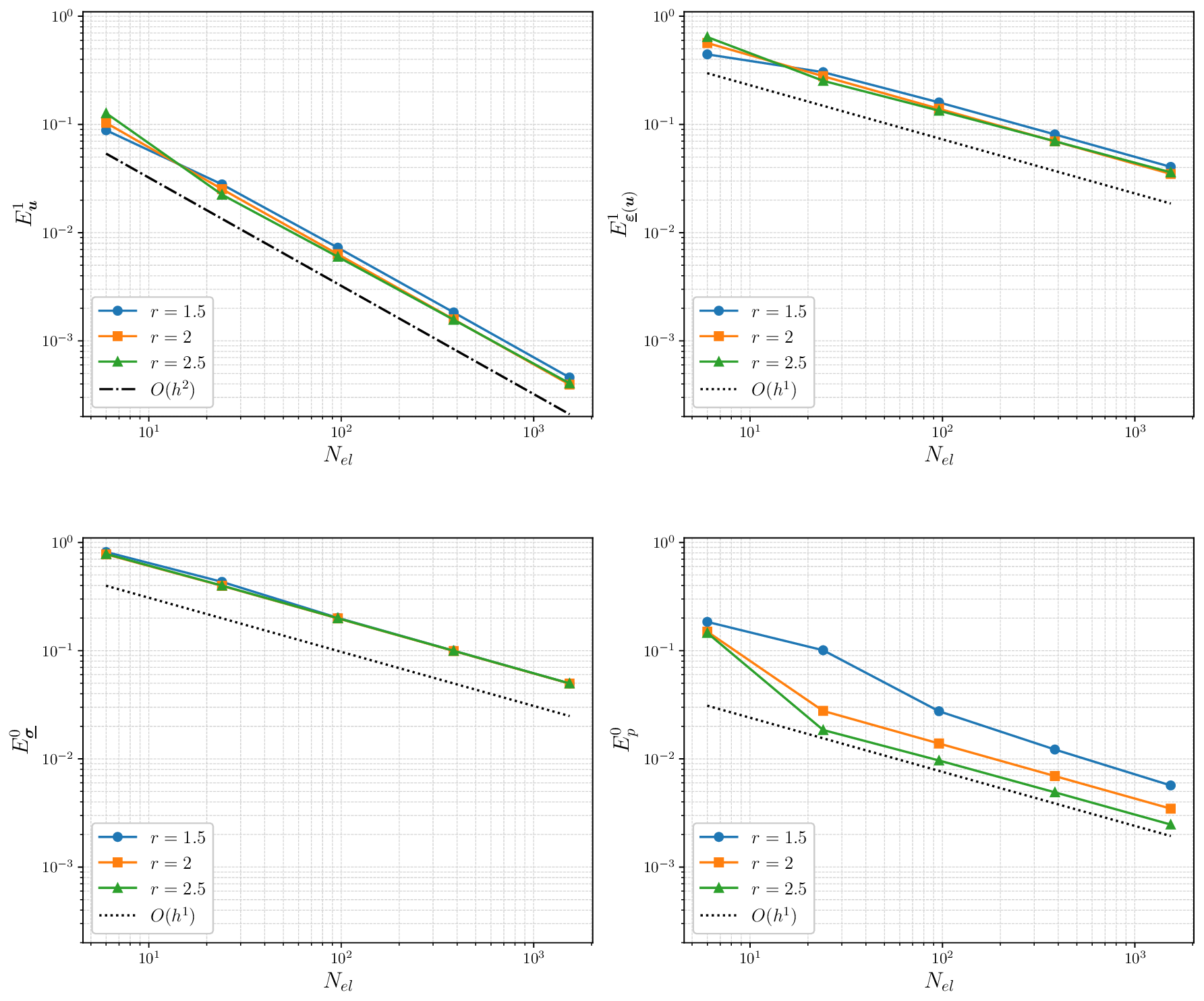}
    \caption{\footnotesize Problem of Section~\ref{subsection:channel-flow}:
    convergence plot under $h$-refinement for
    the discrete solutions under the power-law fluid model with $r=1.5,2,2.5$,
    for polynomial degree $k=1$.}
     \label{fig:non-newtonian-channel-flow}
\end{figure}

\begin{table}[htbp]
\centering
\begin{tabular}{c|ccc}
\toprule
& \multicolumn{3}{c}{$k=1$}\\
$\Nel$
& $\Ekmosigma$
& $\Ekmosigmasym$
& (eoc) \\
\midrule
24
& $8.1\cdot10^{-1}$ & $8.1\cdot10^{-1}$ & (--)\\

96
& $4.3\cdot10^{-1}$ & $4.3\cdot10^{-1}$ & (1.0) \\

384
& $2.0\cdot10^{-1}$ & $2.0\cdot10^{-1}$ & (1.0) \\

1536
& $9.9\cdot10^{-2}$ & $9.9\cdot10^{-2}$ & (0.9) \\

6144
& $4.9\cdot10^{-2}$ & $4.9\cdot10^{-2}$ & (1.0) \\
\bottomrule
\end{tabular}
\caption{\footnotesize Problem of Section~\ref{subsection:channel-flow}:
        stress errors under $h$-refinement and
        corresponding experimental order of convergence (eoc),
        for fixed polynomial degree $k=1$
        and power-law index $r=1.5$.}
\label{tab:symmetry-error-channel}
\end{table}

\section*{Acknowledgments} 
MM has been partially funded by the
Ministero dell'Università e della Ricerca
(MUR)
under PRIN2022 - 202292JW3F.
MM is a member of the
Gruppo Nazionale Calcolo Scientifico-Istituto Nazionale di Alta Matematica
(GNCS-INdAM)
and acknowledges funding
under GNCS-INdAM Project - CUP{\textunderscore}E53C25002010001.
MM gratefully acknowledges the hospitality
of the University of Hamburg.
{\footnotesize
\bibliography{bib.bib}}
\bibliographystyle{plain}

\end{document}